\documentclass[preprint,11pt]{imsart}

\pdfoutput=1

\usepackage{soul}
\usepackage{amsfonts,amsmath,amssymb,amsthm}
\usepackage{dsfont,verbatim}
\usepackage{url}
\usepackage[ruled,vlined,linesnumbered]{algorithm2e}
\usepackage{pgfplots}
\usepackage{subcaption}
\usepackage{graphicx}
\usepackage{mathtools}
\usepackage[textsize=tiny]{todonotes}
\usepackage{bbm}
\usepackage{multicol}
\usepackage{comment}
\usepackage{microtype}
\usepackage{booktabs}
\usepackage{multirow}
\usepackage{rotating}
\usepackage{xcolor}
\RequirePackage[colorlinks=true, citecolor=blue, filecolor=black, urlcolor=blue]{hyperref}
\usepackage[round]{natbib}
\usepackage{overpic}

\newtheorem{theorem}{Theorem}
\newtheorem{corollary}[theorem]{Corollary}
\newtheorem{lemma}[theorem]{Lemma}
\newtheorem{prop}[theorem]{Proposition}

\theoremstyle{definition}

\newtheorem{assumption}[theorem]{Assumption}
\newtheorem{remark}[theorem]{Remark}

\newcommand{\ee}{\varepsilon}

\newtheorem{rem}{Remark}

\newcommand{\EE}{\mathbb{E}}

\newcommand{\PP}{\mathbb{P}}
\newcommand{\RR}{\mathbb{R}}
\newcommand{\R}{\mathbb{R}}

\newcommand{\dint}{\mathrm{d}}

\newcommand{\TRIM}{\texttt{TrIM }}
\renewcommand{\hat}[1]{\widehat{#1}}

\newcommand{\edit}[1]{{\textcolor{black}{{#1}}}}

\hypersetup{
colorlinks=true,
filecolor=black,
citecolor=blue,
urlcolor=black,
}

\usepackage[centering,letterpaper,margin=1in]{geometry}

\ifpdf
  \DeclareGraphicsExtensions{.eps,.pdf,.png,.jpg}
\else
  \DeclareGraphicsExtensions{.eps}
\fi

\begin{document}

\begin{frontmatter}

\title{TrIM: Transformed Iterative Mondrian Forests\\ for Gradient-based Dimension Reduction and High-Dimensional Regression}

\runtitle{Transformed Iterative Mondrian Forests}

\begin{aug}
\author[a1,a1b]{\fnms{Ricardo} \snm{Baptista}%
\ead[label=e1]{r.baptista@utoronto.ca}},
\author[a2]{\fnms{Eliza} \snm{O'Reilly}%
\ead[label=e2]{eoreill2@jh.edu}},
\and
\author[a3]{\fnms{Yangxinyu} \snm{Xie}%
\ead[label=e3]{xinyux@wharton.upenn.edu}}

\runauthor{Baptista, O'Reilly, Xie}

\address[a1]{Computing and Mathematical Sciences Department\\
California Institute of Technology\\
}
\address[a1b]{Department of Statistical Sciences\\
University of Toronto\\
\printead*{e1}
}
\address[a2]{Department of Applied Mathematics and Statistics\\
Johns Hopkins University\\
\printead*{e2}
}
\address[a3]{Department of Statistics and Data Science\\
University of Pennsylvania\\
\printead*{e3}
}
\end{aug}  

\begin{abstract}
We propose a computationally efficient 
algorithm for gradient-based linear dimension reduction and high-dimensional regression. 
The algorithm initially computes a Mondrian forest and uses this estimator to identify a relevant feature subspace of the inputs from an estimate of the expected gradient outer product (EGOP) of the regression function.
In addition, we introduce an iterative approach known as Transformed Iterative Mondrian (TrIM) forest to improve the Mondrian forest estimator by using the EGOP estimate to update the set of features and weights used by the Mondrian partitioning mechanism. 
We obtain consistency guarantees and convergence rates for estimating the EGOP matrix 
and the random forest estimator obtained from one iteration of the TrIM algorithm. Lastly, we demonstrate the effectiveness of our proposed algorithm for learning the relevant feature subspace across various settings with both simulated and real data.

\end{abstract}


\end{frontmatter}

\def\spacingset#1{\renewcommand{\baselinestretch}
{#1}\small\normalsize}

\section{Introduction}

Building accurate models and extracting information from data with a very large number of covariates is a difficult challenge in many scientific and engineering applications.
Successful machine learning approaches crucially rely on the presence of low-dimensional structure in real-world datasets. Supervised dimension reduction techniques play an important role in addressing the challenges of high-dimensional learning problems. These techniques offer two key benefits: (1) they enhance the explainability of models by identifying a low-dimensional set of features that are most relevant for predicting the response variable; (2) they improve the accuracy and computational efficiency of models in downstream regression and classification tasks by initially mapping the covariates to this low-dimensional relevant feature space. The improved accuracy is particularly important for prediction in high-dimensional settings with limited data.

Many approaches for dimension reduction select a subset of relevant features for prediction from the given set of covariates. These methods, often referred to as feature or variable selection, are generally applied after an initial model has been generated. The selection process involves computing ad-hoc feature importance scores that quantify the influence of each variable on the model output. These scores are then used to determine which variables are most crucial for the model's performance.
\emph{Random forests} \citep{breiman2001random} are a widely used class of prediction algorithms that have demonstrated state-of-the-art empirical performance \citep{fernandez2014we} and are equipped with several feature importance scores including the Mean Decrease in Impurity (MDI) and Mean Decrease in Accuracy (MDA)~\citep{breiman2001random}. %
These scores have increased the popularity of the algorithms in applications where insight into the prediction process is highly valued, but they have well-known biases in favoring features with higher entropy or lower dependence on other features \citep{Strobletal2007Bias,Strobletal2008Bias}. %
There has been significant effort to study and address these drawbacks (see recent work by \citet{Hooker2021,agarwal2023mdiflexiblerandomforestbased,ScornetMDA2022} and the references therein) and several alternative approaches for variable selection using random forests have also been proposed \citep{TreeSHAP2020,ColemanMentch2022,zhu2015reinforcement,deng2022towards}. %

Currently, all of the tree-based feature importance measures are limited to evaluating the relevance of a single covariate at a time for predicting the output. The corresponding theoretical results showing that such dimension reduction techniques allow the model to overcome the curse of dimensionality rely on the assumption that the regression model is \emph{sparse}, that is, the output only varies along a small subset of the input dimensions. %
\edit{Instead, a more flexible \emph{multi-index model} assumes that the output may depend on \emph{all} input coordinates, but only varies along the span of a general low-dimensional subspace of the inputs called the \emph{relevant feature subspace}.
In this paper, we study a new approach for dimension reduction and regression that is designed to adapt to the low-dimensional structure captured by this model.} 
The algorithm combines Mondrian random forests with a linear dimension reduction mechanism based on estimating the expected gradient outer product of the function, which captures the relevant feature subspace. 

Mondrian random forests \citep{lakshminarayanan2014mondrian,balog2016mondrian} are a computationally efficient random forest variant that is generated from a random hierarchical partitioning process called the \emph{Mondrian process} \citep{roy2008mondrian}. They were introduced to efficiently handle the online setting where data arrives over time and the estimator is updated accordingly. The partitioning mechanism is also particularly amenable to theoretical analysis. In fact, Mondrian random forests and their generalizations have so far been the only random forest estimators shown to obtain minimax optimal rates of convergence for regression in arbitrary dimension~\citep{mourtada2020minimax, OReillyTran2021minimax, KlusowskiMondrian2023}. The drawback of this algorithm is that the partitioning mechanism is not data-adaptive, and thus the performance suffers in high-dimensional feature spaces as compared to more data-driven random forest algorithms. %

In light of these challenges, we propose a new algorithm called \emph{Transformed Iterative Mondrian} (\texttt{TrIM}) forests that incorporates data adaptivity into the Mondrian process through an iterative two-step procedure. First, a set of features made up of linear combinations of the covariates are selected in a data-driven way. Second, a random forest estimator is built using a Mondrian process to hierarchically split the dataset along these selected features. This approach mitigates increased computational costs as the second step retains the efficiency of the Mondrian partitioning method, as well as allowing one to import much of the existing asymptotic theory for Mondrian forests and their generalizations. %

More specifically, %
the first step of our algorithm is to estimate the expected gradient outer product (EGOP) of the regression function. Under the assumption that the regression model is a ridge function, this matrix is a projection onto the relevant feature subspace. %
We will use an efficient difference quotient approximation as proposed in \cite{Trivedietal2014} to estimate this matrix. This approximation requires an initial estimate of the true regression function, which is a standard Mondrian forest estimator in our algorithm. The advantage of using Mondrian forests for this initial estimate is twofold. They are more computationally efficient than kernel estimates when the data set is large \citep{lakshminarayanan2016mondrian}, and they attain minimax optimal risk bounds. %
Our first theoretical result in Theorem \ref{t:Hn_approx_rate} shows consistency and a convergence rate of our EGOP estimator. The rate depends on properly scaling the parameters of the Mondrian forest, as well as a step size parameter from the difference quotient approximation, with the amount of data. 

The second step of our algorithm applies the estimated EGOP to the training inputs and computes a Mondrian forest using this transformed dataset. This step implicitly generates a hierarchical partition of the input space via a \emph{stable under iteration (STIT) process} \citep{Nagel2005} that divides the data along a set of features that are linear combinations of the original covariates. %
The resulting estimator is a particular kind of random tessellation forest \citep{TehRTFs2019, OReillyTran2021} that we will refer to as an \emph{oblique Mondrian forest}.
\citet{OReilly_ObliqueMondrian} recently obtained convergence rates of such estimators illustrating their robustness to error in estimating the projection onto the relevant feature subspace when the underlying regression function is a multi-index model. Combining these results with the error estimate of the EGOP gives our second main theoretical result in Theorem \ref{thm:conv_rate} that shows consistency and a convergence rate for the estimator obtained from one iteration of the \TRIM algorithm.%

In this work, we also present simulation and real data experiments that demonstrate the effectiveness of our proposed algorithm. Our results show that \TRIM is able to learn the relevant feature subspace and improve the accuracy of the Mondrian forest regression model across a variety of settings. The code to reproduce the numerical results can be found at~\href{https://github.com/Xieyangxinyu/TrIM}{\url{github.com/Xieyangxinyu/TrIM}}. %

\subsection{Related Work}

There has been a lot of attention in the statistical literature on learning the relevant feature subspace for a multi-index model. These methods are collectively known as sufficient dimension reduction, and we refer to \citet{CookLecture2007} and \citet{LiBook2017} and the references therein for a summary of this literature. 
We highlight a subset of these linear dimension reduction methods called inverse regression, which broadly aims to estimate the relevant feature subspace by approximating functionals of the covariates conditioned on the response, i.e. $X | Y$. These techniques include Sliced Inverse Regression (SIR) \citep{SIR_Li1991, SIR_Lin2018} %
and Sliced Average Variance
Estimation (SAVE) \citep{dennis2000save}. %
Related to these methods, \citet{ma2012semiparametric} showed how to construct semi-parametric estimators and obtain consistent estimators directly for the subspace. 
\cite{loyal2022dimension} used SIR and SAVE to select linear combinations of covariates as inputs for \edit{Breiman's} random forest algorithm and empirically compares the performances. In Section \ref{sec:simulation} we perform numerical experiments comparing our approach to those from \cite{loyal2022dimension}. A drawback of these inverse regression techniques is that the number of relevant features is not known \emph{a priori}, and so one may miss any relevant features by selecting a set of directions that does not have the ability the span the relevant feature subspace. Another set of approaches use non-parametric methods to approximate gradients of the function and identify the relevant subspace~\citep{Xia2002MAVE, Trivedietal2014}, which trace back to~\cite{li1992principal}. We note that most approaches do not account for the subsequent regression step as we do here. A related work that accounts for errors in both the relevant subspace and the regression step is in~\citet{liu2024learning}. The %
regression relies on piece-wise polynomial approximations, which result in a convergence rate that only depends on the dimension of the central subspace, but \edit{exact knowledge of a variance quantity is assumed and it} can be challenging to use in practice.

A parallel line of research in the computational modeling literature focuses on high-dimensional function approximation %
\edit{under a multi-index model assumption. In this literature, the linear dimension reduction model is referred to as a \emph{ridge function} and} the relevant feature subspace has been coined the \emph{active subspace}~\citep{constantine2014active}. These methods identify the active subspace using the EGOP matrix as in our proposed algorithm. In particular, they extract the subspace components from the dominant eigenvectors of the EGOP matrix. A number of theoretical guarantees have been obtained in this literature for recovery of the active subspace, but the results assume access to gradient evaluations of the function. %
In particular,~\citet{zahm2020gradient} derived a bound for the approximation error of vector-valued functions based on the eigenvalues of the EGOP matrix. We do not assume access to such gradient measurements in this work. Instead, we perform regression and learn the underlying low-dimensional subspace from paired input-output samples of the function alone. Moreover, we do not require solving an eigenvalue problem, which may be costly in high-dimensional settings.%

Finally, there have been a few other works that consider an iterative approach of a feature learning or dimension reduction step combined with a subsequent model for prediction. For instance, given a set of importance scores for each feature, one can try to improve the performance of a random forest by only using the features with large importance scores. To avoid removing features that may have some influence even if the importance score is low, one can also use the scores to \emph{reweight} the distribution used to select a set of features along which to make splits in each tree and iterate this procedure by using the reweighted random forest to compute new feature importance scores. This is the mechanism for \emph{iterative random forests} proposed by \citet{basu2018iterative}. We consider a variant of that approach in Section \ref{sec:weighted_mondrian}, where we use a gradient-based feature importance score to reweight the distribution over the covariate dimensions in the Mondrian process. Of course, this approach is limited to updating feature importance scores of the covariates rather than more general features made of linear combinations of covariates. Iterative algorithms that alternatively estimate the EGOP and a regression function \edit{using} kernel machines have been studied as well \citep{radhakrishnan2023mechanismfeaturelearningdeep} where the EGOP and the model are kernel estimators. Neither work has the corresponding statistical guarantees and consistency rates we obtain here for the \TRIM algorithm.

\subsection{Outline} The remainder of this paper is organized as follows. Section~\ref{sec:regression} presents the structured regression problem and the general form of the random forest and EGOP estimator we will use to learn the low-dimensional subspace and the regressor. Section~\ref{sec:methodology} introduces our full algorithm built from Mondrian processes and linear transformations of the data called Transformed Iterative Mondrian (\texttt{TrIM}) forests. Section~\ref{sec:theory} presents the consistency guarantees and Section~\ref{sec:simulation} shows the experimental results on small and large-scale examples. 

\section{Regression setting and notation} \label{sec:regression}

Suppose a data set $\mathcal{D}_n := \{(X_1, Y_1), \ldots, (X_n,Y_n)\}$ is given consisting of $n$ i.i.d. samples from a random pair $(X,Y) \in \RR^d \times \RR$ such that $\EE[Y^2] < \infty$.  Let $\mu$ denote the unknown distribution of $X$ and assume
\begin{align}\label{e:model1}
Y = f(X) + \ee,
\end{align}
for some unknown function $f: \RR^d \to \RR$ and noise $\ee$ satisfying $\EE[\ee|X] = 0$ and $\mathrm{Var}(\ee|X) = \sigma^2 < \infty$ almost surely. The goal is to estimate the \edit{regressor} $f(x) = \EE[Y | X = x]$, i.e., the conditional mean for the output. We make the additional assumption that the function $f$ has low-dimensional structure as follows.

\begin{assumption} \label{assumption:f_ridge} The function $f$ is of the form
\begin{align}\label{e:multi-index_model}
f(x) = g(Bx), \quad x \in \RR^d,
\end{align}
where $g\colon \RR^s \to \RR$ and $B \in \RR^{s \times d}$ for $1 \leq s \leq d$.    
\end{assumption} 
This dimension reduction model is commonly referred to as a \emph{multi-index model} or a \emph{ridge function}~\citep{pinkus2015ridge, hastie2009elements}. This assumption implies that the regression function depends only on the inputs $\langle b_1, x \rangle, \ldots, \langle b_s, x \rangle$, where $\{b_i\}_{i=1}^s$ are the rows of $B$. The subspace $S \coloneqq \mathrm{span}(\{b_i\}_{i=1}^s)$ will be called the associated \emph{relevant feature subspace} of the model. %
Our goal in this work is to estimate a function of the form in~\eqref{e:multi-index_model} by simultaneously estimating $f$ and $B$. Our estimators for $f$ based on random forest models are described in Section~\ref{sec:random_forest_regression}. We will identify $B$ using an EGOP matrix estimate, also known as a diagnostic matrix from the active subspace literature, described in Section~\ref{sec:EGOP}. \edit{While estimating the matrix $B$ is not identifiable for a given multi-index model, our goal is to identify a low-dimensional subspace with $s \ll d$ and corresponding $g$ that results in an accurate estimator for $f$.}

\subsection{Random forest regression} \label{sec:random_forest_regression}

The approach we will take for estimating $f$ is to build a \emph{random forest estimator} from a random partition of the input space and the data set $\mathcal{D}_n$. Let $\mathcal{P}$ be a random partition of $\RR^d$. The regression tree estimator based on $\mathcal{P}$ is
\begin{align}\label{e:tree}
    \hat{f}_n(x, \mathcal{P}) := \sum_{i=1}^n \frac{1_{\{X_i \in Z_x\}}}{\mathcal{N}_n(x)}Y_i,
\end{align}
where $Z_x$ is the cell of $\mathcal{P}$ that contains $x$ and $\mathcal{N}_n(x) := \sum_{i=1}^n 1_{\{X_i \in Z_x\}}$ is the number of points in $Z_x$. If $\mathcal{N}_n(x) = 0$, then it is assumed that $\hat{f}_n(x, \mathcal{P}) = 0$. The random forest estimator based on $\mathcal{P}$ is defined by averaging $M$ i.i.d. copies of the tree estimator, i.e.
\begin{align}\label{e:forest}
    \hat{f}_{n,M}(x) := \frac{1}{M} \sum_{m=1}^M \hat{f}_n(x, \mathcal{P}_m),
\end{align}
where $\mathcal{P}_1, \ldots, \mathcal{P}_M$ are $M$ i.i.d. copies of $\mathcal{P}$. The performance and asymptotic properties of this estimator depends on how the partition $\mathcal{P}$ is generated.

\subsection{Estimation of the expected gradient outer product}\label{sec:EGOP}

The first step of our algorithm will be a feature learning step that computes a linear transformation approximating a projection onto the relevant feature subspace $S$. We use a gradient-based approach as described in the introduction, and begin by defining the \emph{expected gradient outer product} (EGOP) matrix
\begin{align}\label{e:EGOP}
H = \EE[\nabla f(X) \nabla f(X)^T].
\end{align}
This matrix $H \in \R^{d \times d}$ \edit{corresponds to the second moment of the random variable $\nabla f(X)$. %
Thus, the projection $v^THv$ along any input direction $v \in \R^d$ measures the variability of the gradient of $f$ with respect to this direction.} If the variability is small along $v$, then $f$ can be interpreted having a nearly constant gradient along $v$ on average over the support of $X$.

If we suppose $f$ satisfies Assumption~\ref{assumption:f_ridge} for $B \in \RR^{s \times d}$, then $\nabla f(x) = B^T \nabla g(Bx)$, and
\[H = B^T\EE[\nabla g(BX) \nabla g(BX)^T]B.\]
Thus, the rank of $H$ is at most $s$, and the image of $H$ is contained in the relevant feature subspace $S$. \edit{Additionally, under mild assumptions (Assumption \ref{assump:f} and Assumption \ref{assump:mu} below are sufficient), the image is exactly the relevant feature subspace \citep{WuGradients2010}.} This motivates constructing the expected gradient outer product matrix to identify $S$. 
Without knowledge of the functional form of $f$ or the distribution of $X$, however, one must proceed by obtaining an approximation of $H$. 
In our algorithm, we will use the approximation proposed in \cite{Trivedietal2014} using the second-order difference quotient approximations of the derivatives. %

Given the data set $\mathcal{D}_n$ and an initial \edit{Mondrian forest} estimator $\hat{f}_{n}$ of $f$, we first consider for each $j \in [d]$ the second-order difference quotient estimator $\triangle_{j,t}\hat{f}_{n}(x)$ of $\partial_j f(x)$ for a step size $t > 0$. \edit{To define this, let $\{e_j\}_{j \in [d]}$ denote the standard basis vectors. Then, }for a function $h: \RR^d \to \RR$,
\begin{align}
 \triangle_{j,t}h(x) := \frac{h(x + te_j) - h(x - te_j)}{2t}
\end{align}
is the second-order difference quotient approximation to the derivative of $h$ at $x$. We then define the vector $\hat{\nabla}_t\hat{f}_n(x) \in \RR^d$ for $x \in \RR^d$ which is an approximation of the gradient vector $\nabla f(x)$ given step size $t$. For each $j = 1,\ldots d$, we define
\begin{align}\label{e:grad_approx}
(\hat{\nabla}_t\hat{f}_n(x))_j := \triangle_{j,t}\hat{f}_{n}(x)\mathbbm{1}_{E_{n,t,j}(x)},
\end{align}
where $E_{n,t, j}(x)$ is the event that the cells containing $x + te_j$ and $x - te_j$ both contain training data. Finally, let $\mathcal{X}_n = \{x_i\}_{i=1}^n$ be i.i.d. samples of the input random variable $X$ that are independent of $\mathcal{D}_n$.
The empirical approximation of $H$ is then given by the matrix 
\begin{align}\label{e:H_approx}
\hat{H}_{n,t} := \frac{1}{n} \sum_{i=1}^n \hat{\nabla}_t \hat{f}_n(x_i) \hat{\nabla}_t\hat{f}_n(x_i)^T.
\end{align}
We use an independent sample of inputs $\mathcal{X}_n$ in~\eqref{e:H_approx} for estimating the expectation with respect to $X$ above in order to prove theoretical guarantees for the convergence of the estimator. In practice, one could hold out a collection of inputs from the data sample to use for the evaluation of the expected gradient outer product matrix. However, if data is limited this may not be practical. The algorithm we describe below and implement in the experiments reuses the input samples that were used to build the Mondrian forest estimator to compute~\eqref{e:H_approx}.

\section{TrIM: Transformed Iterative Mondrian Forests} \label{sec:methodology}

The approach we propose in this paper to estimate the regressor $f(x) = \EE[Y | X = x]$ is to combine a feature learning step consisting of a \edit{data-driven} choice of linear transformation of the input data, and the computation of a Mondrian forest estimator. Mondrian forests \citep{lakshminarayanan2014mondrian,lakshminarayanan2016mondrian, balog2016mondrian} belong to a computationally efficient class of random forest estimators, obtained as in \eqref{e:forest} where the random partition $\mathcal{P}$ is generated using a \emph{Mondrian process}. This stochastic process is parameterized by a \emph{lifetime} $\lambda > 0$, which determines the complexity of the random partition. 

The construction of a Mondrian tree with lifetime $\lambda$ restricted to an axis-aligned box in $\RR^d$ is shown in Algorithm~\ref{a:mondrian}. The Mondrian tree generating process is a recursive algorithm that partitions the input data into smaller and smaller blocks, where the partitioning is determined by the Mondrian process. The process is defined by a set of independent exponential clocks, one for each dimension of the input space. The first clock to ring determines the dimension of the split, and the point of the split is chosen uniformly at random from the range of the data along that dimension. The process is terminated when the first clock to ring has a time greater than the lifetime $\lambda$. An illustration of a Mondrian tree is shown in Figure \ref{fig:Mondrian_tree}. 

\begin{figure}[hbt!]
    \centering
    \centering
\begin{tikzpicture}[scale=0.8]
   
     \filldraw[color=black, fill=red!40, very thick](-1,1) rectangle (4, 6);

     \filldraw[color=black, fill=green!40, very thick](2,1) rectangle (4, 6 );
     \filldraw[color=black, fill=white, thick]  (2,3) circle (8pt);
     \draw[minimum size=5mm]   (2,3)  node {\footnotesize .8};

    \filldraw[color=black, fill=blue!40, very thick](-1,2) rectangle (2, 6);
     \filldraw[color=black, fill=white, thick]  (0,2) circle (8pt);
     \draw[minimum size=5mm]   (0,2)  node {\footnotesize 1.7};
     
     \filldraw[color=black, fill=white, thick]  (2,3) circle (8pt);
     \draw[minimum size=5mm]   (2,3)  node {\footnotesize.8};

     \filldraw[color=black, fill=magenta!40, very thick](2,4) rectangle (4, 6 );
     \filldraw[color=black, fill=white, thick]  (3,4) circle (8pt);
        \draw[minimum size=5mm]   (3,4)  node {2.2};
     
     \filldraw[color=black, fill=white, thick]  (0,2) circle (8pt);
     \draw[minimum size=5mm]   (0,2)  node {\footnotesize 1.7};
     \filldraw[color=black, fill=white, thick]  (2,3) circle (8pt);
     \draw[minimum size=5mm]   (2,3)  node {\footnotesize .8};

     \filldraw[color=black, fill=yellow!40, very thick](-1,4.5) rectangle (2, 6);

\filldraw[color=black, fill=white, thick]  (0,4.5) circle (9pt);
\draw[minimum size=5mm]   (0,4.5)  node {\footnotesize 3.1};
\filldraw[color=black, fill=white, thick]  (3,4) circle (9pt);
        \draw[minimum size=5mm]   (3,4)  node {\footnotesize 2.2};
     
     \filldraw[color=black, fill=white, thick]  (0,2) circle (9pt);
     \draw[minimum size=5mm]   (0,2)  node {\footnotesize 1.7};
     \filldraw[color=black, fill=white, thick]  (2,3) circle (9pt);
     \draw[minimum size=5mm]   (2,3)  node {\footnotesize .8};

\end{tikzpicture} \hspace{.2in}
\begin{tikzpicture}[scale=0.9]
 \draw[->, thick] (3,4.8) -- (3,.5);

 \draw[gray, thick] (0,4.8) -- (0,4);
 \draw[gray, thick] (0,4) -- (-1,3.1);
 \draw[gray, thick] (0,4) -- (1,2.6);
 \draw[gray, thick] (-1,3.1) -- (-1.6,.8);
 \draw[gray, thick] (-1,3.1) -- (-.3,1.7);
  \draw[gray, thick] (-.6,.8) -- (-.3,1.7);
  \draw[gray, thick] (.1,.8) -- (-.3,1.7);
    \draw[gray, thick] (.6,.8) -- (1,2.6);
  \draw[gray, thick] (1.6,.8) -- (1,2.6);
 
 \draw[dotted, gray, thick] (-1,3.1) -- (3.1,3.1);
 \draw[dotted, gray, thick] (0,4.8) -- (3.1,4.8);
 \draw[dotted, gray, thick] (0,4) -- (3.1,4);
 \draw[dotted, gray, thick] (1,2.6) -- (3.1,2.6);
   \draw[dotted, gray, thick] (-.3,1.7) -- (3, 1.7);
  \draw[dashed, gray, thick] (-1.6,.8) -- (3, .8);

 \draw  (3.3,4.8) node { \footnotesize 0};
 \draw  (3.3,4) node { \footnotesize .8};
 \draw  (3.4,3.1) node { \footnotesize 1.7};
 \draw  (3.4,2.6) node {\footnotesize 2.2};
 \draw  (3.4,1.7) node {\footnotesize 3.1};
 \draw  (3.6,.8) node {\footnotesize $\lambda = 4$};
 
 \draw[gray, thick] (2.9,4.8) -- (3.1,4.8);
 \draw[gray, thick] (2.9,4) -- (3.1,4);
 \draw[gray, thick] (2.9,3.1) -- (3.1,3.1);
 \draw[gray, thick] (2.9,2.6) -- (3.1,2.6);
 \draw[gray, thick] (2.9,1.7) -- (3.1,1.7);
 \draw[gray, thick] (2.9,.8) -- (3.1,.8);

\filldraw[black] (0,4.8) circle (2pt);
\filldraw[black] (0,4) circle (2pt);
\filldraw[black] (-1,3.1) circle (2pt);
\filldraw[black] (1,2.6) circle (2pt);
\filldraw[black]  (-.3,1.7) circle (2pt);

\filldraw[color=black, fill=red!40, thick] (-1.6,.8) circle (2pt);
\filldraw[color=black, fill=blue!40, thick]  (-.6,.8) circle (2pt);
\filldraw[color=black, fill=yellow!40, thick]  (.1,.8) circle (2pt);
\filldraw[color=black, fill=green!40, thick]  (.6,.8) circle (2pt);
\filldraw[color=black, fill=magenta!40, thick]  (1.6,.8) circle (2pt);
 
\end{tikzpicture}
    \vspace{-0.8cm}
    \caption{An illustration of the axis-aligned hierarchical partition generated by a Mondrian process with lifetime $\lambda = 4$ and the associated tree embedded on a vertical time axis.} 
    \label{fig:Mondrian_tree}
\end{figure}

\begin{algorithm}[hbt!]
\SetAlgoLined
\DontPrintSemicolon
\KwIn{A dataset $X$, lifetime $\lambda$, and dimension $d$}
\KwOut{A partition of $X$ into subsets based on the Mondrian partitioning process}
\BlankLine
\For{each dimension $c \in [d]$}{
  Compute the lower bound $l_c$ and upper bound $u_c$ of $X$ along dimension $c$\;
  Associate an independent exponential clock with rate $|u_c - l_c|$ to dimension $c$\;
}
\BlankLine
$T \gets$ time of the first clock to ring\;
\edit{$c \gets$} dimension of the clock that rang first\;
\If{$T > \lambda$}{
  \Return{Terminate the process}\;
}
\BlankLine
Choose a point $\xi$ uniformly at random from $[l_c, u_c]$\;
Split $X$ into $X_{<} = \{x \in X | x_c < \xi\}$ and $X_{\ge} = \{x \in X | x_c \ge \xi\}$ with a hyperplane perpendicular to $X_c$ at point $\xi$\;
\BlankLine
Discard the remaining $d - 1$ clocks\;
Recursively call the Mondrian Tree Generating Process with parameter ($X_{<}, \lambda, d$)\;
Recursively call the Mondrian Tree Generating Process with parameter ($X_{\ge}, \lambda, d$)\;
\caption{Mondrian Tree Generating Process \citep{roy2008mondrian}}
\label{a:mondrian}
\end{algorithm}

\begin{figure}
    \centering
    \includegraphics[width=0.9\textwidth]{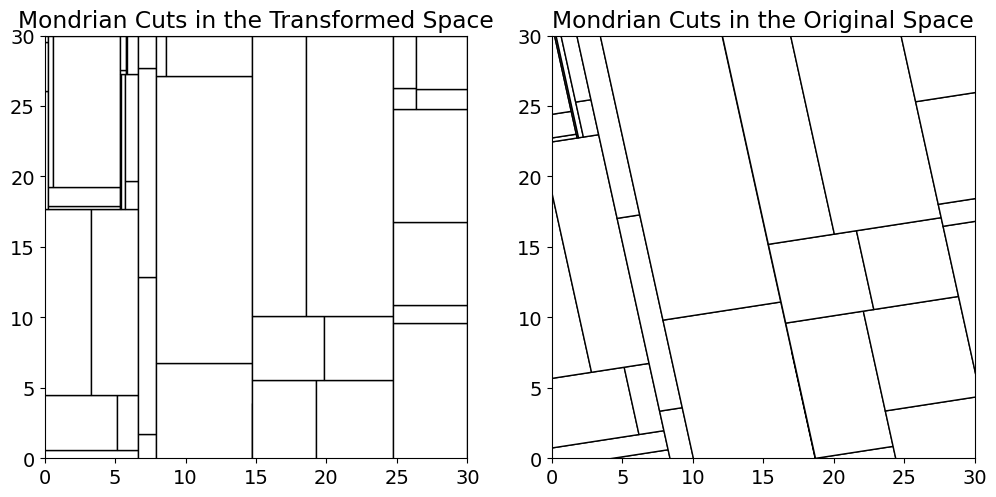}
    \caption{The original 2-dimensional Euclidean space is transformed by rotating by 10 degrees clockwise and \edit{scaling}  the x-axis by 0.9 and y-axis by 0.8. In other words, the original space is transformed by the matrix $A = \begin{bmatrix}
    0.8 \cos(10^\circ) & -0.7 \sin(10^\circ) \\
    \sin(10^\circ) & 0.8 \cos(10^\circ)
    \end{bmatrix}$. The Mondrian cuts are made in the transformed space and then transformed back to the original space by applying $A^{-1}$. 
    Left: The Mondrian cuts in the transformed space. Right: The Mondrian cuts in the original space. }
    \label{fig:Mondrian_cuts_transformed_space}
\end{figure}

As is clear from the construction, an important hyperparameter for Mondrian forests is the \emph{lifetime} parameter $\lambda$, which determines how long the Mondrian partitioning process is run. The choice of this parameter determines the complexity of the estimator and thus it should depend on the amount of data available. This has been observed theoretically by \citet{mourtada2020minimax}, \citet{OReillyTran2021minimax} and \citet{OReilly_ObliqueMondrian}, where minimax rates of convergence for Mondrian forests and their generalizations are obtained by precisely tuning the lifetime parameter with the amount of data points $n$. The significance of this parameter will be highlighted in the forthcoming theoretical results and numerical experiments. \edit{In particular, convergence guarantees rely on scaling this parameter appropriately with the amount of data, see Theorems \ref{t:Hn_approx_rate} and \ref{thm:conv_rate}.} 

The Mondrian process has many appealing computational properties such as the Markov property and spatial consistency, which make Algorithm~\ref{a:mondrian} amenable to efficient online algorithms. \edit{For instance, spatial consistency implies the distribution of a Mondrian partition on a given window has the same distribution as the Mondrian partition on a larger window, restricted to the original one, enabling efficient algorithms for extending a Mondrian tree to new data \citep{lakshminarayanan2014mondrian}.} The major drawback, as described in the introduction, is that the process is not data-adaptive. This shortcoming becomes particularly problematic in high-dimensional settings where learning from features is essential for achieving good performance. To address this, we propose a novel two-phase approach. First, we apply a data-adaptive linear transformation to the input data; then, we compute a Mondrian forest on the transformed data. We first dive into the details of the first step, which can be interpreted as a generalized form of feature selection: creating a new set of features based on linear combinations of covariates in the original input space. To compute the linear transformation matrix, we first train a Mondrian forest estimator $\hat{f}_n$ on the data set $\mathcal{D}_n$ with lifetime $\lambda$ using Algorithm~\ref{a:mondrian}. We then compute the approximation $\hat{H}_{n,t}$ as in \eqref{e:H_approx} to the expected gradient outer product $H$ by approximating the gradient vector at each point in the data set and taking an average of outer products of the gradient vectors. The pseudocode for computing $\hat{H}_{n,t}$ is shown in Algorithm~\ref{a:estimate H}. In Theorem \ref{t:Hn_approx_rate} of Section \ref{sec:theory}, we prove consistency of this estimator for $H$ and obtain a convergence rate for the estimation of \edit{multi-index models} given a particular scaling of the parameters $\lambda$ and $t$ with respect to the amount of data $n$. 

\begin{algorithm}[hbt!]
\SetAlgoLined
\DontPrintSemicolon
\KwIn{$\mathcal{D}_n = \{(x_i,y_i)\}_{i=1}^n \subset \RR^d \times \RR$; step size $t > 0$; lifetime $\lambda > 0$; a (transformed) Mondrian forest estimator $\hat{f}_{n, \lambda}$ on $\mathcal{D}_n$}
\KwOut{Estimated gradient outer product matrix $\hat{H}_{n,t}$}
\BlankLine

\For{each point $(x,y) \in \mathcal{D}_n$}{
        \For{each dimension $j = 1$ \KwTo $d$}{
            Compute the approximation of the $j$-th component of the gradient vector at $x$:
            \[ (\hat{\nabla}\hat{f}_{n,\lambda}(x))_j := \frac{\hat{f}_{n,\lambda}(x + te_j) - \hat{f}_{n,\lambda}(x - te_j)}{2t} \]
        }
    }
\BlankLine

Compute the matrix $\hat{H}_{n,t}$ as an average of outer products of the gradient vectors:
    \[ \hat{H}_{n,t} := \frac{1}{n} \sum_{i=1}^n \hat{\nabla}_t \hat{f}_{n, \lambda}(x_i) \hat{\nabla}_t\hat{f}_{n,\lambda}(x_i)^T \]
\caption{Pseudocode for computing $\hat{H}_{n,t}$}
\label{a:estimate H}    
\end{algorithm}

Before applying the linear transformation, we normalize the matrix $\hat{H}_{n,t}$ to avoid adjusting the lifetime parameter $\lambda$. Specifically, we compute the normalized matrix 
\begin{align}\label{e:An}
A_n = d \cdot \hat{H}_{n,t}/\| \hat{H}_{n,t} \|_{2,1}
\end{align}
\edit{when $\| \hat{H}_{n,t} \|_{2,1} \neq 0$,} where the $L_{2,1}$ norm $\| \cdot \|_{2,1}$ is given by the sum of the Euclidean norms of the columns of the matrix, i.e., $\|A\|_{2,1} = \sum_{j=1}^d \|A_{\cdot, j}\|_2$. \edit{If $\| \hat{H}_{n,t} \|_{2,1} = 0$, we set $A_n = 0$.} %
Once we have computed the linear transformation $A_n$, we transform the data set $\mathcal{D}_n$ by applying $A_n$ to each point. We then compute a Mondrian forest estimator using the transformed data set $\mathcal{D}_n^{+} = \{(A_n x_i, y_i)\}_{i=1}^n$ with lifetime $\lambda$, which implicitly computes an oblique Mondrian forest estimator $\hat{f}_{n,\lambda}^+$. \edit{The normalization in \eqref{e:An} by the $L_{2,1}$ norm ensures that an oblique Mondrian forest estimator with lifetime $\lambda$ trained on the original data set $\mathcal{D}_n$ gives the same estimator (in distribution) as training a standard Mondrian forest with \emph{the same lifetime} $\lambda$ on the transformed data set $\mathcal{D}_n^{+} = \{(A_n x_i, y_i)\}_{i=1}^n$; see Lemma \ref{l:fn_is_STIT}.} 

Algorithm \ref{a:estimator} presents pseudocode describing this procedure. %
With appropriate scaling of $\lambda$ and $t$ with respect to the amount of data $n$, Theorem \ref{thm:conv_rate} in Section \ref{sec:theory} proves the consistency of this estimator \edit{and provides a convergence rate} as $n$ goes to infinity.

\begin{algorithm}[hbt!]
\SetAlgoLined
\DontPrintSemicolon
\KwIn{$\mathcal{D}_n = \{(x_i,y_i)\}_{i=1}^n \subset \RR^d \times \RR$; lifetime parameter $\lambda > 0$; number of trees $M$; transformation matrix $\hat{H}_{n,t}$}
\KwOut{Mondrian forest estimator $\hat f^+_{n,\lambda}$}
\BlankLine

Compute the normalized matrix
\[
A_n = d \cdot \frac{\hat{H}_{n,t}}{\| \hat{H}_{n,t} \|_{2,1}}
\]
\BlankLine
Transform data set $\mathcal{D}_n$: for each $i = 1, \ldots, n$,  $x_i \mapsto A_n x_i$; i.e. Set $\mathcal{D}_n^{+} = \{(A_n x_i, y_i)\}_{i=1}^n$\;
Compute Mondrian forest estimator $\hat{f}_{n,\lambda}$ consisting of $M$ Mondrian trees with lifetime $\lambda$ on the transformed dataset $\mathcal{D}_n^{+}$\ using Algorithm \ref{a:mondrian}\;
Return $\hat f^+_{n,\lambda}$ where $\hat f^+_{n,\lambda}(x) := \hat f_{n,\lambda}(A_nx)$
\caption{Pseudocode for computing estimator $\hat{f}^+_{n,\lambda}$}
\label{a:estimator}
\end{algorithm}

We can extend this approach by iteratively updating the linear transformation $\hat{H}_{n,t}$ and the Mondrian forest estimator $\hat{f}_{n,\lambda}$ to improve the performance of the estimator until some stopping criterion is met. This leads to our overall algorithm, which we call the \emph{Transformed Iterative Mondrian (\texttt{TrIM}) Forest}. The pseudocode for \TRIM is shown in Algorithm \ref{a:TMFE}.

\begin{algorithm}[hbt!]
\SetAlgoLined
\DontPrintSemicolon
\KwIn{$\mathcal{D}_n = \{(x_i,y_i)\}_{i=1}^n \subset \RR^d \times \RR$; lifetime parameter $\lambda > 0$; number of iterations $K$; step-size $t > 0$}
\KwOut{Oblique Mondrian forest estimator $\hat{f}^+_{n,\lambda}$}
\BlankLine
$\hat{H}_{n,t} \leftarrow I_d$\;
\For{$k = 1$ \KwTo $K$}{
    Given $\mathcal{D}_n, \lambda$ and $\hat{H}_{n,t}$, compute the transformed  Mondrian forest estimator $\hat{f}^+_{n,\lambda}$ using Algorithm \ref{a:estimator}\;
    Compute the linear transformation matrix $\hat{H}^{+}_{n,t}$ using Algorithm \ref{a:estimate H} with input $\mathcal{D}_n, t, \lambda, \hat{f}^+_{n,\lambda}$\;
    Set $\hat{H}_{n,t} \leftarrow \hat{H}^{+}_{n,t}$\;
}
\caption{Transformed Iterative Mondrian (\texttt{TrIM}) Forest}
\label{a:TMFE}
\end{algorithm}

\section{Theoretical Results}\label{sec:theory}

In this section, we provide theoretical justification for our algorithm by providing a rate of convergence for the expected error of the estimated EGOP matrix (see Theorem~\ref{t:Hn_approx_rate}) as well as for an \edit{\emph{honest} version} of the subsequent oblique Mondrian estimator (see Theorem~\ref{thm:conv_rate}). \edit{In particular, we assume that the data used to built the EGOP estimate has been split from, and is thus independent of, the data used to generate estimates within each cell of a tree. This implies the forests satisfy the honesty condition introduced in earlier work on the statistical analysis of random forests \citep{wager2018estimation}.} 

We will make the following assumptions about the distribution $\mu$ of the input random variable. First recall that a function $f: \RR^d \to \RR$ is $L$-Lipschitz if $|f(x) - f(y)| \leq L\|x - y\|_2$ for all $x, y \in \RR^d$.

\begin{assumption}\label{assump:mu}
Assume the distribution $\mu$ \edit{is supported on the unit cube $[0,1]^d \subset \RR^d$ and has }%
a density $p$ that is $C_p$-Lipschitz and strictly positive on \edit{$[0,1]^d$}.%
\end{assumption}

We will also assume the regression function satisfies the following regularity condition.

\begin{assumption}\label{assump:f}
Assume $f: \edit{\RR^d} \to \RR$ satisfies, for some finite $L > 0$ \edit{and $\beta \in (0,1]$},
\[\|\nabla f(x) - \nabla f(y)\|_2 \leq L\|x - y\|_2^{\edit{\beta}} \text{  and  } \|\nabla f(x)\|_2 \leq L, \text{ for all } x,y \in \edit{\RR^d}.\]
\end{assumption}

Our first main result is the following asymptotic error bound for the estimator $\hat{H}_{n,t}$ of the EGOP $H$ defined in \eqref{e:H_approx} as the amount of data $n$ grows. Recall the estimator depends on the lifetime parameter $\lambda$ of the initial Mondrian forest estimator and the stepsize parameter $t$ in the difference quotient approximation. To obtain consistency of this estimator, the lifetime must grow to infinity, and the stepsize decay to zero appropriately with $n$. Throughout, $\|\cdot\|$ will denote the \edit{$\ell_2 - \ell_2$} operator norm of a matrix and $\gtrsim$ denotes a lower bound on the asymptotic scaling.

\begin{theorem}\label{t:Hn_approx_rate}
Let $\mu$ and $f$ satisfy Assumptions \ref{assump:mu} and \ref{assump:f}. Let $\hat{H}_n = \hat{H}_{n, t_n}$ be the estimator of $H$ as in \eqref{e:H_approx} built from a Mondrian forest estimator $\hat{f}_n$ with lifetime $\lambda_n$ and number of trees $M_n$. \edit{For $\beta \leq \frac{1}{2}$,   letting
$\lambda_n \sim n^{\frac{1}{d + 2+2\beta}}$, $M_n \gtrsim n^{\frac{1}{d + 2 + 2\beta}}$ and $t_n \sim n^{-\frac{1}{d + 2+2\beta}}$ as $n \to \infty$ gives
\begin{align*}
    \EE[\|\hat{H}_n - H\|] \lesssim  
    n^{-\frac{\beta}{d+2 + 2\beta}},
\end{align*}
and for $\beta > \frac{1}{2}$,} letting $\lambda_n \sim n^{\frac{1}{d + 3}}$, $M_n \gtrsim n^{\frac{1}{d + 3}}$ and $t_n \sim n^{-\frac{3}{4d + 12}}$ as $n \to \infty$ gives
\begin{align*}
    \EE[\|\hat{H}_n - H\|] \lesssim n^{-\frac{3}{4d+12}}, %
\end{align*}
where the expectation is taken with respect to $\mathcal{D}_n$, $\mathcal{P}$, and $\mathcal{X}_n$.   
\end{theorem}

As a corollary we obtain an approximation rate of the normalized matrix $A_n$ defined in \eqref{e:An} that we apply to the data in Algorithm \ref{a:estimator}. 

\begin{corollary}\label{cor:An_error}
Consider the setting of Theorem \ref{t:Hn_approx_rate}. \edit{Define $A :=  d \cdot H/\|H\|_{2,1}$ when $\|H\|_{2,1} = 0$ and set $A = 0$ otherwise. For $\beta \leq \frac{1}{2}$,  letting
$\lambda_n \sim n^{\frac{1}{d + 2+2\beta}}$, $M_n \gtrsim n^{\frac{1}{d + 2 + 2\beta}}$ and $t_n \sim n^{-\frac{1}{d + 2+2\beta}}$ as $n \to \infty$ gives
\begin{align*}
    \EE[\|A_n - A\|] \lesssim  
    n^{-\frac{\beta}{d+2 + 2\beta}}.
\end{align*}
and for $\beta > \frac{1}{2}$,} letting $\lambda_n \sim n^{\frac{1}{d + 3}}$, $M_n \gtrsim n^{\frac{1}{d + 3}}$ and $t_n \sim n^{-\frac{3}{4d + 12}}$ %
as $n \to \infty$ gives
\begin{align*}
    \EE[\|A_n - A\|_{2,1}] \lesssim n^{-\frac{3}{4d+12}},%
\end{align*}
where the expectation is taken with respect to $\mathcal{D}_n$, $\mathcal{P}$, and $\mathcal{X}_n$.
\end{corollary}

The proofs of the above results appear in Appendix~\ref{app:estimate_H}. \edit{Our approach to showing Theorem~\ref{t:Hn_approx_rate} uses similar techniques to those developed in~\cite{Trivedietal2014} where the difference quotient approximation of the EGOP was introduced. The approach in~\cite{Trivedietal2014} uses a kernel estimator for the regression function, while we use a Mondrian estimator which has additional randomness. The existing work has %
weaker assumptions on $f$ and $\mu$ and aims only to show consistency under minimal assumptions, whereas we assume conditions under which mimimax optimality of the Mondrian forest estimator holds in order to obtain a rate of convergence for estimating $H$. The main technical novelty in the proof is %
a risk bound on the estimation of the gradient of $f$ via the second order difference quotient, which relies on the asymptotic behavior of the expected risk for a Mondrian forest estimator with respect to the input shifted by $t_n$; see Lemma \ref{p:shifted_risk_bnd}.} 

\edit{We note that the case $\beta > \frac{1}{2}$ is treated separately because for this level of smoothness we must use a suboptimal rate of convergence for the Mondrian estimator. The minimax optimal rate in this case has only been obtained by conditioning on the input away from the boundary. If we want to apply this rate, we could approximate a \emph{conditional} EGOP defined as $H_{\varepsilon} := \mathbb{E}[\nabla f(X) \nabla f(X) | X \in [\varepsilon, 1 - \varepsilon]^d]$ for a fixed $\varepsilon \in (0, \frac{1}{2})$. Since the support of $X$ conditioned on $X \in [\varepsilon, 1 - \varepsilon]^d$ is not the entire support of $X$, there is no guarantee $H_{\varepsilon}$ has an image that spans the entire relevant feature subspace. To ensure this, we need to estimate the unconditional EGOP $H$, which requires using the risk bound for the unconditional risk of a Mondrian estimator, which is suboptimal for $\beta > \frac{1}{2}$.}

We next turn to proving a convergence rate for the estimator of the regression function $f$ obtained from Algorithm \ref{a:estimator} after applying the learned transformation $A_n$ to the input data. %
We first characterize the distribution of the random partition of the original data $\mathcal{D}_n$ induced by this algorithm. Indeed, \citet{OReillyTran2021} (see also \cite{OReilly_ObliqueMondrian}) showed that the random partition of the data obtained by applying a linear transformation and then running a Mondrian process has the same distribution as randomly partitioning the data with a \emph{stable under iteration (STIT) process} \citep{Nagel2005} with a discrete directional distribution determined by the linear transformation. %
The distribution and construction of a STIT process is determined by a \emph{lifetime} parameter $\lambda > 0$ and a %
probability measure $\phi$ on $\mathbb{S}^{d-1}$ called the \emph{directional distribution}, which governs the directions where partitioning occurs. %
When the directional distribution $\phi$ is uniform over the standard unit basis vectors, the corresponding STIT process with lifetime $\lambda$ is the same as the Mondrian process with lifetime $\lambda$. For more on this general class of random partitioning processes, the corresponding class of oblique random forests, and their applications, we refer the reader to \cite{OReillyTran2021, OReillyTran2021minimax} and the references therein. Following \cite{OReilly_ObliqueMondrian}, we will refer to STIT forests where $\phi$ is a discrete measure as \emph{oblique Mondrian forests}.
The following result follows immediately from Proposition 17 by \cite{OReilly_ObliqueMondrian} under the aforementioned honesty assumption.

\begin{lemma}\label{l:fn_is_STIT}
\edit{Assume the matrix $A_n$ is applied to a dataset $\mathcal{D}_n'$ that is independent of the original dataset $\mathcal{D}_n$ used to compute $A_n$ and the forest estimator $\hat{f}_n$ output by Algorithm \ref{a:estimator} is built from $M$ Mondrian processes $\mathcal{P}'$ independent from those used to obtain $A_n$.
Then, conditioned on $A_n$, $\hat{f}_n$} is an oblique Mondrian forest estimator with $M$ trees, lifetime $\lambda$ and directional distribution
$$\phi_{n} = \frac{1}{2}\sum_{i=1}^d \frac{\|a^{(n)}_i\|_2}{\|A_n\|_{2,1}}\left( \delta_{a^{(n)}_i/\|a^{(n)}_i\|_2} + \delta_{- a^{(n)}_i/\|a^{(n)}_i\|_2}\right),$$
where $(a^{(n)}_1, \ldots, a^{(n)}_d)$ are the columns of $A_n \in \RR^{d \times d}$ and $\|\cdot\|_{2,1}$ is the $L_{2,1}$ norm. %
Recall that from Algorithm~\ref{a:estimator}, $\|A_n\|_{2,1} = \sum _{i=1}^{d}\|a^{(n)}_{i}\|_{2} = d$.
\end{lemma}

We now obtain an upper bound on the expected risk of an honest %
version of the estimator obtained from Algorithm~\ref{a:estimator} when the underlying regression function $f$ is assumed to be a ridge function. %
The proof can be found in Appendix~\ref{app:proof_conv_rate_f}. 

\begin{theorem}\label{thm:conv_rate}
\edit{For each $n$, let $A_n$ be the estimator of the normalized EGOP from \eqref{e:An} and let $\hat{f}_{n}$ be the forest estimator obtained from Algorithm \ref{a:estimator} with lifetime parameter $\lambda_n > 0$ and number of trees $M_n$ under the assumptions of Corollary \ref{cor:An_error} and Lemma \ref{l:fn_is_STIT}}. %
Fix $\delta \in [0,1)$.

\edit{For $\beta \leq \frac{1}{2}$, assume $\lambda_n \sim n^{\frac{1}{d+2} + \frac{\beta(1 - \delta)(d-s)}{(d + 2)(d + 2 + 2\beta)}}$ and $M_n \gtrsim \lambda_n$.
Then, for all $n$ large enough, we have that with probability at least $1 - Cn^{-\frac{\beta\delta}{d + 2 + 2\beta}}$ with respect to $\mathcal{D}_n$, $\mathcal{P}$, and $\mathcal{X}_n$,
\begin{align*}
\EE[(\hat{f}_{n}(X) - f(X))^2] &\lesssim %
n^{-\frac{2}{d+2} - \frac{2\beta(1 - \delta)(d-s)}{(d + 2)(d + 2 + 2\beta)}}. %
\end{align*}
For $\beta > \frac{1}{2}$,} assume $\lambda_n \sim n^{\frac{1}{d+2} + \frac{3(1 - \delta)(d-s)}{4(d+3)(d+2)}}$ and $M_n \gtrsim \lambda_n$. %
For all $n$ large enough, we have that with probability greater than $1 - Cn^{-\frac{3\delta}{4d + 12}}$ with respect to $\mathcal{D}_n$, $\mathcal{P}$, and $\mathcal{X}_n$,
\begin{align*}
\EE[(\hat{f}_{n}(X) - f(X))^2] \lesssim n^{-\frac{2}{d+2} - \frac{3(1 - \delta)(d-s)}{2(d+3)(d+2)}}. %
\end{align*}
The expectations above are taken with respect to $X, \mathcal{D}'_n$, and $\mathcal{P}'$.
\end{theorem}

\begin{remark} The result of Theorem~\ref{thm:conv_rate} provides a convergence rate for a \TRIM forest after one iteration under the assumption the regression function $f$ is a \edit{multi-index model}. For $s = d$, this rate matches that obtained for Mondrian and STIT random forests for Lipschitz functions $f$ \citep{mourtada2020minimax, OReillyTran2021minimax}, and improves on this rate for $s < d$.  While it is challenging to account for dependencies between estimators for $f$ and $A$ in the repeated iterations, we expect the dependence on $d$ to become milder as the estimate of $A$ improves. Indeed, Corollary 9 in \cite{OReilly_ObliqueMondrian} implies that with the correct knowledge of the low-dimensional subspace, the convergence rate should scale as $n^{-2/(s+2)}$. In effect, this corresponds to the case $d=s$.
\end{remark}

\edit{\begin{remark}
    The results in Theorems~\ref{t:Hn_approx_rate} and \ref{thm:conv_rate} are given in terms of asymptotic moment bounds and highlight only the dependence on $n$ as $n \to \infty$. The proofs in the appendix detail more explicit upper bounds with constants depending on other parameters such as the dimension $d$, the H\"{o}lder constant $L$, and the distribution $\mu$ of $X$.
\end{remark}}

\edit{We now make a few additional comments on the preceding result, the proof of which appears in Appendix \ref{app:proof_conv_rate_f} and relies on a risk bound obtained in \cite{OReilly_ObliqueMondrian} for oblique Mondrian trees. First, observe that the choice of the parameter $\delta$ represents a trade-off; as $\delta$ decreases to zero, the rate of convergence of the upper bound is faster, but the probability with which the bound holds decreases. Second, given the regularity assumption on $f$ one might expect the first term in the exponent to be $\frac{-(2 + 2\beta)}{d+2 + 2\beta}$, matching the rate of convergence of a standard Mondrian forest for $f$ satisfying the smoothness condition in Assumption~\ref{assump:f} \citep{mourtada2020minimax}. %
This is not the case in our result because we have applied the risk bound in \cite{OReilly_ObliqueMondrian} for Lipschitz functions. We do not have an analogous risk bound under the stronger smoothness assumption we make here, as the corresponding result in \cite{OReilly_ObliqueMondrian} requires stronger assumptions on the matrix $\hat{H}_n$ than are satisfied.}

\edit{With an improved risk bound for oblique Mondrian estimators of functions satisfying the smoothness condition in Assumption~\ref{assump:f}, the rate of convergence only yields a mild improvement on the minimax optimal rate for H\"{o}lder functions \emph{without} the multi-index model assumptions. In particular, such a rate will still depend on the ambient dimension $d$ rather than on the dimension $s$ of the relevant feature subspace. This is because the estimation error of $\hat{H}_n$ depends on the nonparametric convergence rate for a standard Mondrian estimator. There are a few possibilities for improving the rate assuming additional knowledge about the regression function. First, by having direct access to the evaluations of the function's gradient at the training points, $\hat{H}_n$ becomes a covariance estimator with a $n^{-1}$ convergence rate. Alternatively, if the dimension $s$ of the relevant feature subspace is known, one can perform an eigendecomposition of $\hat{H}_n$ and use the best rank-$s$ approximation to $\hat{H}_n$ to project the data onto the approximation of the lower-dimensional relevant feature subspace. %
The estimation error for this projection may be an improvement over the full EGOP estimate, but would incur the additional cost of the eigendecomposition (an expensive operation in high dimensions). Recently, \cite{liu2024learning} obtained a convergence rate of $n^{-1}$ for learning the relevant feature subspace using the generalized contour regression method, but this result makes additional assumptions on the data distributions and requires knowledge that we do not assume here. In general, projecting the data onto an estimated lower-dimensional feature space risks the estimator missing a relevant feature. This is in contrast to our approach which, rather than restricting to a small set of features, places \emph{weights} on a set of features that span the whole feature space.} 

\edit{In this work, we try to minimize the assumptions and computational cost of our approach, which comes at the expense of a weaker theoretical guarantee on the method's performance, particularly in higher dimensional spaces. To address this computational-statistical trade-off, we advocate for performing multiple iterations of the TrIM algorithm. We observe that experimentally, multiple iterations improve the performance. A theoretical study of this approach, including the number of iterations needed to obtain optimal theoretical guarantees, will be explored in future work.}

\subsection{Weighted Mondrian estimators}\label{sec:weighted_mondrian}

We now consider the setting where the algorithm is restricted to a diagonal matrix transformation that scales each covariate by a gradient-based importance weight. By Lemma \ref{l:fn_is_STIT} this is equivalent to reweighting the axis-aligned cuts of a Mondrian process. We also now assume that the regression function is \emph{sparse}. That is, let $S \subseteq \{1, \ldots, d\}$ be a subset of size \edit{$|S| = s < d$} that corresponds to a low dimensional set of relevant feature dimensions. Assume the regression setting of Section \ref{sec:regression} and additionally that the true function $f$ is of the form
\begin{align}\label{e:sparse_f}
f(x) = g(x_S) = g(\{x_i\}_{i \in S}),
\end{align}
for some function $g: \RR^s \to \RR$. This study is motivated by the iterative random forests algorithm~\citep{basu2018iterative} and the approach of \cite{deng2022towards}, which computes single-covariate feature scores to identify $S$.

For a Mondrian forest estimator $\hat{f}_n$ of $f$ built from the training data $\mathcal{D}_n$ and each covariate index $i$, we can define the feature importance score
\begin{align}\label{e:omega_i}
\omega^{(n,t)}_{i} = \frac{1}{n} \sum_{j=1}^n |(\hat{\nabla}_{t}\hat{f}_{n}(x_j))_i|^2,
\end{align} %
where $\hat{\nabla}_t\hat{f}_{n}(x)$ is the gradient estimator as defined in \eqref{e:grad_approx} and $\mathcal{X}_n = \{x_j\}_{j=1}^n$ are $n$ i.i.d. samples of the input random variable $X$ that are independent of $\mathcal{D}_n$. Note that these scores form the diagonal of the EGOP estimator \eqref{e:H_approx}. 

In Theorem~\ref{thm:omega_bnds} below we first show that these weights, with appropriate tuning of the parameters with $n$, are consistent estimators of the diagonal of the EGOP matrix normalized to have unit trace, which are nonzero for relevant covariate indices $i \in S$ and zero for $i \notin S$. %

\begin{theorem}\label{thm:omega_bnds}
Let $\hat{f}_{n}$ be a Mondrian forest estimator with lifetime $\lambda_n > 0$ and number of trees $M_n$ 
of a function $f$ as in \eqref{e:sparse_f} for some set of relevant features $S$ and $g$ satisfying Assumption \ref{assump:f}. \edit{Additionally, assume $\EE[\|\nabla f(X)\|_2^2] \neq 0$.} For each $i \in [d]$, let $\omega^{(n)}_i := \omega_i^{(n,t_n)}$ be defined as in \eqref{e:omega_i} \edit{and define the normalized weights
\begin{align*}
    \alpha_{i}^{(n)} := \frac{\omega_i^{(n)}}{\sum_{j=1}^d \omega_j^{(n)}},
\end{align*}
when $\sum_{j=1}^d \omega_j^{(n)} \neq 0$, and $\alpha^{(n)}_i = 0$ otherwise. }

\edit{For $\beta \leq 1/2$, letting $\lambda \sim n^{\frac{1}{d + 2+2\beta}}$, $M \gtrsim n^{\frac{1}{d + 2 + 2\beta}}$ and $t \sim n^{-\frac{1}{d + 2+2\beta}}$ 
gives %
\begin{align*}
    \EE\left[ \max_{i \in\{1, \ldots, d\}} \left|\alpha_i^{(n)} - \frac{\EE[|(\nabla f(X))_i|^2]}{\EE[\|\nabla f(X)\|_2^2]}\right|\right] \lesssim
    n^{-\frac{\beta}{d+2 + 2\beta}}.
\end{align*}}
\edit{For $\beta > 1/2$}, letting $\lambda_n \sim n^{\frac{1}{d + 3}}$, $M_n \gtrsim n^{\frac{1}{d + 3}}$ and $t_n \sim n^{-\frac{3}{4d + 12}}$ as $n \to \infty$ gives
\begin{align*}
     \EE\left[ \max_{i \in\{1, \ldots, d\}} \left|\edit{\alpha_i^{(n)}} - \frac{\EE[|(\nabla f(X))_i|^2]}{\EE[\|\nabla f(X)\|_2^2]}\right|\right] \lesssim n^{-\frac{3}{4d+12}}. %
\end{align*}
where the expectation is taken with respect to $\mathcal{X}_n$, $\mathcal{P}$ and $\mathcal{D}_n$.
\end{theorem}

Now consider the updated \emph{weighted} Mondrian forest estimator built from $n$ i.i.d. samples of $(X,Y)$ that are independent of $\mathcal{D}_n$, and $M_n$ Mondrian trees with lifetime $\lambda_n > 0$ and directional distribution
\begin{align}\label{e:model_mondrian}
\phi_n = \sum_{i=1}^d \frac{\omega^{(n)}_i}{\sum_{j=1}^d \omega^{(n)}_j} \delta_{e_i},
\end{align} 
where the weights $\omega^{(n)}_i > 0$ determine how often covariates are split in the partitioning process are obtained from the feature importance scores defined in \eqref{e:omega_i}. Theorem~\ref{t:conv_rate_weighted_mondrian} gives an asymptotic upper bound on the expected quadratic risk of this weighted Mondrian forest estimator. %
The proofs of these results are found in Appendix~\ref{app:proof_conv_rate_f_weighted}.%

\begin{theorem}\label{t:conv_rate_weighted_mondrian}
Assume $\mathrm{supp}(\mu) = [0,1]^d$ and that $\mu$ has a positive and Lipschitz density on its support. Let $\hat{f}_{n}$ be the weighted Mondrian forest estimator with directional distribution \eqref{e:model_mondrian}, lifetime $\lambda_n > 0$, and number of trees $M_n$ of a function $f$ as in \eqref{e:sparse_f} for some set of relevant features $S$ and $g$ satisfying Assumption \ref{assump:f}. Fix $\delta \in (0,1)$.

\edit{For $\beta \leq 1/2$, let $\lambda_n \sim n^{\frac{1}{d+2 + 2\beta} + \frac{\beta(1 - \delta)(d-s)}{(d + 2 + 2\beta)^2}}$ and $M_n \gtrsim \lambda_n^{2\beta}$. Then, for all $n$ large enough, with probability at least $1 - Cn^{-\frac{\beta\delta}{d + 2 + 2\beta}}$ with respect to $\mathcal{D}_n$, $\mathcal{P}$, and $\mathcal{X}_n$, 
\begin{align*}
\EE[(\hat{f}_{n}(X) - f(X))^2] &\lesssim %
n^{-\frac{2 + 2\beta}{d+2 + 2\beta} - \frac{\beta(1 - \delta)(d-s)(2 + 2\beta)}{(d + 2 + 2\beta)^2}}. 
\end{align*}}

\edit{For $\beta > 1/2$,} let $M_n \gtrsim \lambda_n$ and $\lambda_n \sim n^{\frac{1}{d+3}+ \frac{3(1-\delta)(d-s)}{4(d+3)^2}}$. Then for fixed $\delta \in (0,1)$ and all $n$ large enough, with probability at least $1 - Cn^{-\frac{3\delta}{4d + 12}}$ with respect to $\mathcal{D}_n$, $\mathcal{P}$, and $\mathcal{X}_n$,
\begin{align*}
\EE[(\hat{f}_{n}(X) - f(X))^2] \lesssim n^{-\frac{3}{d+3} - \frac{3\beta (1 - \delta)(d-s)}{(d+3)(d + 2 + 2\beta)}}.  
\end{align*}
\end{theorem}

\begin{remark}
\edit{Note that for $\beta \leq 1/2$, the result gives a rate of convergence that is faster than the minimax optimal rate of convergence $n^{-\frac{2 + 2\beta}{d+2 +2\beta}}$ for general functions on $\RR^d$ satisfying Assumption \ref{assump:f} with $s < d$. If the quadratic risk is conditioned on the input being a fixed distance $\varepsilon \in (0,1/2)$ from the boundary of the support, one could also obtain a rate of convergence above that is faster than the minimax optimal rate, i.e., %
\begin{align*}
\EE[(\hat{f}_{n}(X) - f(X))^2| X \in [\varepsilon, 1 - \varepsilon]^d] &\lesssim n^{-\frac{2 + 2\beta}{d+2 + 2\beta} - \frac{3(1 - \delta)(d-s)(2 + 2\beta)}{4(d+3)(d + 2 + 2\beta)}}. 
\end{align*}
However, we emphasize that using the estimator requires a-priori knowledge that the relevant feature subspace can be detected at inputs away from the domain boundary. 
}
\end{remark}

\section{Synthetic Data Experiments}
\label{sec:simulation}

Our synthetic data experiments aim to demonstrate how the proposed \TRIM method from Section~\ref{sec:methodology} can identify the relevant feature subspace and leverage this structure for prediction. 
\begin{figure}[hbt!]
    \centering
    \begin{subfigure}[b]{0.48\textwidth}
        \includegraphics[width=\textwidth]{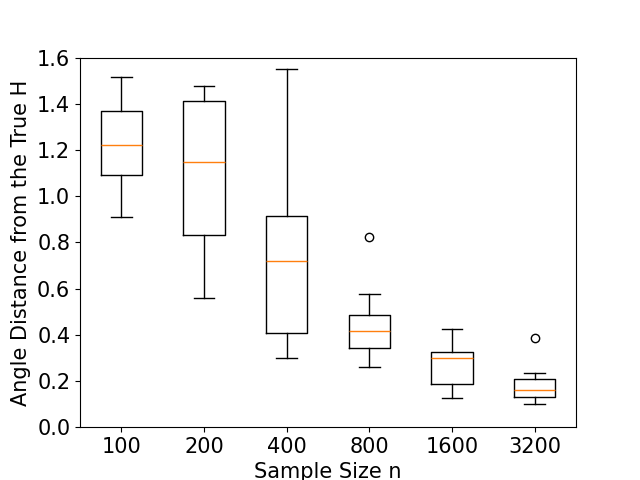}
        \caption{Scenario 1}
    \end{subfigure}
    \hfill %
    \begin{subfigure}[b]{0.48\textwidth}
        \includegraphics[width=\textwidth]{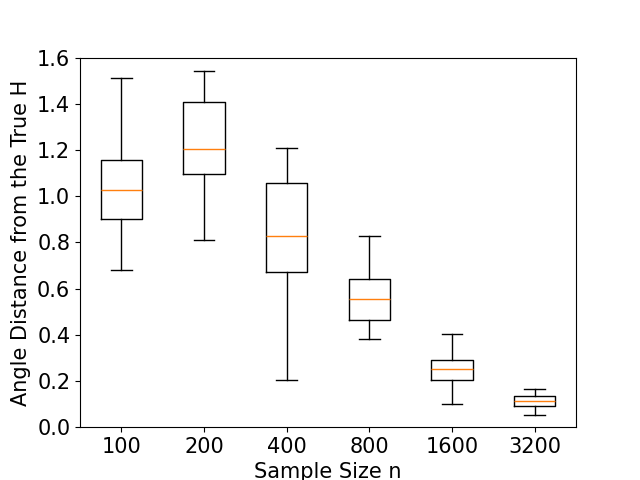}
        \caption{Scenario 2}
    \end{subfigure}
    \hfill %
    \begin{subfigure}[b]{0.48\textwidth}
        \includegraphics[width=\textwidth]{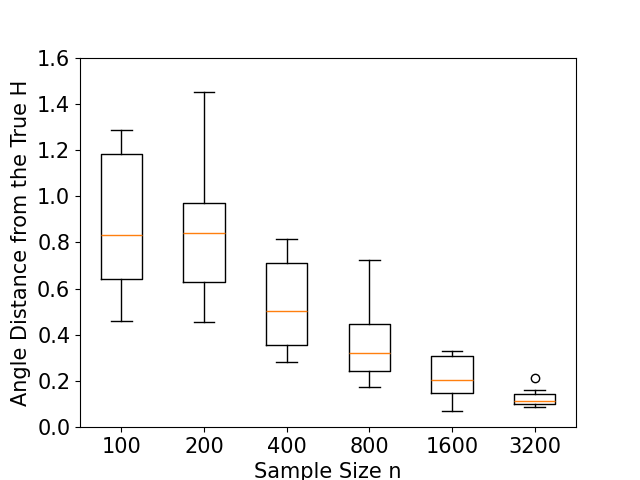}
        \caption{Scenario 3}
    \end{subfigure}
    \hfill %
    \begin{subfigure}[b]{0.48\textwidth}
        \includegraphics[width=\textwidth]{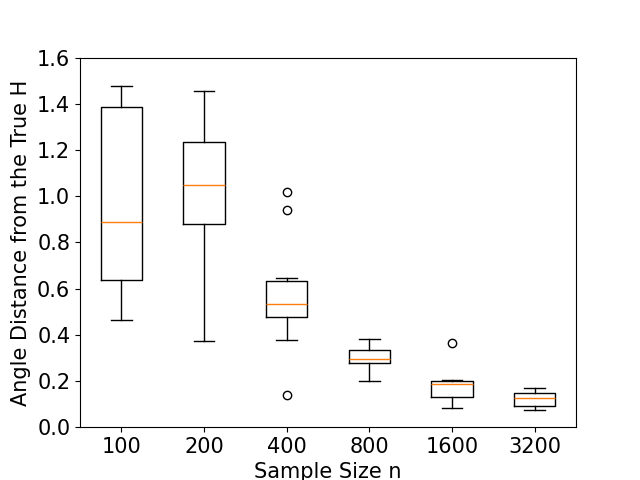}
        \caption{Scenario 4}
    \end{subfigure}
    \caption{Maximum principal angles between the subspaces spanned by the EGOP matrix $H$ and estimated EGOP matrix $\hat{H}_{n,t}$ for different scenarios. This confirms the theoretical results that $\hat{H}_{n,t}$ converges to $H$ as the training sample size increases.}
    \label{fig:dist_to_true_H}
\end{figure}
We consider four different %
\edit{multi-index models} of the form $f(x) = g(Bx)$ for some low-rank matrix $B \in \R^{s \times d}$ and $g \colon \R^{s} \rightarrow \R$ with $d = 5$ and $s = 2$. \edit{We also replicate the following analysis for $d = 50$ and $s = 2$ and arrive at similar observations; see Section \ref{sec:d large} for further details.} In particular, we consider the combinations of two transformation matrices $B$ and two functions $g$. To define the relevant subspace, we consider the following two transformations:
\begin{enumerate} \itemsep0pt
    \item The first row of $B_1$ captures dependence on the first three inputs, and the second row of $B_1$ captures the dependence of all but the third input.
    \vspace{-5pt}
    $$%
    \renewcommand{\arraystretch}{0.5}
    B_1 = \begin{bmatrix}
        1 & 1 & 1 & 0 & 0\\
        1 & 1 & 0 & 1 & 1
    \end{bmatrix}.$$
    \item $B_2$ contains the first two rows of a random matrix from the orthogonal group $O(5)$. Specifically, we consider
    \vspace{-5pt}
    $$%
    \renewcommand{\arraystretch}{0.5}
    B_2 = \begin{bmatrix}
        -0.49424072 &  0.11211344 & -0.27421644 & -0.62783889 &  0.52324025\\
        -0.0014017 &  0.71072528 &  0.69059226 & -0.11064719 &  0.07554563\\
    \end{bmatrix}.$$
\end{enumerate}
For the predictive task, we consider the two regression functions: $g_1(x) = x_1^4 + x_2^4$ and $g_2(x) = \exp(-0.25\min(x_1^2, x_2^2))$. This leads to four different scenarios: (1) $f(x) = g_1(B_1x)$, (2) $f(x) = g_2(B_1x)$, (3) $f(x) = g_1(B_2x)$ and (4) $f(x) = g_2(B_2x)$.

\begin{figure}[hbt!]
    \centering
    \begin{subfigure}[b]{0.48\textwidth}
        \includegraphics[width=\textwidth]{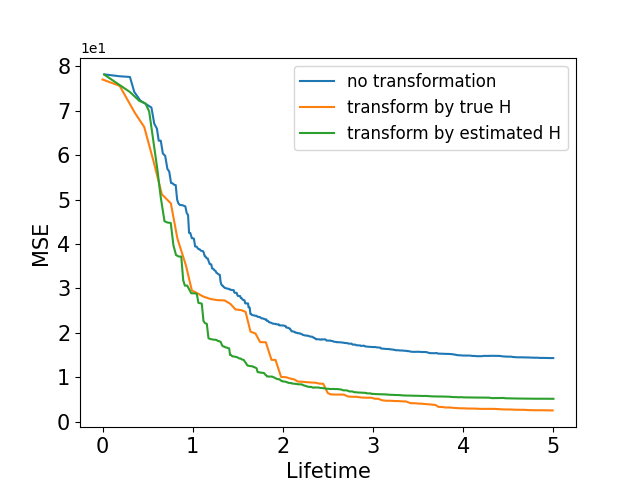}
        \caption{Scenario 1}
    \end{subfigure}
    \hfill %
    \begin{subfigure}[b]{0.48\textwidth}
        \includegraphics[width=\textwidth]{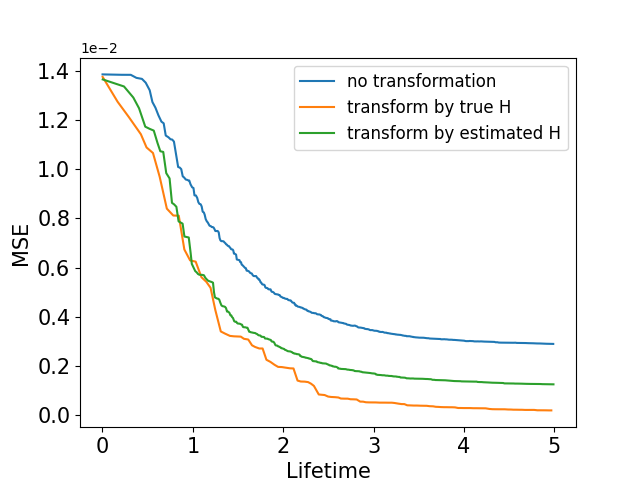}
        \caption{Scenario 2}
    \end{subfigure}
    \hfill %
    \begin{subfigure}[b]{0.48\textwidth}
        \includegraphics[width=\textwidth]{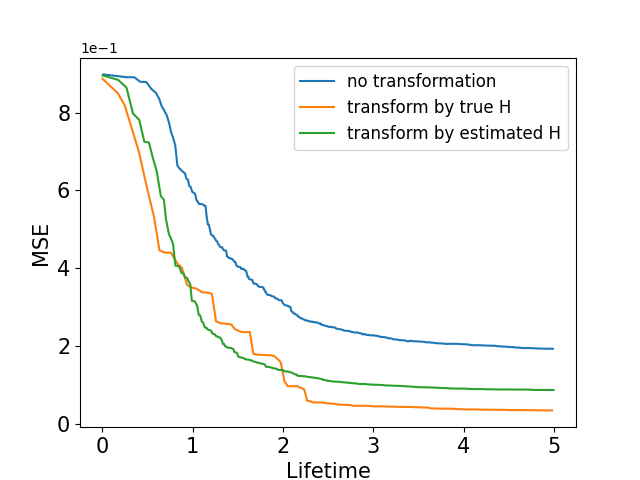}
        \caption{Scenario 3}
    \end{subfigure}
    \hfill %
    \begin{subfigure}[b]{0.48\textwidth}
        \includegraphics[width=\textwidth]{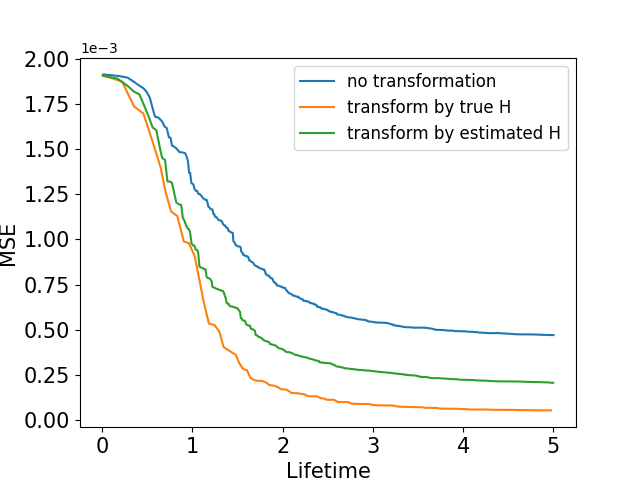}
        \caption{Scenario 4}
    \end{subfigure}
    \caption{Comparison of the test mean squared error (MSE) of three different methods: Baseline, Proposed, and Oracle.}
    \label{fig:test MSE comparison}
\end{figure}

For each scenario, we set the maximum lifetime parameter $\lambda = 5$, the number of Mondrian trees $ = 10$, the step size $t = 0.1$, the inputs $X \sim \mathrm{Uniform}([0,1]^d)$, and noise $\ee \sim \mathcal{N}(0, \sigma^2 = 0.01)$. We set the training sample size to at most 3200 and the test sample size to 3200. 

For the first experiment, we vary the training sample size $n \in \{100, 200, 400, 800, 1600, 3200\}$ and evaluate the recovery of the low-dimensional relevant feature subspace using the \TRIM  method. In particular, we compare the maximum principal angles between the subspaces spanned by the EGOP matrix $H$ and estimated EGOP matrix $\hat{H}_{n,t}$. The value of the principal angles range from $0$ to $\pi/2$ \citep{ye2016schubert}\edit{, and smaller maximum principal angles implies better estimation of the EGOP}. We refer the reader to Appendix~\ref{sec:MPA} for the definition and interpretation of the maximum principal angle. The true EGOP matrix in~\eqref{e:EGOP} is estimated based on exact evaluations of the gradient $\nabla f$, which are computed using the %
\texttt{Python} package 
\texttt{JAX} \citep{jax2018github} at the test samples. The estimated EGOP matrix is calculated by our proposed method using the training samples, as described in Algorithm \ref{a:estimate H}.

As shown in Figure \ref{fig:dist_to_true_H}, across all scenarios, the maximum principal angles between the subspaces spanned by the EGOP matrix $H$ and estimate $\hat{H}_{n,t}$ decrease as the training sample size increases. This aligns with the theoretical result that $\hat{H}_{n,t}$ converges to the true EGOP matrix $H$ as the training sample size increases; see Theorem \ref{t:Hn_approx_rate}.

For the second experiment, we use the estimated EGOP matrix $\hat{H}_{n,t}$ to transform the data and construct the Mondrian forest estimator with lifetime $\lambda$ from 0 to 5 using Algorithm~\ref{a:estimator}. For each scenario, we set the training sample size to $n = 3200$ and compare the test mean squared error (MSE) of three different methods: 
\begin{itemize} \itemsep0pt
    \item {\bf Baseline}: Mondrian forest estimator with lifetime $\lambda$ without transforming the data;
    \item {\bf Proposed}: Mondrian forest estimator with lifetime $\lambda$ after transforming the data using the estimated EGOP matrix $\hat{H}_{n,t}$;
    \item {\bf Oracle}: Mondrian forest estimator with lifetime $\lambda$ after transforming the data using the EGOP matrix $H$.
\end{itemize}
Figure~\ref{fig:test MSE comparison} shows that the test MSE of the proposed method is consistently lower than the baseline  across all scenarios. This suggests that the proposed method can effectively identify the relevant feature subspace and leverage this knowledge for predictive tasks. Moreover, there is only a small gap between the test MSE of the proposed and oracle method. 

In the next two experiments, we show that it is possible to decrease the gap with the oracle method by iterating the proposed method, i.e., we run Algorithm~\ref{a:TMFE} with $K = 2$ by using the updated Mondrian forest to a second estimate of the EGOP matrix. As the figures are similar to the previous experiment, we present them in Appendix \ref{sec:reiterate}. As shown in Figure \ref{fig:dist_to_true_H reiterate}, after one round of the iteration, the principal angles between the subspaces spanned by the EGOP matrix $H$ and the updated version of the estimate $\hat{H}_{n,t}$ are smaller than the principal angles between the subspaces spanned by $H$ and the original estimate $\hat{H}_{n,t}$ shown in Figure \ref{fig:dist_to_true_H}. This suggests that multiple iterations of the proposed method can improve the estimation of $H$ and further decrease the test MSE. In fact, as shown in Figure \ref{fig:test MSE comparison reiterate}, the test MSE after the second iteration of the proposed method is closer to the test MSE of the oracle method than the test MSE of the proposed method with a single iteration. Future work will establish rigorous convergence rates for multiple iterations of the \TRIM forest estimator.

Based on the MSE curves across different scenarios, we observe that a lifetime of 3 captures most of the performance gains, given that the curves begin to plateau after this value, while avoiding the higher computational cost associated with larger lifetimes, where only marginal improvements are observed. To systematically investigate the impact of other tuning parameters, we conducted an ablation study across the four different scenarios without reiteration (i.e., running Algorithm~\ref{a:TMFE} with $K = 0$). Our base configuration used 10 trees, lifetime of 3, step size of 0.2 and sample sizes equal to 3200 for both training and testing. For each scenario, we considered the following sets for two key parameters: (1) the number of trees in \{5, 10, 25, 50\}, and (2) the step size in \{0.025, 0.1, 0.4, 1.0\}. We evaluated the performance of the estimators using three metrics, averaged over 10 repetitions: the distance to the true $H,$ which measures accuracy of the recovered subspace, mean squared error of the predictions, and computational runtime. The mean square error is measured using the \textbf{Proposed} method above. 

Our results show clear patterns across all scenarios: increasing the number of trees from 5 to 25 yields substantial improvements in both MSE and subspace recovery (e.g., largely reducing MSE and distance to the true $H$), but further increases to 50 trees offer only marginal gains at more than triple the computational cost. For the step size, we observe that very small (0.025) and very large (1.0) values consistently perform poorly, while values between 0.1 and 0.4 achieve better results. An example is shown in Figure \ref{fig:ablation scenario 1}.
\begin{figure}[hbt!]
    \centering
    \begin{subfigure}[b]{0.95\textwidth}
        \includegraphics[width=\textwidth]{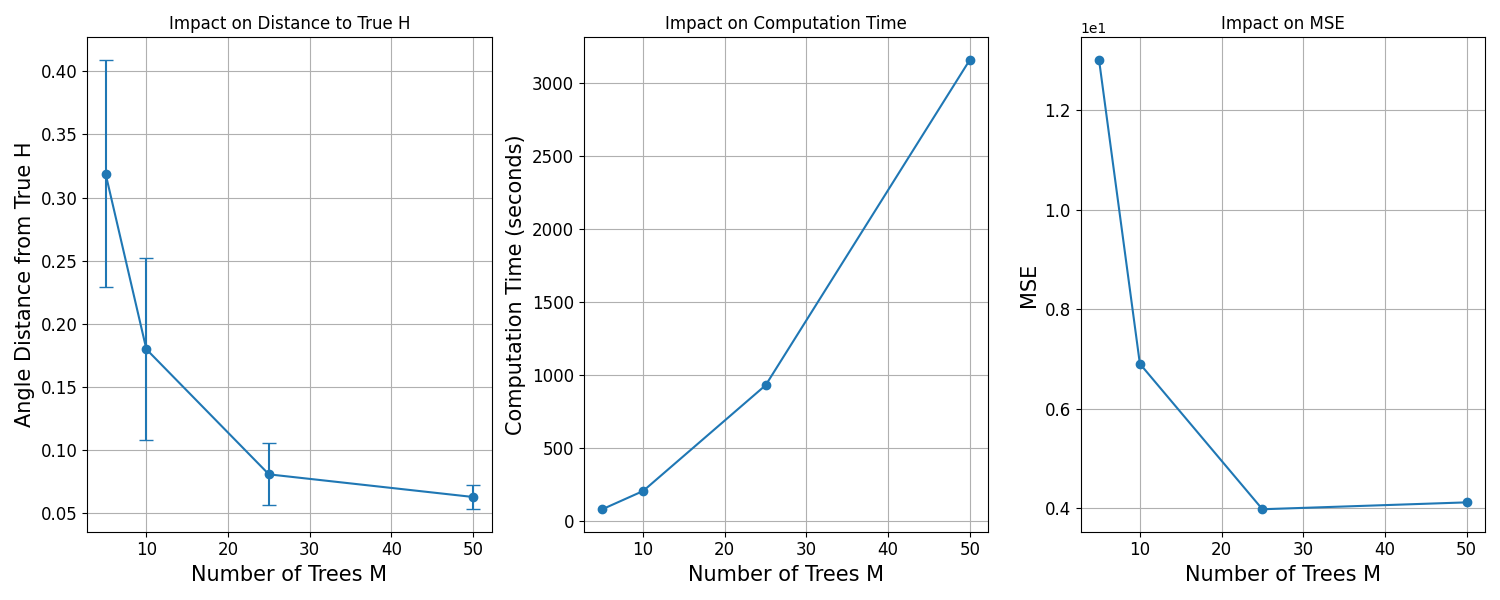}
        \caption{Ablation Study from varying the number of trees $M$}
    \end{subfigure}
    \hfill %
    \begin{subfigure}[b]{0.95\textwidth}
        \includegraphics[width=\textwidth]{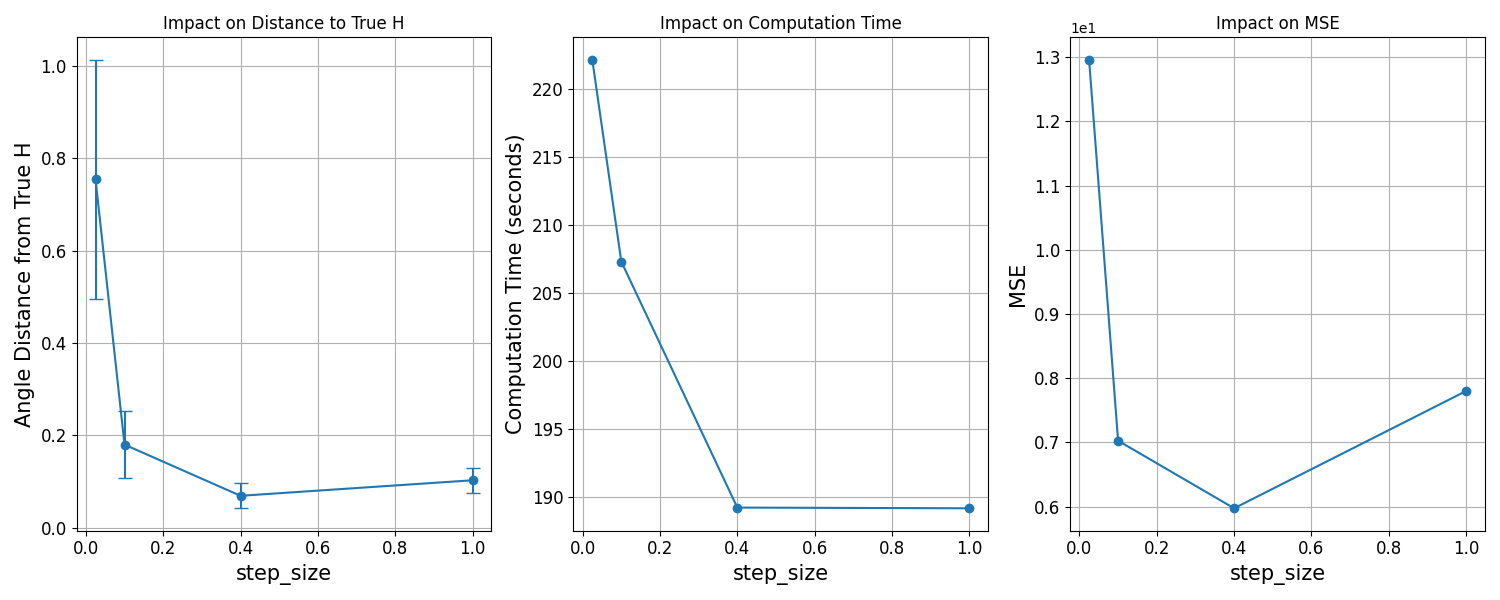}
        \caption{Ablation Study from varying the step size $t$}
    \end{subfigure}
    \caption{Ablation studies to evaluate the effect of hyperparameter choices on the performance of the estimator on Scenario 1.}
    \label{fig:ablation scenario 1}
\end{figure}
Given these findings, we recommend that practitioners conduct a small-scale cross-validation study to determine the optimal parameters for their specific problem. The key trade-offs to consider are: (1) for the number of trees, balancing the substantial accuracy gains from more trees against the linear increase in computational cost; (2) for step size, weighing the stability of smaller values against the faster convergence of larger values. An efficient approach is to first perform cross-validation on a subset of the data with a coarse grid of parameter values, followed by refining the search around promising parameter regions based on specific accuracy requirements and computational constraints.

\subsection{Weighted Mondrian Estimators}
\label{sec:weighted_mondrian_simulation}

\edit{We now consider the situation where $B = I$, the identity matrix, and $g$ is either $g_1(x)$ or $g_2(x)$ as defined above. In other words, we consider the following two scenarios: (1) $f(x) = g_1(x)$, and (2) $f(x) = g_2(x)$. In this section, we compare the performance of \TRIM with the weighted Mondrian forest estimator, as described in Section~\ref{sec:weighted_mondrian}. We vary the training sample size $n \in \{200, 400, 800, 1600\}$ and evaluate the mean squared error (MSE) of the two methods on the test set of size 1600. To ensure fair comparison, we first train an initial Mondrian forest estimator, which would be used by both methods to estimate the EGOP matrix $H$. We then use these estimates to construct the corresponding estimators.}

\begin{figure}[hbt!]
    \centering
    \begin{subfigure}[b]{0.48\textwidth}
        \includegraphics[width=\textwidth]{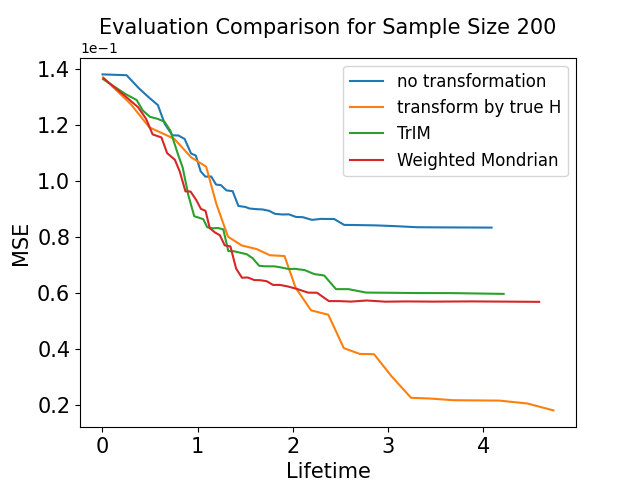}
        \caption{Scenario 1}
    \end{subfigure}
    \hfill %
    \begin{subfigure}[b]{0.48\textwidth}
        \includegraphics[width=\textwidth]{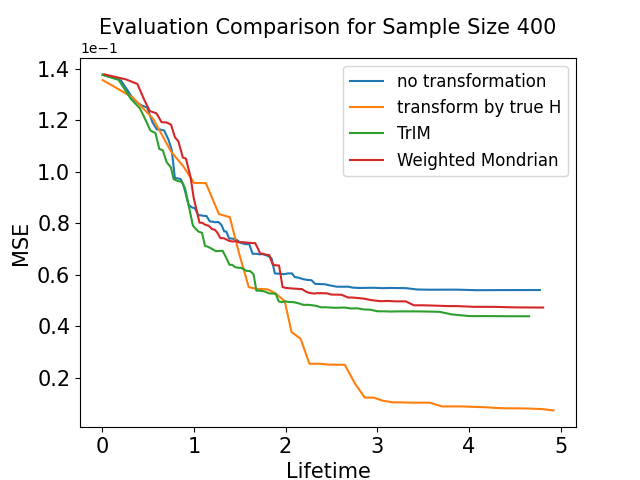}
        \caption{Scenario 2}
    \end{subfigure}
    \hfill %
    \begin{subfigure}[b]{0.48\textwidth}
        \includegraphics[width=\textwidth]{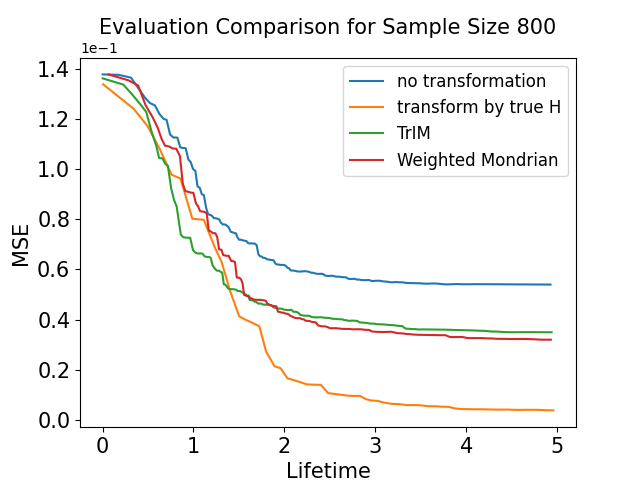}
        \caption{Scenario 3}
    \end{subfigure}
    \hfill %
    \begin{subfigure}[b]{0.48\textwidth}
        \includegraphics[width=\textwidth]{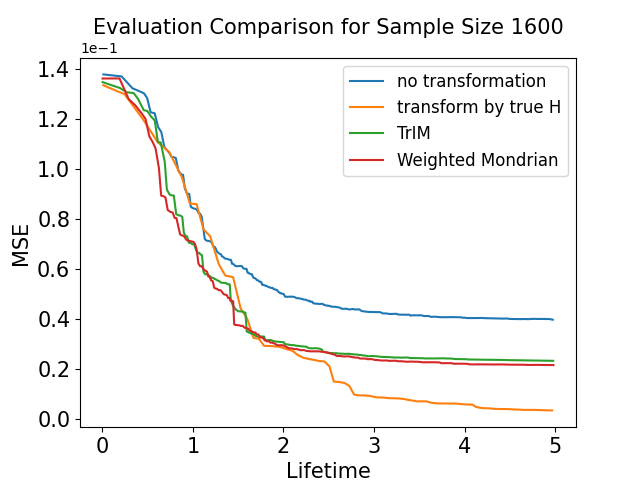}
        \caption{Scenario 4}
    \end{subfigure}
    \caption{Comparison of the test mean squared error (MSE) of three different methods: Baseline, Oracle, TrIM and Weighted Mondrian Forest for different training sample sizes. In this case, we use $g = g_1$.}
    \label{fig:aligned MSE comparison scenario 1}
\end{figure}

\begin{figure}[hbt!]
    \centering
    \begin{subfigure}[b]{0.48\textwidth}
        \includegraphics[width=\textwidth]{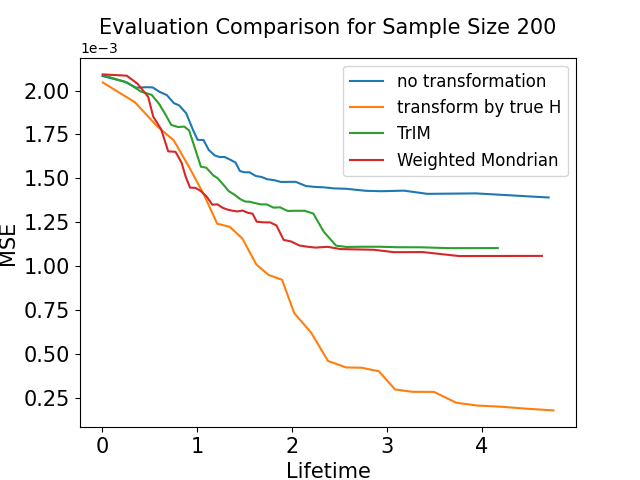}
        \caption{Scenario 1}
    \end{subfigure}
    \hfill %
    \begin{subfigure}[b]{0.48\textwidth}
        \includegraphics[width=\textwidth]{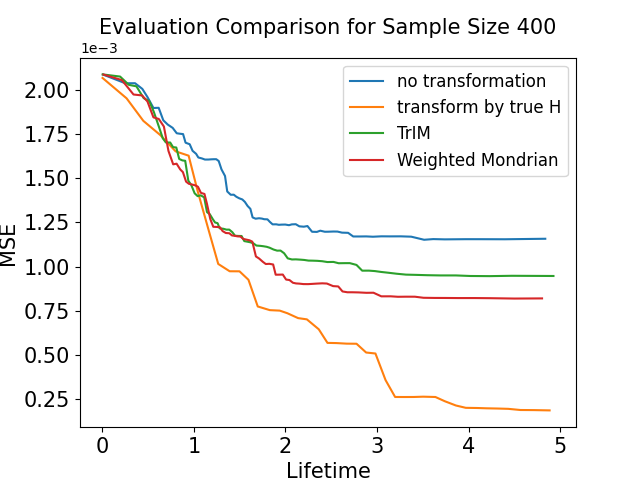}
        \caption{Scenario 2}
    \end{subfigure}
    \hfill %
    \begin{subfigure}[b]{0.48\textwidth}
        \includegraphics[width=\textwidth]{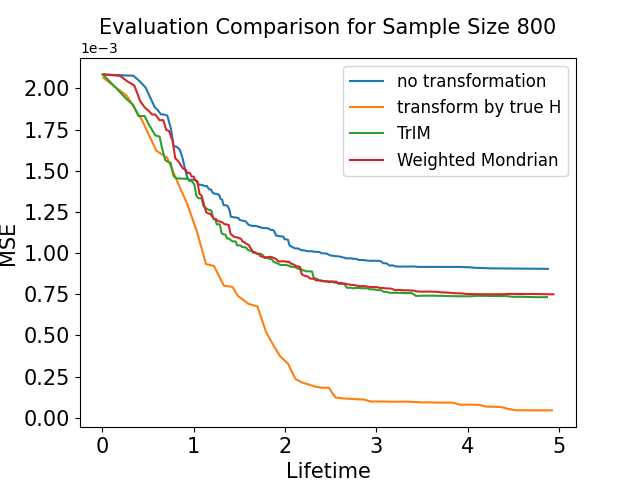}
        \caption{Scenario 3}
    \end{subfigure}
    \hfill %
    \begin{subfigure}[b]{0.48\textwidth}
        \includegraphics[width=\textwidth]{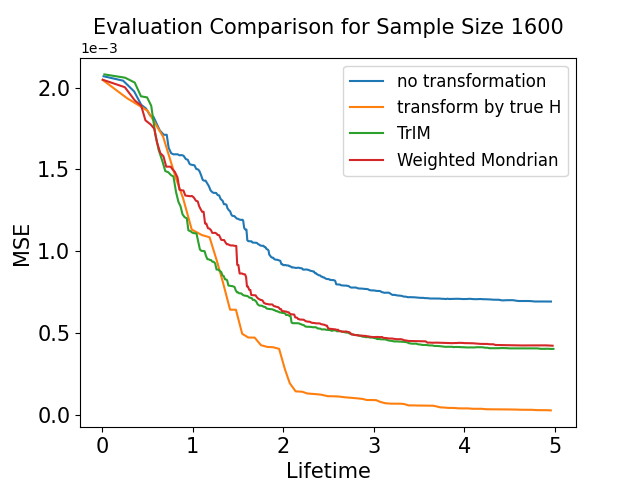}
        \caption{Scenario 4}
    \end{subfigure}
    \caption{Comparison of the test mean squared error (MSE) of three different methods: Baseline, Oracle, TrIM and Weighted Mondrian Forest for different training sample sizes. In this case, we use $g = g_2$.}
    \label{fig:aligned MSE comparison scenario 2}
\end{figure}

\edit{Figures \ref{fig:aligned MSE comparison scenario 1} and \ref{fig:aligned MSE comparison scenario 2} show that the test MSE of the weighted Mondrian forest estimators are lower than \TRIM when the training sample size is small (i.e., 200 and 400). However, as the training sample size increases to 800 and 1600, the gap between the test MSE of both methods becomes negligible. In practice, if the relevant features for a problem are known to be axis-aligned, then the weighted Mondrian forest estimator can be a good choice, especially when the training sample size is small. In other scenarios, we recommend using \TRIM, as it can handle more complex structures (i.e., relevant feature subspaces that are not axis aligned) given sufficient data.}

\subsection{Higher Dimensional Experiments}
\label{sec:d large}

We extend the previous analysis to higher dimensions by extending the matrix $B \in \R^{s \times d}$. Setting $s = 2$ and $d = 50$, we consider the following two transformations:
\begin{enumerate} \itemsep0pt
    \item The first row of $B_1$ captures dependence on the first $\lceil d/2 \rceil$ inputs, while the second row captures dependence on the first and last $\lfloor d/2 \rfloor$ inputs, i.e.
    $$
        B_1[i,j] =
        \begin{cases}
        1, & 
        \text{if } 
        \begin{cases}
        i = 1 \text{ and } j < \lceil \tfrac{d}{2} \rceil, \\
        i = 2 \text{ and } \big(j = 1 \text{ or } j \ge \lceil \tfrac{d}{2} \rceil \big),
        \end{cases} \\
        0, & \text{otherwise.}
        \end{cases}
    $$
    Or equivalently,
    \vspace{-5pt}
    $$%
    \renewcommand{\arraystretch}{0.5}
    B_1 = 
    \begin{bmatrix}
        1 & 1 & \cdots & 1 & 0 & \cdots & 0\\
        0 & 1 & 0 & \cdots & 0 & 1 & \cdots & 1
    \end{bmatrix}
    $$
    \item $B_2$ contains the first two rows of a random matrix from the orthogonal group $O(50)$. Specifically, we draw a random orthogonal matrix using \texttt{special\_ortho\_group} from the \texttt{scipy.stats} package \citep{sciPy_special_ortho_group_2025}, and take its first two rows.
\end{enumerate}

\begin{figure}[hbt!]
    \centering
    \begin{subfigure}[b]{0.48\textwidth}
        \begin{overpic}[width=\textwidth]{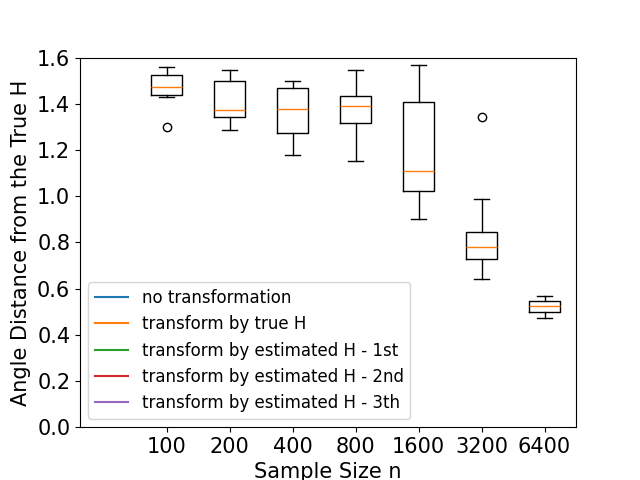}
        \put(13,8.5){\color{white}\rule{3.6cm}{2cm}} %
        \end{overpic}
        \caption{Scenario 1}
    \end{subfigure}
    \hfill %
    \begin{subfigure}[b]{0.48\textwidth}
        \includegraphics[width=\textwidth]{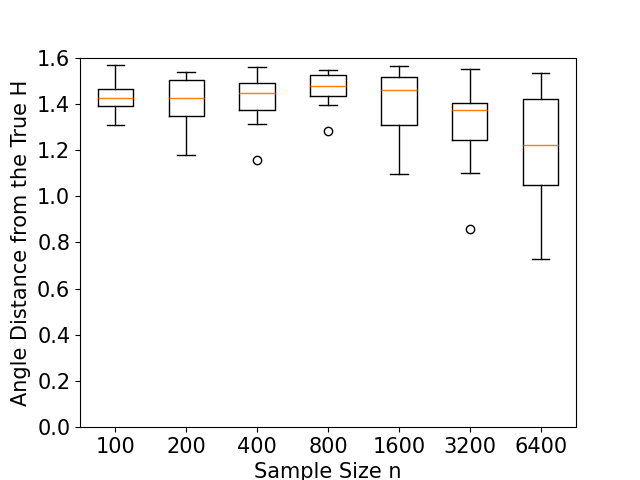}
        \caption{Scenario 2}
    \end{subfigure}
    \hfill %
    \begin{subfigure}[b]{0.48\textwidth}
        \includegraphics[width=\textwidth]{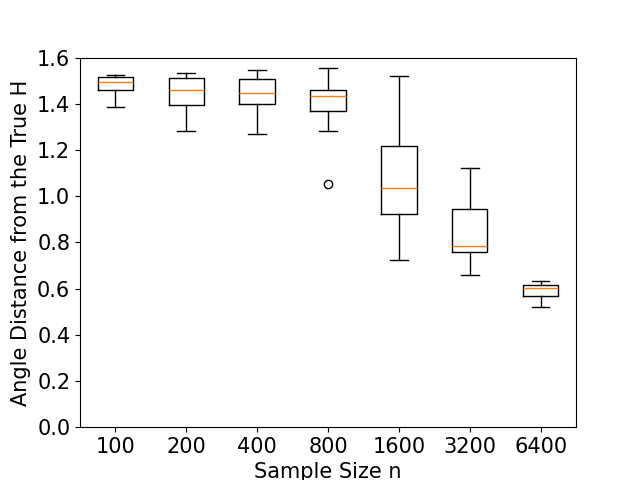}
        \caption{Scenario 3}
    \end{subfigure}
    \hfill %
    \begin{subfigure}[b]{0.48\textwidth}
        \includegraphics[width=\textwidth]{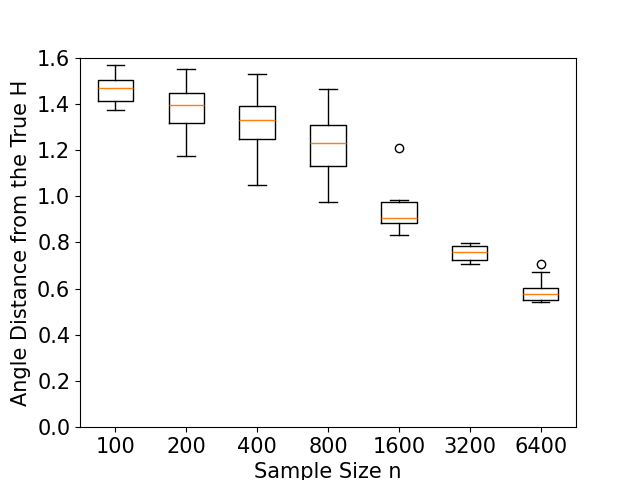}
        \caption{Scenario 4}
    \end{subfigure}
    \caption{Simulation study: Maximum principal angles between the subspaces spanned by the expected outer product matrix $H$ and estimated expected outer product matrix $\hat{H}_{n,t}$ after three round of iteration of the proposed method ($K=4$).}
    \label{fig:dist_to_true_H, large}
\end{figure}

\begin{figure}[hbt!]
    \centering
    \begin{subfigure}[b]{0.48\textwidth}
        \includegraphics[width=\textwidth]{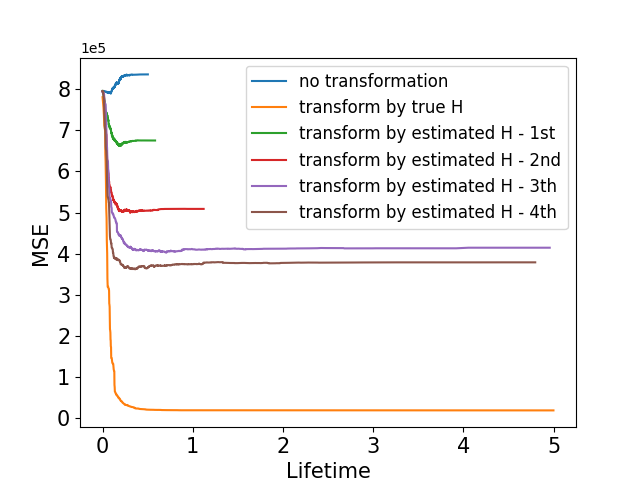}
        \caption{Scenario 1}
    \end{subfigure}
    \begin{subfigure}[b]{0.48\textwidth}
        \includegraphics[width=\textwidth]{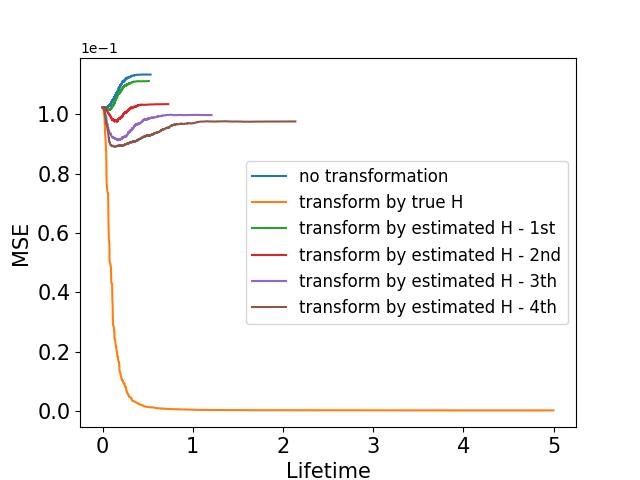}
        \caption{Scenario 2}
    \end{subfigure}
    \begin{subfigure}[b]{0.48\textwidth}
        \includegraphics[width=\textwidth]{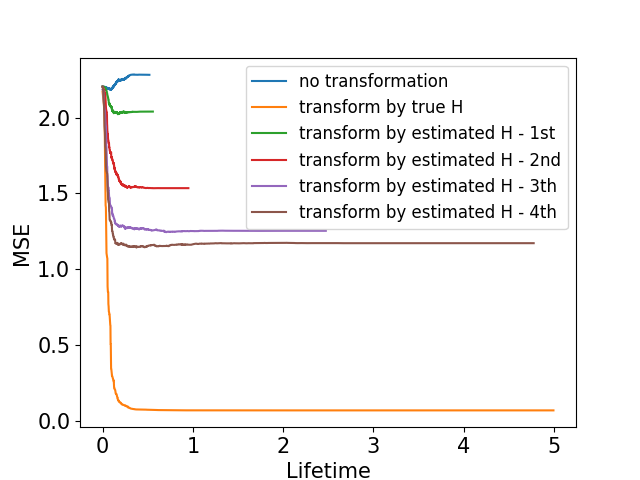}
        \caption{Scenario 3}
    \end{subfigure}
    \begin{subfigure}[b]{0.48\textwidth}
        \includegraphics[width=\textwidth]{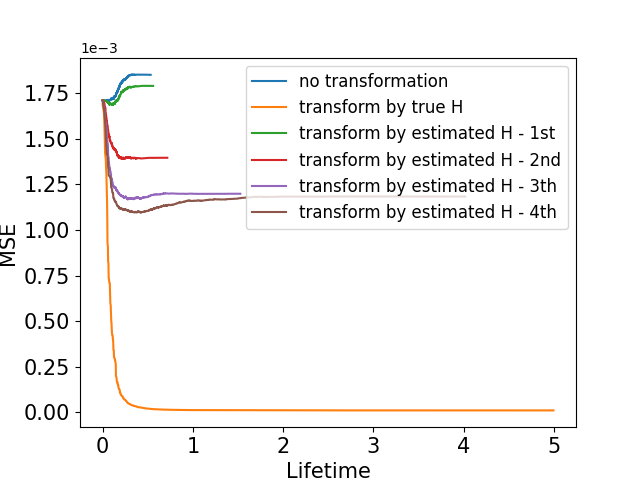}
        \caption{Scenario 4}
    \end{subfigure}
    \caption{Simulation study: Comparison of the test mean squared error (MSE) of the proposed method and the oracle method after three round of iteration of the proposed method ($K=4$).}
    \label{fig:test MSE comparison, large}
\end{figure}

For each scenario, we keep all the parameter settings the same as before, except for the step size, which we set to $t = 1.5$. Moreover, we set the training sample size to at most 6400 and the test sample size to 3200. As we work in high dimensions, we reiterate for three times and report the results below; in other words we run Algorithm~\ref{a:TMFE} with $K = 4$. 

\edit{As shown in Figures \ref{fig:dist_to_true_H, large} and \ref{fig:test MSE comparison, large}, across all scenarios, we observe a similar trend as before. Although the learning tasks becomes harder in higher dimensions, we obtain continued gains with several reiterations of the algorithm, even though the incremental gains begin to reduce as we keep reiterating. Interestingly, we observe that \TRIM struggles with Scenario 2 when compared with other scenarios. This is not obvious in the earlier lower-dimensional simulations. This may be due to the $\min$-based nonlinearity in $g_2$, which limits gradient informativeness and hampers subspace recovery, although such a problem may be much more amenable in lower dimensions. We also notice a mild ``valley" pattern, where the MSE briefly dips before rising again as lifetime elongates, especially in later iterations. Hence, in high-dimensional regimes with limited computational budgets, we note that selecting a shorter lifetime can be beneficial to prevent overfitting, as prolonged iterations may amplify noise from estimating $H$ rather than yield further improvement.}

\section{Real Data Experiments}

In this section, we apply our proposed method to data from two applications: a simulation model for the spread of Ebola in Western Africa considered in~\cite{constantine_howard_2016_asdatasets}, and multiple datasets %
from the UCI Machine Learning Repository \citep{Dua:2019}. 

First, we consider a modified SEIR model for the spread of Ebola in Western Africa \citep{constantine_howard_2016_asdatasets, diaz2018modified}. The SEIR model is a type of compartmental model used in epidemiology to simulate how diseases spread through a population. It is an extension of the basic SIR model, which stands for Susceptible, Infected, and Recovered compartments. The SEIR model adds an additional compartment, Exposed (E), to account for a latency period between when an individual is exposed to a disease and when they become infectious. For this example, we have access to a computational model and its gradients so we can compute the true expected outer product matrix in order to identify the relevant feature subspace %
following a similar setup as in Section \ref{sec:simulation}. As the results are similar to the synthetic data studies, we refer the reader to Appendix~\ref{sec:ebola_numerics} for the data generation model and the detailed results on the performance of our proposed \TRIM method.

Next, we focus on the empirical performance of our proposed method on real data applications, against other popular dimension reduction methods combined with Random Forests. We consider multiple datasets from the UCI Machine Learning Repository \citep{Dua:2019}, OpenML \citep{Vanschoren2014} and Penn Machine Learning Benchmarks \cite{Olson2017PMLB}, which represent a variety of regression tasks on real-valued covariates, see Table \ref{t:datasets}.

\begin{table}
\begin{small}
\begin{tabular}{|c|c|c|c|p{7.5cm}|}
\hline
Task & Dataset & $n$ & $d$ & Description \\
\hline
\multirow{8}{*}{\rotatebox[origin=c]{90}{Regression}}
 & \texttt{Abalone} & 4177 & 8 & Physical measurements of abalones and their shell rings. \\
 & \texttt{Auto Price} & 159 & 15 & Car characteristics and prices. \\
 & \texttt{Diabetes} & 768 & 8 & Clinical measurements from diabetes patients. \\
 & \texttt{Mu284} & 284 & 9 & Molecular energy measurements. \\
 & \texttt{Bank8FM} & 8192 & 8 & Bank customer attributes. \\
 & \texttt{Kin8nm} & 8192 & 8 & Molecular energy measurements. \\
 & \texttt{CPU Activity} & 8192 & 21 & Computer system activity used to predict CPU user-mode time. \\
 & \texttt{MAGIC} & 2000 & 63 & Multiparent advanced generation inter-cross. \\
 \hline
\multirow{4}{*}{\rotatebox[origin=c]{90}{\edit{Classification}}}
 & \texttt{Breast} & 683 & 9 & Features from fine-needle aspirate images of breast masses. \\
 & \texttt{SA-heart} & 462 & 9 & Cardiovascular risk factors from a South African male cohort. \\
 & \texttt{Forex} & 1832 & 10 & AUD/NZD candlestick data for next-day price movement. \\
 & \texttt{Vehicle} & 846 & 18 & Silhouette features for vehicle type classification. \\
\hline
\end{tabular}
\end{small}
\caption{Description of datasets used for the real data experiments.}
\label{t:datasets}
\end{table}

For these datasets, we do not have access to the true relevant feature subspace, so we compare the performance of the proposed iterative method with the baseline Mondrian forest and other state-of-the-art random forest methods. We compare \TRIM with the Mondrian forest baseline (MF) and two popular options: traditional Breiman random forests (RF) \citep{breiman2001random} and Nadaraya-Watson kernel estimators (NW Kernel) \citep{nadaraya1964estimating, watson1964smooth}. We also include two variations on the traditional random forest (SIR + RF and SAVE + RF) and the Nadaraya-Watson kernel estimator (SIR + NW Kernel and SAVE + NW Kernel). As their names suggest, these two variants first extract predictors using SIR or SAVE, and then train a random forest or a Nadaraya-Watson kernel estimator using only these features. This approach of using global dimension reduction predictors within a kernel estimator is common practice, see \cite{adragni2009sufficient}. We use the \texttt{Python} package \texttt{scikit-learn} for the traditional random forests and the code from \cite{loyal2022dimension} for the other methods.\footnote[1]{We were unable, however, to reproduce the Dimension Reduction Forest introduced in \cite{loyal2022dimension} due to an incompatibility of the \texttt{Cython} package used in their code.}

\begin{figure}[hbt!]
    \centering
    \begin{subfigure}[b]{0.48\textwidth}
        \includegraphics[width=\textwidth]{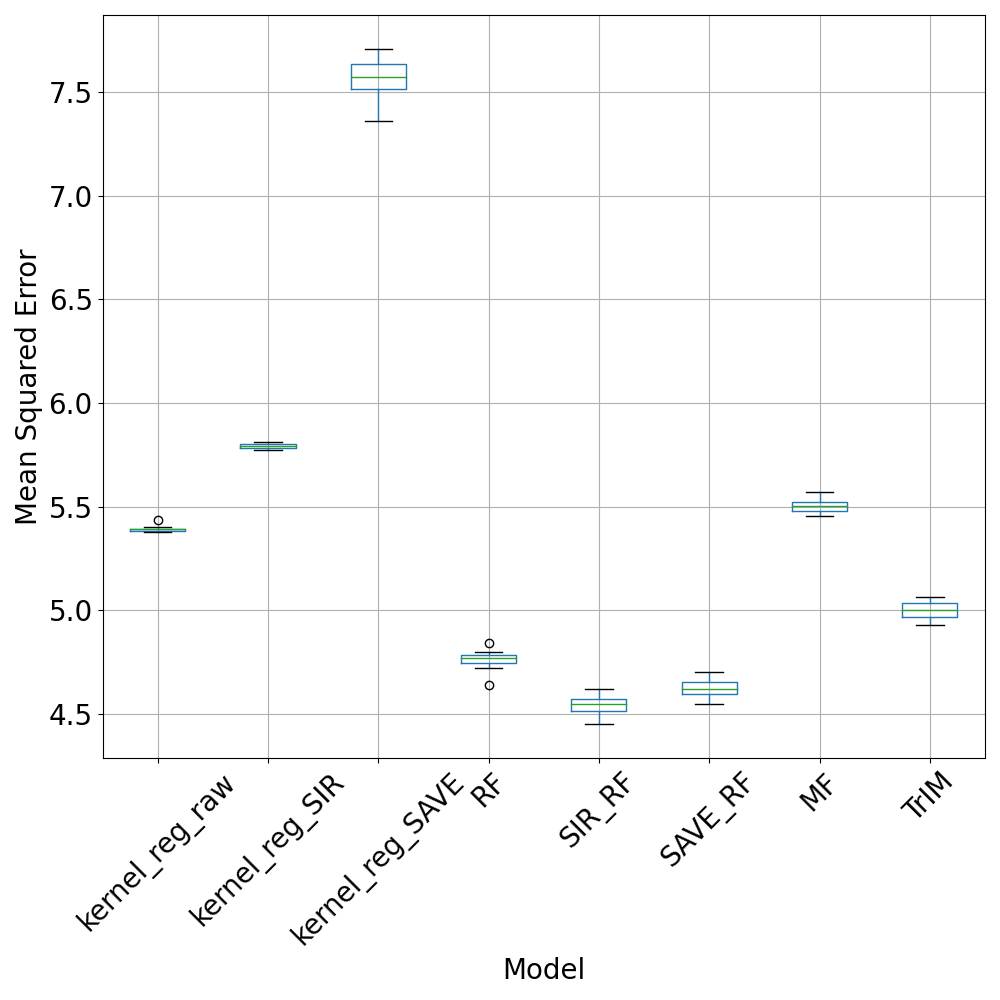}
        \caption{Abalone}
    \end{subfigure}
    \hfill
    \begin{subfigure}[b]{0.48\textwidth}
    \includegraphics[width=\textwidth]{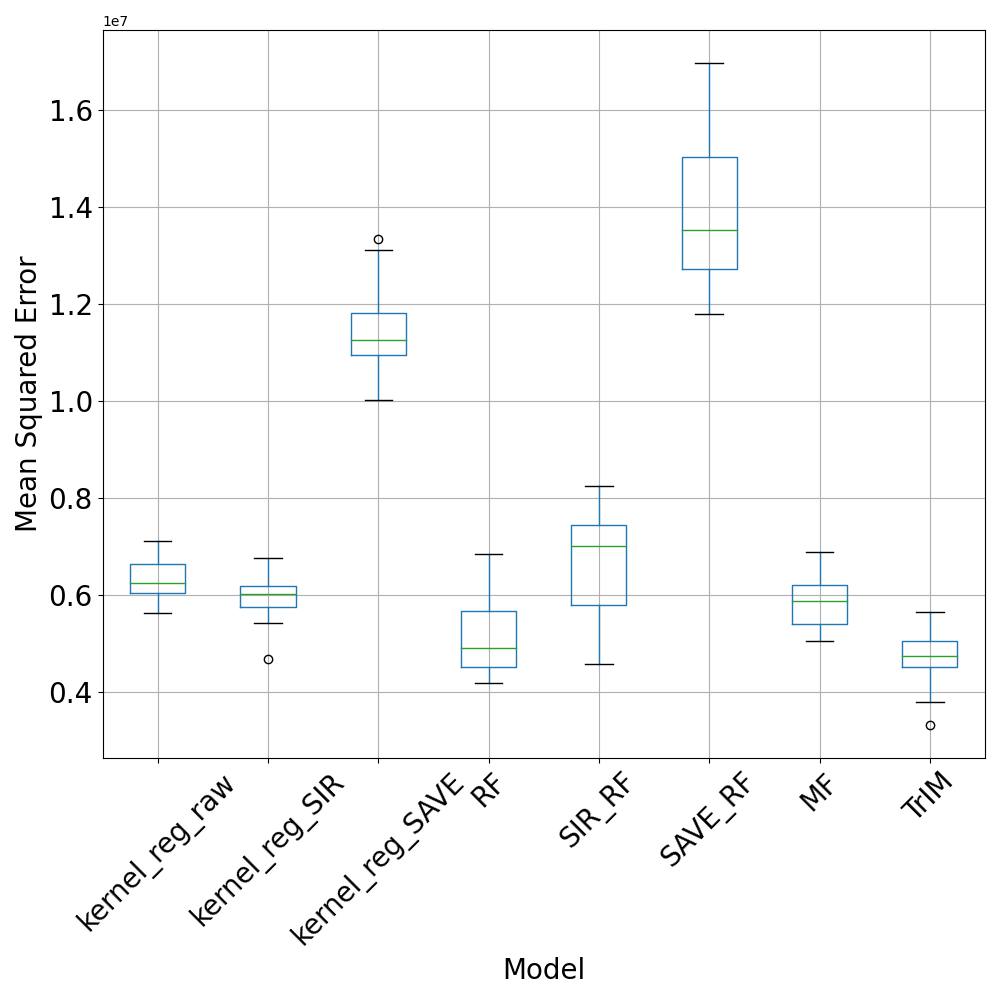}
        \caption{Auto Price}
    \end{subfigure}
    \begin{subfigure}[c]{0.48\textwidth}
    \includegraphics[width=\textwidth]{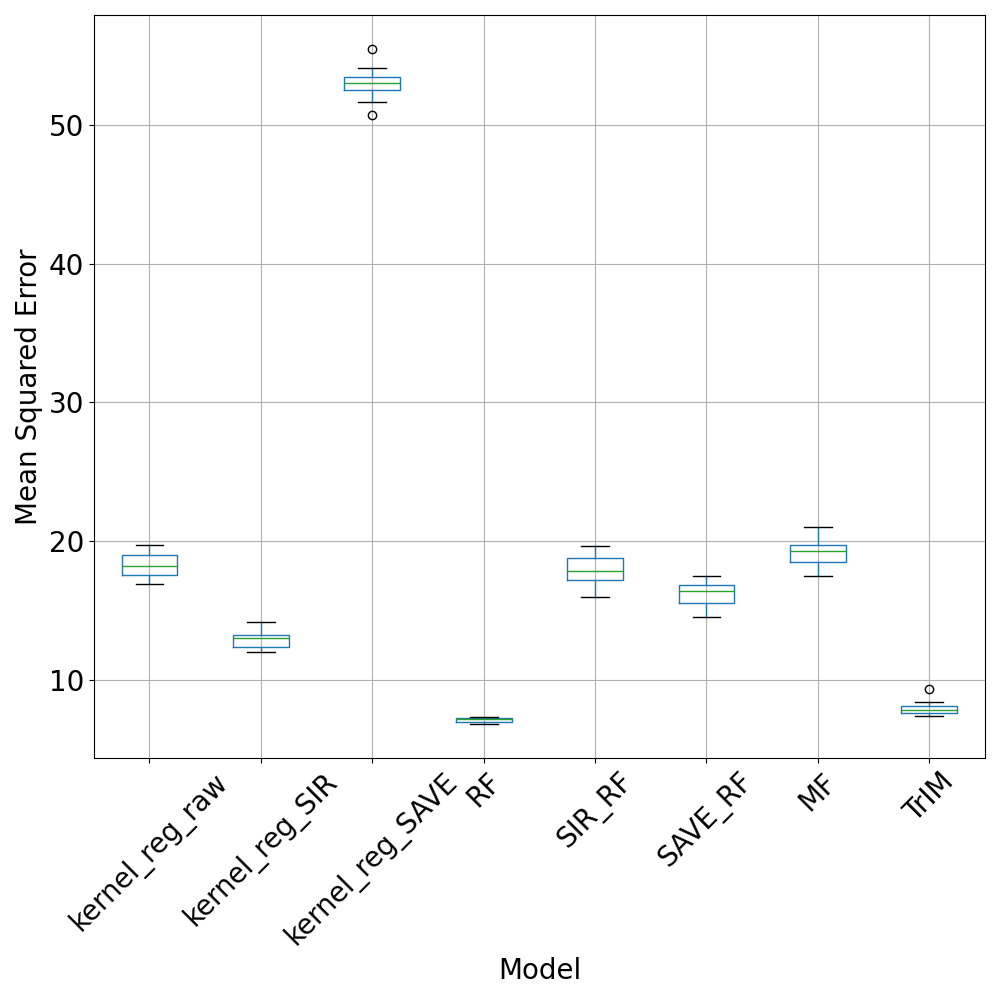}
        \caption{CPU Activity}
    \end{subfigure}
    \begin{subfigure}[d]{0.48\textwidth}
    \includegraphics[width=\textwidth]{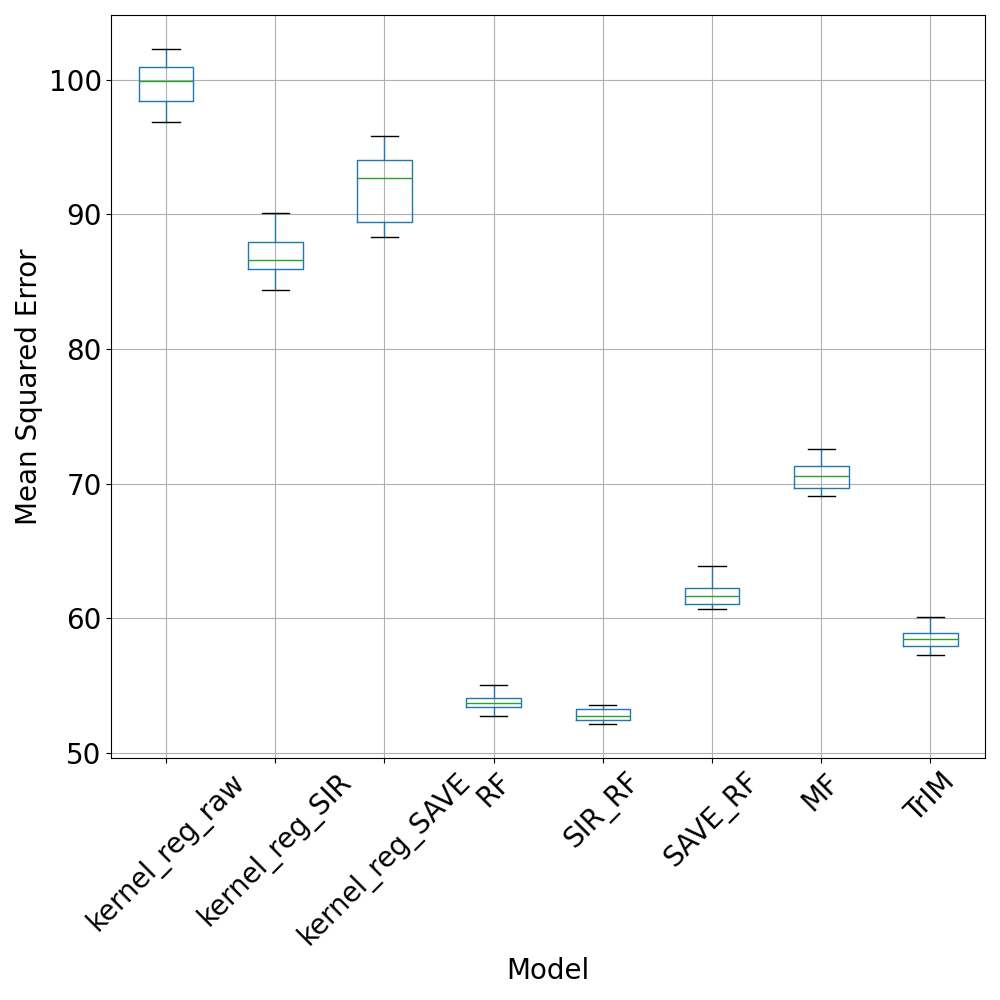}
        \caption{MAGIC}
    \end{subfigure}
    \caption{Comparison of the test mean squared error (MSE) of different methods for real data applications, where the box plot displays the variation across 15 trials. Similar patterns are observed across more datasets, which are presented in Appendix \ref{sec:real_numerics}. The unlabeled column corresponds to a simple baseline predictor that uses the sample mean of the response.}
    \label{fig:real test MSE comparison}
\end{figure}

For every dataset, we conduct a 10-fold cross-validation procedure. This involved dividing the dataset into 10 equal parts, training our models on 9 parts, and testing them on the remaining part. This process was repeated 10 times, each time with a different part used as the test set, to ensure thorough evaluation. After completing these cycles, we calculated the mean squared error (MSE) for the predictions made on the test sets across all 10 folds to assess the models' out-of-sample performance. The cross-validation procedure was repeated 15 times to construct the box-plots shown in Figure \ref{fig:real test MSE comparison}.

For all methods, we used $M=10$ trees for constructing the random forest and applied a default grid search strategy, which systematically explored various combinations of hyperparameters and choose the best possible model found among all hyperparameter choices using the training set. Then we evaluate the performance of the chosen models on the test set. For methods involving the traditional random forest, the specific hyperparameters we focused on were \texttt{min\_samples\_leaf} and \texttt{max\_features}. For \texttt{min\_samples\_leaf}, we evaluated both 1 and 5 as potential values. Meanwhile, for \texttt{max\_features}, we considered a range of options: 2, 4, 6, a third of the features (expressed as 1/3.), the square root of the number of features (\texttt{sqrt}), and the option to use all features (\texttt{None}). For both Mondrian forest methods (MF and \texttt{TrIM}), we choose the lifetime parameter $\lambda \in \{1, 2, 3, 4, 5\}$. For our proposed method (\texttt{TrIM}), we set the maximum number of iterations $T = 2$ and the step size $t \in \{0.1, 0.2, 0.5\}$. \edit{For SIR and SAVE, we keep the original number of features as the number of directions and set number of slices used when calculating the inverse regression curve to 10.}

Figure~\ref{fig:real test MSE comparison} presents the results of our real data experiment: each data point in the plot represents the average MSE across the 10 test folds in cross-validation, and each box shows the variation of the average MSE across 15 random trials of the cross-validation procedure. Our proposed method consistently outperforms the baseline Mondrian forest across all datasets. The proposed method is also competitive against the other methods. The results suggest that \TRIM can effectively identify the relevant feature subspace and leverage this knowledge to improve predictive performance of Mondrian forests on real data applications.

\section{Conclusion and Future Directions}

In this work, we propose a new iterative random forest algorithm (\texttt{TrIM}) that is competitive with state-of-the-art methods and adaptive to low-dimensional structure in the regression function captured by the linear dimension reduction model known as a multi-index model. Convergence rates for the estimator after one iteration of the algorithm are also obtained that account for the estimation error in the relevant feature subspace. %
In future work we will obtain guarantees that take into account multiple steps of the iterative approach, quantifying the improvement in the rate with multiple iterations. 

Another direction of future work is to compare different approaches for exploiting the low-dimensional structure contained in the estimated expected gradient outer product (EGOP) matrix. In this work, we use the EGOP matrix to iteratively rotate the data and construct oblique Mondrian estimators that are biased toward the relevant feature subspace. Alternatively, one can extract the leading $r$ eigenvectors of the EGOP matrix, which requires an expensive eigen-decomposition, but results in a cheaper regression step that is performed in the $r$-dimensional feature subspace defined by the eigenvectors rather than the ambient space. Future work will the study the computational and statistical trade-offs of estimating $r$ and constructing an estimator on a low-dimensional space that %
may introduce bias if $r$ is misspecified.

Other future directions include studying different variants of sufficient dimension reduction and active subspace learning techniques and incorporating them into the computationally efficient and theoretically amenable framework of Mondrian forest prediction. For example, one can consider the setting of sparse sufficient dimension reduction \citep{LisparseSDR2020}, where the relevant feature subspace is assumed to be spanned by a set of sparse vectors consisting of linear combinations of a small number of covariates, a setting that increases the interpretability of the model and is relevant in many applications. Another direction is to study nonlinear dimension reduction mechanisms. Indeed, the linear dimension reduction methods considered in this paper are limited to finding \emph{globally} relevant features. While they have been successful in several applications, the data may exhibit \emph{locally} relevant features that are not captured by linear features. Nonlinear dimension reduction has also been considered in the sufficient dimension reduction literature \citep{Wu_kernelSIR,Lee_kernelSIR} and from the gradient-based perspective \citep{zahm_nonlinear_2022} in active subspace learning, but has not yet been combined with random forest prediction.

\section*{Acknowledgments}
The authors thank the High Performance Computing Center (HPCC) at the Wharton School, University of Pennsylvania for providing computational resources that supported the experimental results presented in this paper. The authors would also like to extend their gratitude to Wenying Deng, Giles Hooker and Ngoc Mai Tran for their valuable suggestions regarding the methodology of this study. RB is grateful for support from the von K\'{a}rm\'{a}n instructorship at Caltech and a Department of Defense (DoD) Vannevar Bush Faculty Fellowship (award N00014-22-1-2790) held by Andrew M. Stuart. EO is grateful for support from NSF Grant DMS-2402234.

\setlength\bibsep{3pt}
\bibliographystyle{plainnat}
\bibliography{Biblio}

\appendix
\section{Approximation of Gradient Outer Product: Proof of Theorem \ref{t:Hn_approx_rate} and Corollary \ref{cor:An_error}}\label{app:estimate_H}

We first consider the proof of the estimation error rate of the estimator $\hat{H}_n := \hat{H}_{n,t}$ as defined in \eqref{e:H_approx} of the expected gradient outer product $H$ given by \eqref{e:EGOP}. Throughout, we will denote by $\EE_{Y}$ the expectation taken with respect to a random variable $Y$.

\subsection{Risk Bound for Shifted Input}

A core lemma to control the error of the gradient estimator is a risk bound for the Mondrian forest estimator with a shift input. This proof follows closely that of Theorem 3 in \cite{mourtada2020minimax}, %
but must take into account \edit{the shifted input. Throughout the following, we define the subset $B_t := [t, 1 - t]^d$ of the support $[0,1]^d$ that is at least $t$ distance away from the boundary.} %

\begin{lemma}\label{p:shifted_risk_bnd}
Suppose $f$ satisfies Assumption \ref{assump:f} and the distribution $\mu$ of the random input $X$ satisfies Assumption \ref{assump:mu}. Let $\hat{f}_{\lambda, n, M}$ be the Mondrian forest estimator with $M$ trees with lifetime $\lambda d$. Let $t \in (0, 1/4)$. For $s \in \{-t,t\}$ and $j \in [d]$, there exists a constant $c_{\mu}$ just depending on $\mu$ such that
\begin{align*}
    &\EE\left[(f(X  + se_j) - \hat{f}_{\lambda,n,M}(X + se_j))^2 \, \big|\,X \in B_t\right] \\
    &\qquad \leq \frac{144 L^2p_1 d}{\lambda^3 p_0(1 - 4t)^d} e^{-\lambda t} + \frac{72L^2d^3}{\lambda^4}\left(\frac{p_1C_p}{p_0^2}\right)^2 + \edit{\frac{16L^2 d^{1+\beta}}{\lambda^{2(1+\beta)}}}\left(\frac{p_1}{p_0}\right)^2 \\ &\qquad  \qquad + \frac{8dL^2}{\lambda^2 M} 
+ \frac{(10\|f\|_{\infty}^2 + 4\sigma^2)}{n}(1 + \lambda)^d (1 + t)^d.   
\end{align*}
As $n \to \infty$, assuming $\lambda, M \to \infty$ and $t \to 0$, the asymptotic behavior of the upper bound satisfies
\begin{align}\label{e:shifted_err_asymp}
\EE\left[(f(X  + se_j) - \hat{f}_{\lambda,n,M}(X + se_j))^2 \, \big|\,X \in B_t\right] \lesssim  %
\frac{\lambda^d}{n} + \edit{\frac{L^2}{\lambda^{\min\{3, 2(1+\beta)\}}}} + \frac{L^2}{\lambda^2M} . 
\end{align}
\end{lemma}

\begin{proof}
We proceed as in the proof of Theorem 3 in \cite{mourtada2020minimax} and consider the following upper bound on the pointwise risk of the forest estimator decomposed into bias and variance terms:
\begin{align}\label{e:bias-var_forests}
     \EE[(f(x) - \hat{f}_{\lambda,n, M}(x))^2] \leq  2\EE[(f(x)- \bar{f}_{\lambda,M}(x))^2] + 2\EE[(\bar{f}_{\lambda,M}(x) - \hat{f}_{\lambda,n,M}(x) )^2],
\end{align}
where we define for each $m$ and $x \in \RR^d$, 
\[\bar{f}^{(m)}_{\lambda}(x) := \EE_X[f(X)|X \in Z^{\lambda,(m)}_x],\]
and $\bar{f}_{\lambda,M}(x) := \frac{1}{M}\sum_{m=1}^M \bar{f}^{(m)}_{\lambda}(x)$.

\paragraph{Variance term.} We first consider the variance term, and condition on the Mondrian tessellation $\mathcal{P}(\lambda)$ with lifetime $\lambda$ to obtain the variance of the tree estimator corresponding to a fixed tessellation. Note that the assumption $\mathcal{D}_n$ and $\mathcal{P}(\lambda)$ are independent allows us to take the expectations separately. Also, recall that if no points of $\{X_1, \ldots, X_n\}$ fall in $Z_x^{\lambda}$, then $\hat{f}_{\lambda,n}(x) = 0$. For each $C \in \mathcal{P}(\lambda)$, let $\mathcal{N}_n(C) = \sum_{i=1}^n 1_{\{X_i \in C\}}$ be the number of covariates inside $C$ and define $p_{\lambda,C} := \PP_X(X \in C)$. Then,
\begin{align*}
    \EE_{\mathcal{D}_n}\left[(\bar{f}_{\lambda}(x) - \hat{f}_{\lambda,n}(x))^2\right] 
     &= \sum_{\substack{C \in \mathcal{P}(\lambda): \\ C \cap K \neq \emptyset}} \mathbf{1}_{\{x \in C\}} \EE_{\mathcal{D}_n}\left[\left(\EE_X[f(X)| X \in C] - \frac{\sum_{i=1}^n Y_i \mathbf{1}_{\{X_i \in C\}}}{\mathcal{N}_n(C)}\right)^2\right],
\end{align*}
where $K := [0,1]^d$ is the support of $\mu$. From the proof of Lemma 20 in \cite{OReillyTran2021minimax} we have that the expectation in the sum has the upper bound
\begin{align*}
&\EE_{\mathcal{D}_n}\left[\left(\EE_{X}[f(X)| X \in C] - \frac{\sum_{i=1}^n Y_i\mathbf{1}_{\{X_i \in C\}}}{\mathcal{N}_n(C)}\right)^2\right] \\
&\leq (2\|f\|_{\infty}^2 + \sigma^2) \sum_{k=1}^n \binom{n}{k} p_{\lambda, C}^{k} (1 - p_{\lambda, C})^{n-k} k^{-1} + \|f\|_{\infty}^2(1 - p_{\lambda,C})^n.
\end{align*}
Now, note that for $B \sim \mathrm{Binomial}(n, p_{\lambda, C})$,
\begin{align*}
 \sum_{k=1}^n \binom{n}{k} np_{\lambda, C}^{k+1} (1 - p_{\lambda, C})^{n-k} k^{-1} &= \EE[B] \EE[B^{-1}1_{\{B > 0\}}],
\end{align*}
and $\EE[B] \EE[B^{-1}1_{\{B > 0\}}] \leq \frac{2n p_{\lambda,C}}{(n+1)p_{\lambda,C}} \leq 2$ \cite[Lemma 4.1]{Gyorfi}. Also, the upper bounds $1 - x \leq e^{-x}$ and $xe^{-x} \leq e^{-1}$ for all $x \geq 0$ imply
\begin{align*}
  np_{\lambda, C}(1 - p_{\lambda,C})^n \leq  e^{-1} \leq 1. 
\end{align*}
Thus we have the pointwise upper bound
\begin{align*}
    \EE_{\mathcal{D}_n}\left[(\bar{f}_{\lambda}(x) - \hat{f}_{\lambda,n}(x))^2\right] 
    &\leq \frac{(4\|f\|_{\infty}^2 + 2\sigma^2)}{n} \sum_{\substack{C \in \mathcal{P}(\lambda): \\ C \cap K \neq \emptyset}} 1_{\{x \in C\}}  p_{\lambda, C}^{-1}  + \frac{\|f\|_{\infty}^2}{n} \sum_{\substack{C \in \mathcal{P}(\lambda): \\ C \cap K \neq \emptyset}} 1_{\{x \in C\}}p_{\lambda, C}^{-1} \\
    &\leq \frac{(5\|f\|_{\infty}^2 + 2\sigma^2)}{n} \sum_{\substack{C \in \mathcal{P}(\lambda): \\ C \cap K \neq \emptyset}} 1_{\{x \in C\}}p_{\lambda, C}^{-1}.
\end{align*}

Now, note that for $t \geq 0$ the support of $X + t e_j$ is contained in $(1 + t)K$ since $K = [0,1]^d$. Taking the expectation with respect to the shifted input gives
\begin{align*}
&\EE_{X}\left[\EE_{\mathcal{D}_n}\left[(\bar{f}_{\lambda}(X + te_j) - \hat{f}_{\lambda,n}(X + te_j))^2\right] \right] \\
    &\leq \frac{5\|f\|_{\infty}^2 + 2\sigma^2}{n} \sum_{\substack{C \in \mathcal{P}(\lambda): \\ C \cap K \neq \emptyset}} \PP(X + te_j \in C)p_{\lambda, C}^{-1} \\
    & \leq \frac{(5\|f\|_{\infty}^2 + 2\sigma^2)p_1}{np_0} \sum_{\substack{C \in \mathcal{P}(\lambda)}} \mathrm{vol}_d((1 + t)K \cap C)\mathrm{vol}_d(K \cap C)^{-1} \\
    & \leq \frac{(5\|f\|_{\infty}^2 + 2\sigma^2)p_1}{np_0} N_{\lambda}(K) (1 + t)^d.
\end{align*}
Finally, taking the expectation with respect to $\mathcal{P}(\lambda)$ and applying Proposition 2 in \cite{mourtada2020minimax}, we have
\begin{align*}
&\EE\left[(\bar{f}_{\lambda}(X + te_j) - \hat{f}_{\lambda,n}(X + te_j))^2\right] \leq \frac{(5\|f\|_{\infty}^2 + 2\sigma^2)p_1}{np_0}(1 + \lambda)^d (1 + t)^d.
\end{align*}
A similar argument gives
\begin{align*}
&\EE\left[(\bar{f}_{\lambda}(X -te_j) - \hat{f}_{\lambda,n}(X -te_j))^2\right] \leq \frac{(5\|f\|_{\infty}^2 + 2\sigma^2)p_1}{np_0}(1 + \lambda)^d (1 + t)^d.
\end{align*}
The conditional probabilities then satisfy, for $s \in \{-t,t\}$,
\begin{align*}
&\EE\left[(\bar{f}_{\lambda}(X + se_j) - \hat{f}_{\lambda,n}(X + se_j))^2 \, \big| \, X \in B_t\right] \leq \frac{(5\|f\|_{\infty}^2 + 2\sigma^2)p_1}{np_0^2(1 - 2t)^d}(1 + \lambda)^d (1 + t)^d.
\end{align*}

\paragraph{Bias term.} We now turn to bounding the bias term. For $x \in \mathrm{supp}(\mu) = [0,1]^d$, define
\[\tilde{f}_{\lambda}(x) = \EE_{\mathcal{P}}[\bar{f}_{\lambda}(x)] = \EE\left[ \frac{1}{\mu(Z_x^{\lambda})} \int_{Z^{\lambda}_x} f(y) \dint \mu(x)\right].\]
\edit{From the proof of Theorem 3 in \cite{mourtada2020minimax},
\begin{align*}
    \EE_{\mathcal{P}}\left[|f(x) - \bar{f}_{\lambda}(x)|^2 \right] \leq (f(x) - \tilde{f}_{\lambda}(x))^2 + \frac{4dL^2}{\lambda^2 M},
\end{align*}
and we can extract the pointwise upper bound %
\begin{align*}
(f(x) - \tilde{f}_{\lambda}(x))^2 &\leq \frac{36L^2}{\lambda^2}\sum_{j=1}^d e^{-\lambda\min\{x_j, 1 - x_j\}} + \frac{36L^2d^3}{\lambda^4}\left(\frac{p_1C_p}{p_0^2}\right)^2 + \frac{8L^2 d^{1+\beta}}{\lambda^{2(1+\beta)}}\left(\frac{p_1}{p_0}\right)^2. %
\end{align*}
It remains to take the conditional expectation of the first term with respect to the shifted input. By the assumption $0 < t < \frac{1}{4}$,
\begin{align*}
 \EE\left[e^{-\lambda\min\{X_j + te_j, 1 - X_j + te_j\}} \,\big|\, X \in B_t \right] 
 &\leq \frac{p_1}{p_0(1 - 2t)^d }\int_{t}^{1 - t} e^{-\lambda \min\{s + t, 1 - s + t\}} \dint s \\
 &= \frac{p_1}{p_0(1 - 2t)^d} \int_{2t}^{1} e^{-\lambda \min\{s, 1 - s\}} \dint s \\
 &= \frac{p_1}{p_0(1 - 2t)^d}\left(\int_{2t}^{1/2} e^{-\lambda s} \dint s + \int_{1/2}^{1} e^{-\lambda (1 - s)} \dint s  \right)\\
  &\leq \frac{2p_1}{\lambda p_0(1 - 2t)^d}.%
\end{align*}
A similar argument gives
\begin{align*}
 \EE\left[e^{-\lambda\min\{X_j - te_j, 1 - X_j - te_j\}} \, \big|\, X \in B_t\right] 
  &\leq \frac{2p_1}{\lambda p_0(1 - 2t)^d}. %
\end{align*}
Combining the bias and variance bounds, we have for $s \in \{-t,t\}$, %
\begin{align*}
&\EE\left[(f(X  + se_j) - \hat{f}_{\lambda,n,M}(X + se_j))^2 \, \big|\, X \in B_t \right] \\
&\qquad \leq \frac{144 L^2p_1 d}{\lambda^3 p_0(1 - 2t)^d} + \frac{72L^2d^3}{\lambda^4}\left(\frac{p_1C_p}{p_0^2}\right)^2 + \frac{16L^2 d^{1+\beta}}{\lambda^{2(1+\beta)}}\left(\frac{p_1}{p_0}\right)^2 \\ &\qquad  + \frac{8dL^2}{\lambda^2 M} 
+ \frac{(10\|f\|_{\infty}^2 + 4\sigma^2)}{n (1 - 4t)^d}(1 + \lambda)^d (1 + t)^d,    
\end{align*}
which gives the first conclusion of the Theorem. In the asymptotic regime where $\lambda, M, n \to \infty$ and $t \to 0$, we obtain \eqref{e:shifted_err_asymp}.} %
\end{proof}

\subsection{Gradient estimation error}

\edit{We next prove a bound on the estimator error of the gradient of the regression function using a Mondrian forest estimator. The approach is similar to the proof of Lemma 4 in \cite{Trivedietal2014}, where gradient estimator is obtained using a continuous kernel regressor rather than a piecewise constant Mondrian forest regressor here.}

\begin{prop}\label{l:gradvec_error}
Assume $\mu$ and $f$ satisfy Assumptions \ref{assump:mu} and \ref{assump:f}. Let $\hat{f}_n = \hat{f}_{n, \lambda, M}$ be the Mondrian forest estimator as in Lemma \ref{p:shifted_risk_bnd}. As $n \to \infty$, assuming $\lambda, M \to \infty$ and $t \to 0$,
\edit{\begin{align}\label{e:grad_err_2}
 \EE\left[\|\nabla f(X) - \hat{\nabla} \hat{f}_{n}(X)\|^2_2 \right] \lesssim %
\frac{1}{t^2}\left(\frac{\lambda^d}{n} + \frac{L^2}{\lambda^{\min\{3, 2 + 2\beta\}}} + \frac{L^2}{\lambda^2M} \right) + L^2t^{2\beta},  
\end{align}
and 
\begin{align}\label{e:grad_err_1}
\EE\left[\|\nabla f(X) - \hat{\nabla} \hat{f}_{n}(X)\|_2 \right] \lesssim
\frac{1}{t}\left(\frac{\lambda^d}{n} + \frac{L^2}{\lambda^{\min\{3, 2 + 2\beta\}}} + \frac{L^2}{\lambda^2M} \right)^{\frac{1}{2}} + Lt^{\beta},
\end{align}
where the expectations are taken with respect to $X, \mathcal{D}_n$, and $\mathcal{P}$. Then, for $\beta \leq \frac{1}{2}$, letting 
$\lambda \sim n^{\frac{1}{d + 2+2\beta}}$, $M \gtrsim n^{\frac{1}{d + 2 + 2\beta}}$ and $t \sim n^{-\frac{1}{d + 2+2\beta}}$ 
gives for $k \in \{1,2\}$,
\begin{align*}
    \EE\left[\|\hat{\nabla} \hat{f}_n(X) -  \nabla f (X) \|^k_2 \right]  \lesssim %
    n^{-\frac{k \beta}{d+2 + 2\beta}}.
\end{align*}
For $\beta > \frac{1}{2}$, letting %
$\lambda \sim n^{\frac{1}{d + 3}}$, $M \gtrsim n^{\frac{1}{d + 3}}$ and $t \sim n^{-\frac{3}{4d + 12}}$ 
gives for $k \in \{1,2\}$,
\begin{align*}
    \EE\left[\|\hat{\nabla} \hat{f}_n(X) -  \nabla f (X) \|^k_2 \right]  \lesssim %
    n^{-\frac{3k}{4d+12}}.
\end{align*}}
\end{prop}

\begin{proof}
First define the vector $\hat{\nabla} f(x) := (\triangle_{i,t}f(x))_{i \in [d]}$.%

To show \eqref{e:grad_err_2}, we first use the inequality  $\|a - b\|^2 \leq 2\|a\|^2 + 2\|b\|^2$ to obtain
\begin{align}\label{e:split_bnd}
    \mathbb{E}[\|\hat{\nabla} \hat{f}_n(X) -  \nabla f (X)\|^2_2] %
    &\leq 2\EE[\|\hat{\nabla} \hat{f}_n(X) -  \hat{\nabla} f (X)\|^2_2] +  2\| \nabla f (X) - \hat{\nabla} f(X) \|^2_2.
\end{align}

To obtain an upper bound on the second term above, we see that by the assumption on $f$, for each $i = 1, \ldots, d$ and $x \in \RR^d$, \edit{
\begin{align*}
 f(x + te_j) - f(x - te_j) &= \int_{-t}^t \partial_jf(x + se_j) \dint s \leq \int_{-t}^t \left(\partial_jf(x) + L\|se_j\|^{\beta}_2\right) \dint s \\
 &= 2t\partial_jf(x) + 2L\int_0^t s^{\beta} \dint s = 2t\partial_jf(x) + \frac{2L}{1+\beta} t^{1 + \beta}.
\end{align*}
Then, 
\begin{align}\label{e:deriv_error}
 \left|\partial_j f(x) - \frac{f(x + te_j) - f(x - t e_j)}{2t}\right|
    &\leq  \frac{L t^{\beta}}{1+\beta},
\end{align}
and thus,
\begin{align}\label{e:bnd_III}
\| \nabla f (x) - \hat{\nabla} f(x) \|^2_2
&= \sum_{j=1}^d \left(\partial_j f(x) - \frac{f(x + te_j) - f(x - t e_j)}{2t}\right)^2 \leq \sum_{j=1}^d \frac{L^2 t^{2\beta}}{(1+\beta)^2}  = \frac{L^2dt^{2\beta}}{(1 + \beta)^2}.   
\end{align}}

We now obtain a bound on the first expectation above. First, for fixed $x \in \RR^d$ and $j \in [d]$ we have
\begin{align}\label{e:quotient_error}
&\left|\frac{f(x + te_j) - f(x - te_j)- \hat{f}_{n}(x + te_j) + \hat{f}_{n}(x - te_j)}{ 2t}\right| \nonumber \\
& \qquad \leq   \frac{|f(x + t e_j) - \hat{f}_{n}(x + te_j)|}{2t}+ \frac{|f(x - te_j) - \hat{f}_{n}(x - te_j)|}{2t}.
\end{align}
Then, 
\begin{align*}
   \|\hat{\nabla} \hat{f}_n(x) -  \hat{\nabla} f (x)\|^2_2 
    &\leq \frac{1}{2t^2} \sum_{j=1}^d \left(|f(x + t e_j) - \hat{f}_{n}(x + te_j)|^2 + |f(x - te_j) - \hat{f}_{n}(x - te_j)|^2\right).%
\end{align*}
Taking the expectation with respect to $\mathcal{P}$ and $\mathcal{D}_n$ gives the upper bound
\begin{align*}
    \EE[\|\hat{\nabla} \hat{f}_n(x) -  \hat{\nabla} f (x) \|^2_2] & \leq  \frac{1}{2t^2}\sum_{j=1}^d \left(\EE[|f(x + t e_j) - \hat{f}_{n}(x + te_j)|^2]  + \EE[|f(x - te_j) - \hat{f}_{n}(x - te_j)|^2]\right).
\end{align*}
Now recall that for $x \in [0,1]^d$, if $x + se_j \notin [0,1]^d$, then 
\begin{align*}
    \EE[|f(x + se_j) - \hat{f}_n(x + se_j)|^2] = \EE[|f(x + se_j)|^2] \leq L^2(1 + s)^2.
\end{align*}
In particular, note that the above bound holds for all $x \notin B_t$. Then taking the expectation with respect to $X$, we have
\begin{align*}
 \EE[|f(X + se_j) - \hat{f}_n(X + se_j)|^2] \leq L^2(1 + s)^2 \mathbb{P}(X \notin B_t) + \EE[|f(X + se_j) - \hat{f}_n(X + se_j)|^2 | X \in B_t].   
\end{align*}
Then by Proposition \ref{p:shifted_risk_bnd}, there exists an absolute constant $c > 0$ for $s \in \{-t,t\}$,
\begin{align}\label{e:bnd_I}
    &\EE[\|\hat{\nabla} f_n(X) -  \hat{\nabla} f (X) \|^2_2 ] \leq L^2(1+s)^2(1 - (1-2t)^d) \nonumber \\ & + \frac{cd}{t^2}\bigg(\frac{L^2p_1 d }{\lambda^3 p_0(1 - 2t)^d}  + \frac{L^2d^3}{\lambda^4}\left(\frac{p_1C_p}{p_0^2}\right)^2 + \frac{L^2 d^{1+\beta}}{\lambda^{2 + 2\beta}}\left(\frac{p_1}{p_0}\right)^2 + \frac{dL^2}{\lambda^2 M} 
+ \frac{(\|f\|_{\infty}^2 + \sigma^2)}{n}(1 + \lambda)^d (1 + t)^d\bigg).
\end{align}

Combining the bounds %
\eqref{e:bnd_III} and \eqref{e:bnd_I} with \eqref{e:split_bnd} gives:
\begin{align*}
&\EE[\|\hat{\nabla} \hat{f}_n(X) -  \nabla f (X) \|^2_2] \leq \frac{L^2dt^{2\beta}}{(1 + \beta)^2} + L^2(1+s)^2(1 - (1-2t)^d) \\
&+\frac{cd}{t^{2}}\bigg(\frac{L^2p_1 d}{\lambda^3 p_0(1 - 2t)^d}  + \frac{L^2d^3}{\lambda^4}\left(\frac{p_1C_p}{p_0^2}\right)^2 + \frac{L^2 d^{1+\beta}}{\lambda^{2 + 2\beta}}\left(\frac{p_1}{p_0}\right)^2 + \frac{dL^2}{\lambda^2 M} 
+ \frac{(\|f\|_{\infty}^2 + \sigma^2)}{n}(1 + \lambda)^d (1 + t)^d\bigg).%
\end{align*}
Consider the asymptotic regime $t \to 0$ and $\lambda, n \to \infty$ then gives \eqref{e:grad_err_2}.
Minimizing the above bound with respect to $\lambda$ and $t$ gives the final asymptotic result for $k = 2$.

\edit{To show \eqref{e:grad_err_1}, we have by the the triangle inequality, for $x \in [0,1]^d$,
\begin{align}\label{e:split_bnd_1}
    \EE[\|\hat{\nabla} \hat{f}_n(x) -  \nabla f (x)\|_2] 
    &\leq \mathbb{E}[\|\hat{\nabla} \hat{f}_n(x) -  \hat{\nabla} f (x)\|_2]  + \| \nabla f (x) - \hat{\nabla} f(x) \|_2. 
\end{align}
Then, we have by \eqref{e:quotient_error},
\begin{align*}
   \EE[\|\hat{\nabla} \hat{f}_n(x) -  \hat{\nabla} f (x)\|_2] &\leq \EE[\|\hat{\nabla} \hat{f}_n(x) -  \hat{\nabla} f (x)\|_1] \\
    &\leq \frac{1}{2t} \sum_{j=1}^d \left(\EE[|f(x + t e_j) - \hat{f}_{n}(x + te_j)|] + \EE[|f(x - te_j) - \hat{f}_{n}(x - te_j)|]\right).
\end{align*}
Taking the expectation with respect to $X$ and applying Jensen's inequality, we thus have
\begin{align*}
 \EE[|f(X + se_j) - \hat{f}_n(X + se_j)|] &\leq L(1 + s) \mathbb{P}(X \notin B_t) + \EE[|f(X + se_j) - \hat{f}_n(X + se_j)| | X \in B_t] \\
 &\leq L(1 + s) \mathbb{P}(X \notin B_t) + \EE[|f(X + se_j) - \hat{f}_n(X + se_j)|^2 | X \in B_t]^{1/2}.   
\end{align*}
Combing the above bound with Proposition \ref{p:shifted_risk_bnd} and \eqref{e:bnd_III}, we have
\begin{align}\label{e:bnd_I_1}
    &\EE[\|\hat{\nabla} \hat{f}_n(X) -  \nabla f (X) \|_2 ] \leq L(1 + s)p_1(1 - (1-2t)^d) + \frac{Ldt^{\beta}}{1+\beta} \nonumber \\ &+ \frac{\sqrt{cd}}{t}\bigg(\frac{L^2p_1 d }{\lambda^3 p_0(1 - 2t)^d}  + \frac{L^2d^3}{\lambda^4}\left(\frac{p_1C_p}{p_0^2}\right)^2 + \frac{L^2 d^{2}}{\lambda^4}\left(\frac{p_1}{p_0}\right)^2 + \frac{dL^2}{\lambda^2 M} 
+ \frac{(\|f\|_{\infty}^2 + \sigma^2)}{n}(1 + \lambda)^d (1 + t)^d\bigg)^{1/2}.
\end{align}
Asymptotically as $t \to 0$ and $\lambda, n \to \infty$ the above bound implies \eqref{e:grad_err_1} and again minimizing the bound with respect to $\lambda$ and $t$ gives the final asymptotic result for $k = 1$.}

\end{proof}

\subsection{Proof of Theorem \ref{t:Hn_approx_rate} and Corollary \ref{cor:An_error}}

\edit{We next prove a bound on the EGOP estimate, closely following the approach in \cite{Trivedietal2014}, with an adaptation that uses Proposition \ref{l:gradvec_error} on the gradient estimation error using Mondrian estimators.}

\begin{proof}[Proof of Theorem \ref{t:Hn_approx_rate}]
We first observe the following bound: by the triangle inequality
\begin{align}\label{e:H_n_triangle}
\|\hat{H}_n - H\| \leq \|\hat{H}_n - H_n\| + \|H_n - H\|,
\end{align}
where
\begin{align*}
 H_n := \frac{1}{n} \sum_{i=1}^n \nabla f(x_i)\nabla f(x_i)^T.
\end{align*}
By Lemma 2 in \cite{Trivedietal2014}, which follows directly from a concentration bound for random matrices (see \cite{tropp2012user}) we have: with probability at least $1-\delta$ over the i.i.d samples $\{X_i\}_{i=1}^n$,
\begin{align*}
    \|H_n - H\| \leq \frac{6L^2}{\sqrt{n}}\left(\sqrt{\ln d} + \sqrt{\ln \frac{1}{\delta}}\right).
\end{align*}
Then, letting $t = \frac{6L^2}{\sqrt{n}}\left(\sqrt{\ln d} + \sqrt{\ln \frac{1}{\delta}}\right)$ gives
\begin{align*}
    \EE[\|H_n - H\|] &= \int_{0}^{\infty} \PP(\|H_n - H\| \geq t) \dint t \leq \int_0^{\infty} e^{-\left(\frac{\sqrt{n} y}{6L^2} - \sqrt{ \ln d}\right)^2} \dint y \\
    &\leq \int_{\RR} e^{-\frac{n}{36L^4}\left(y - 6L^2\sqrt{\frac{\ln d}{n}}\right)^2} \dint y = \frac{6\sqrt{\pi}L^2}{\sqrt{n}}.
\end{align*}

For the first term in \eqref{e:H_n_triangle}, the same argument as in Lemma 3 of \cite{Trivedietal2014} gives
\begin{align*}
&\|\hat{H}_n - H_n\| =  \| \frac{1}{n} \sum_{i=1}^n \hat{\nabla}\hat{f}_n(x_i) \hat{\nabla} \hat{f}_n(x_i)^T - \frac{1}{n} \sum_{i=1}^n \nabla f (x_i) \nabla f (x_i)^T\| \\
&\leq \frac{1}{n} \sum_{i=1}^n \|\hat{\nabla} \hat{f}_n(x_i) \hat{\nabla} \hat{f}_n(x_i)^T - \nabla f (x_i) \nabla f (x_i)^T\| \\
&\leq \frac{1}{n} \sum_{i=1}^n \|\hat{\nabla} \hat{f}_n (x_i) + \nabla f (x_i)\|_2 \| \hat{\nabla} \hat{f}_n(x_i) -  \nabla f (x_i) \|_2 \\
&\leq \frac{1}{n} \sum_{i=1}^n 2\|\nabla f (x_i)\|_2 \| \hat{\nabla} \hat{f}_n(x_i) -  \nabla f (x_i) \|_2 + \frac{1}{n} \sum_{i=1}^n \|\hat{\nabla} \hat{f}_n (x_i) - \nabla f (x_i)\|_2 \| \hat{\nabla} \hat{f}_n(x_i) -  \nabla f (x_i) \|_2 \\
&\leq \frac{2L}{n} \sum_{i=1}^n \| \hat{\nabla} \hat{f}_n(x_i) -  \nabla f (x_i) \|_2 +  \frac{1}{n} \sum_{i=1}^n \|\hat{\nabla} \hat{f}_n (x_i) - \nabla f (x_i)\|^2_2,
\end{align*}
where we used in the second inequality that for $a, b \in \RR^d$,
\begin{align*}
    \|aa^T - bb^T\| \leq \|(b + a)(a - b)^T\| = \|b + a\| \|a - b\|,
\end{align*}
and for the last inequality, we used the smoothness assumption on $f$.

Taking the expectation with respect to $\mathcal{X}_n$ gives
\begin{align*}
 \EE_{\mathcal{X}_n}[\|\hat{H}_n - H_n\|] &\leq  2L\EE_{X} [\| \hat{\nabla} \hat{f}_n(X) -  \nabla f (X) \|_2]  + \EE_{X}\left[\|\hat{\nabla} \hat{f}_n (X) - \nabla f (X)\|_2^2\right].
\end{align*}
Taking the expectation with respect to $\mathcal{D}_n$ and $\mathcal{P}$ and using Proposition \ref{l:gradvec_error} %
gives: \edit{for $\beta \leq \frac{1}{2}$,  
$\lambda \sim n^{\frac{1}{d + 2+2\beta}}$, $M \gtrsim n^{\frac{1}{d + 2 + 2\beta}}$ and $t \sim n^{-\frac{1}{d + 2+2\beta}}$
\begin{align*}
    \EE[\|\hat{H}_n - H_n\|] \lesssim  
    n^{-\frac{\beta}{d+2 + 2\beta}}.
\end{align*}
and for $\beta > \frac{1}{2}$,} $\lambda \sim n^{\frac{1}{d + 3}}$, $M \gtrsim n^{\frac{1}{d + 3}}$ and $t \sim n^{-\frac{3}{4d + 12}}$,
\begin{align*}
    \EE[\|\hat{H}_n - H_n\|] %
    &\lesssim %
    n^{-\frac{3}{4d+12}}.
\end{align*}
Thus by \eqref{e:H_n_triangle}, the conclusion follows.

\end{proof}

\begin{proof}[Proof of Corollary \ref{cor:An_error}]
By the triangle inequality and reverse triangle inequality for the $\|\cdot\|_{2,1}$ norm,
\begin{align*}
    \|\hat{A}_n - A\|_{2,1} &= d\left\|\frac{\hat{H}_n}{\|\hat{H}_n\|_{2,1}} - \frac{H}{\|H\|_{2,1}}\right\|_{2,1} = d\left\|\frac{\|H\|_{2,1}\hat{H}_n - \|\hat{H}_n\|_{2,1}H}{\|\hat{H}_n\|_{2,1}\|H\|_{2,1}}\right\|_{2,1} \\
    &\leq d\left\|\frac{\|H\|_{2,1}\hat{H}_n -\|\hat{H}_n\|_{2,1}\hat{H}_n + \|\hat{H}_n\|_{2,1}\hat{H}_n - \|\hat{H}_n\|_{2,1}H}{\|\hat{H}_n\|_{2,1}\|H\|_{2,1}}\right\|_{2,1} \\
    &\leq d\left\|\frac{\|H\|_{2,1}\hat{H}_n -\|\hat{H}_n\|_{2,1}\hat{H}_n}{\|\hat{H}_n\|_{2,1}\|H\|_{2,1}}\right\|_{2,1} + d\left\|\frac{\|\hat{H}_n\|_{2,1}\hat{H}_n - \|\hat{H}_n\|_{2,1}H}{\|\hat{H}_n\|_{2,1}\|H\|_{2,1}}\right\|_{2,1} \\
    &= \frac{d|\|H\|_{2,1} -\|\hat{H}_n\|_{2,1}|}{\|H\|_{2,1}} + \frac{d\|\hat{H}_n - H\|_{2,1}}{\|H\|_{2,1}} \leq \frac{2d\|\hat{H}_n - H\|_{2,1}}{\|H\|_{2,1}}.
\end{align*}
The result then follows from Theorem \ref{t:Hn_approx_rate} and equivalence of the column sum norm and operator norm on the matrix space $\RR^{d \times d}$.
\end{proof}
\section{Proof of Theorem \ref{thm:conv_rate}} \label{app:proof_conv_rate_f}

To directly apply results of \cite{OReilly_ObliqueMondrian}, we recall an equivalent assumption to~\eqref{e:multi-index_model} is
\begin{align}
f(x) = \tilde{g}(P_S x),    
\end{align}\label{e:multi-index_model_PS}
where $P_S \in \R^{d \times d}$ is the orthogonal projection matrix onto the subspace $S$ and $\tilde{g} \colon S \to \R$ denotes a function (different than g) that will satisfy assumption \ref{assump:f} for some constant $\tilde{L} >0$.

\begin{proof}
By Lemma \ref{l:fn_is_STIT} and Theorem 8 in \cite{OReilly_ObliqueMondrian}, conditioned on  $A_n$, we have the following upper bound:
\begin{align*}
&\EE[(\hat{f}_{n}(X) - f(X))^2] \\
&\leq %
\frac{9\tilde{L}^2d^4}{\lambda^2 \sigma_s(P_SA_n)^2}  + \frac{5\|f\|^2_{\infty} + 2\sigma^2}{n}\left(2s\sum_{k=s}^d\lambda^k \kappa_k d^{\max\{1, k-s\}}\|P_{S^{\perp}} A_n\|_{2,1}^{\max\{1,k-s\}}  +  \sum_{k=0}^s \frac{\lambda^k \kappa_k}{k!}\right),
\end{align*}
where $\kappa_k$ is the volume of the $k$-dimensional unit ball, and the expectation is taken with respect to $X$, $\mathcal{D}'_n$ and $\mathcal{P}^{'}$.
By Corollary \ref{cor:An_error} and Markov's inequality, for $\delta \in (0,1)$, \edit{we have for $\beta \leq \frac{1}{2}$,  
\begin{align*}
    \PP(\|A_n - A\|_{2,1} > n^{-\frac{\beta(1 - \delta)}{d + 2 + 2\beta}}) \leq n^{\frac{\beta(1 - \delta)}{d + 2 + 2\beta}}\EE[\|A_n - A\|_{2,1}] \lesssim n^{-\frac{\beta\delta}{d + 2 + 2\beta}}
\end{align*}
Then, for all $n$ large enough, there exist constants $C, c > 0$ such that with probability at least $1 - Cn^{-\frac{\beta\delta}{d + 2 + 2\beta}}$ with respect to $\mathcal{D}_n$, $\mathcal{P}$, and $\mathcal{X}_n$, we have that $\sigma_s(P_SA_n) \geq c > 0$, and
\[\|P_S^{\perp}A_n\|_{2,1} = \|P_S^{\perp}(A_n - A)\|_{2,1} \leq \|A_n - A\|_{2,1} \leq n^{-\frac{\beta(1 - \delta)}{d + 2 + 2\beta}}. \] 
Corollary 9 in \cite{OReilly_ObliqueMondrian} then implies that with probability at least $1 - Cn^{-\frac{\beta\delta}{d + 2 + 2\beta}}$ with respect to $\mathcal{D}_n$, $\mathcal{P}$, and $\mathcal{X}_n$, for $\lambda_n \sim n^{\frac{1}{d+2} + \frac{\beta(1 - \delta)(d-s)}{(d + 2)(d + 2 + 2\beta)}}$ and $M_n \gtrsim \lambda_n$,
\begin{align*}
\EE[(\hat{f}_{n}(X) - f(X))^2] &\lesssim %
n^{-\frac{2}{d+2} - \frac{2\beta(1 - \delta)(d-s)}{(d + 2)(d + 2 + 2\beta)}}. %
\end{align*}
Similarly, for $\beta > \frac{1}{2}$,}
\begin{align*}
    \PP(\|A_n - A\|_{2,1} > n^{-\frac{3(1-\delta)}{4d + 12}}) \leq n^{\frac{3(1 - \delta)}{4d + 12}}\EE[\|A_n - A\|_{2,1}] \lesssim n^{-\frac{3\delta}{4d + 12}},
\end{align*}
and for all $n$ large enough, there exist constants $C, c > 0$ such that with probability at least $1 - Cn^{-\frac{3\delta}{4d + 12}}$ with respect to $\mathcal{D}_n$, $\mathcal{P}$, and $\mathcal{X}_n$, we have that $\sigma_s(P_SA_n) \geq c > 0$, and
\[\|P_S^{\perp}A_n\|_{2,1} = \|P_S^{\perp}(A_n - A)\|_{2,1} \leq \|A_n - A\|_{2,1} \leq n^{-\frac{3(1-\delta)}{4d + 12}}. \] 
Corollary 9 in \cite{OReilly_ObliqueMondrian} then implies that with probability at least $1 - Cn^{-\frac{3\delta}{4d + 12}}$ with respect to $\mathcal{D}_n$, $\mathcal{P}$, and $\mathcal{X}_n$, for $\lambda_n \sim n^{\frac{1}{d+2} + \frac{3(1 - \delta)(d-s)}{4(d+3)(d+2)}}$ and $M_n \gtrsim \lambda_n$,
\begin{align*}
\EE[(\hat{f}_{n}(X) - f(X))^2] &\lesssim %
n^{-\frac{2}{d+2} - \frac{3(1 - \delta)(d-s)}{2(d+3)(d+2)}}. %
\end{align*}

\end{proof}

\section{Proofs of Theorem \ref{thm:omega_bnds} and Theorem \ref{t:conv_rate_weighted_mondrian}}\label{app:proof_conv_rate_f_weighted}

\begin{proof}[Proof of Theorem \ref{thm:omega_bnds}]
Similarly to the argument of Corollary \ref{cor:An_error}, we first see that
\begin{align*}
    \max_{i \in\{1, \ldots, d\}} \left|\frac{\omega_i}{\sum_{j=1}^d \omega_j} - \frac{\EE[|(\nabla f(X))_i|^2]}{\EE[\|\nabla f(X)\|_2^2]}\right| &\leq \sum_{i=1}^d \left|\frac{\omega_i}{\sum_{j=1}^d \omega_j} - \frac{\EE[|(\nabla f(X))_i|^2]}{\EE[\|\nabla f(X)\|_2^2]}\right| \\
    &\leq \frac{2\sum_{i=1}^d \left|\omega_i- \EE[|(\nabla f(X))_i|^2]\right|}{\EE[\|\nabla f(X)\|_2^2]}.
\end{align*}
Now noting that $\sum_{i=1}^d \omega_i = \frac{1}{n} \sum_{j=1}^n \|\hat{\nabla}\hat{f}_{n}(x_j)\|^2$, we have
\begin{align*}
    &\sum_{i=1}^d \left|\omega_i- \EE[|(\nabla f(X))_i|^2]\right| = \sum_{i=1}^d \left| \frac{1}{n} \sum_{j=1}^n |(\hat{\nabla}\hat{f}_{n}(x_j))_i|^2 - \EE[|(\nabla f(X))_i|^2]\right| \\
    &\leq \frac{1}{n} \sum_{j=1}^n\sum_{i = 1}^d \left|\left(|(\hat{\nabla}\hat{f}_{n}(x_j))_i - (\nabla f(x_j))_i| + |(\nabla f(x_j))_i|\right)^2 - \EE[|(\nabla f(X))_i|^2]\right| \\
    &= \frac{1}{n} \sum_{j=1}^n\sum_{i = 1}^d \bigg|2|(\hat{\nabla}\hat{f}_{n}(x_j))_i - (\nabla f(x_j))_i|^2 + 2|(\hat{\nabla}\hat{f}_{n}(x_j))_i - (\nabla f(x_j))_i||(\nabla f(x_j))_i| \\
    & \qquad \qquad + |(\nabla f(x_j))_i|^2 - \EE[|(\nabla f(X))_i|^2]\bigg| \\
    &\leq \frac{1}{n} \sum_{j=1}^n\|\hat{\nabla}\hat{f}_{n}(x_j) - \nabla f(x_j)\|_2^2 + \frac{2}{n} \sum_{j=1}^n\|\hat{\nabla}\hat{f}_{n}(x_j) - \nabla f(x_j)\|_2\|\nabla f(x_j)\|_2 \\
    & \qquad \qquad + \frac{1}{n} \sum_{j=1}^n\sum_{i = 1}^d\bigg||(\nabla f(x_j))_i|^2 - \EE[|(\nabla f(X))_i|^2]\bigg|. %
\end{align*}
Taking the expectation with respect to $\mathcal{X}_n$, $\mathcal{P}$, and $\mathcal{D}_n$ and applying Proposition \ref{l:gradvec_error} gives \edit{for $\beta \leq \frac{1}{2}$, letting 
$\lambda \sim n^{\frac{1}{d + 2+2\beta}}$, $M \gtrsim n^{\frac{1}{d + 2 + 2\beta}}$ and $t \sim n^{-\frac{1}{d + 2+2\beta}}$ 
gives %
\begin{align*}
    \EE\left[ \max_{i \in\{1, \ldots, d\}} \left|\frac{\omega_i}{\sum_{j=1}^d \omega_j} - \frac{\EE[|(\nabla f(X))_i|^2]}{\EE[\|\nabla f(X)\|_2^2]}\right|\right] \lesssim
    n^{-\frac{\beta}{d+2 + 2\beta}}.
\end{align*}
For $\beta > \frac{1}{2}$, letting %
$\lambda \sim n^{\frac{1}{d + 3}}$, $M \gtrsim n^{\frac{1}{d + 3}}$ and $t \sim n^{-\frac{3}{4d + 12}}$ 
gives}
\begin{align*}
     \EE\left[ \max_{i \in\{1, \ldots, d\}} \left|\frac{\omega_i}{\sum_{j=1}^d \omega_j} - \frac{\EE[|(\nabla f(X))_i|^2]}{\EE[\|\nabla f(X)\|_2^2]}\right|\right] \lesssim n^{-\frac{3}{4d+12}}. 
\end{align*}
    
\end{proof}

\begin{proof}[Proof of Theorem \ref{t:conv_rate_weighted_mondrian}]
\edit{Let $\beta \leq 1/2$. By Theorem \ref{thm:omega_bnds} and Markov's inequality, for $\delta \in (0,1)$,  
\begin{align*}
&\PP\left(\max_{i \in\{1, \ldots, d\}} \left|\frac{\omega_i}{\sum_{j=1}^d \omega_j} - \frac{\EE[|(\nabla f(X))_i|^2]}{\EE[\|\nabla f(X)\|_2^2]}\right| > n^{-\frac{\beta(1 - \delta)}{d + 2 + 2\beta}})\right) \\
    &\qquad \leq n^{\frac{\beta(1 - \delta)}{d + 2 + 2\beta}}\EE\left[\max_{i \in\{1, \ldots, d\}} \left|\frac{\omega_i}{\sum_{j=1}^d \omega_j} - \frac{\EE[|(\nabla f(X))_i|^2]}{\EE[\|\nabla f(X)\|_2^2]}\right|\right] \lesssim n^{-\frac{\beta\delta}{d + 2 + 2\beta}}.
\end{align*}
Then, for all $n$ large enough, there exist constants $C, c > 0$ such that with probability at least $1 - Cn^{-\frac{\beta\delta}{d + 2 + 2\beta}}$ with respect to $\mathcal{D}_n$, $\mathcal{P}$, and $\mathcal{X}_n$, we have that $\min_{i \in S} \frac{\omega_i}{\sum_{j=1}^d \omega_j} \geq c > 0$, and
\[\max_{i \notin S} \frac{\omega_i}{\sum_{j=1}^d \omega_j} \leq n^{-\frac{\beta(1 - \delta)}{d + 2 + 2\beta}}. \] 
Corollary 14 in \cite{OReilly_ObliqueMondrian} then implies that with probability at least $1 - Cn^{-\frac{\beta\delta}{d + 2 + 2\beta}}$ with respect to $\mathcal{D}_n$, $\mathcal{P}$, and $\mathcal{X}_n$, for $\lambda_n \sim n^{\frac{1}{d+2 + 2\beta} + \frac{\beta(1 - \delta)(d-s)}{(d + 2 + 2\beta)^2}}$ and $M_n \gtrsim \lambda_n^{2\beta}$,
\begin{align*}
\EE[(\hat{f}_{n}(X) - f(X))^2] &\lesssim %
n^{-\frac{2 + 2\beta}{d+2 + 2\beta} - \frac{\beta(1 - \delta)(d-s)(2 + 2\beta)}{(d + 2 + 2\beta)^2}}. 
\end{align*}}

\edit{Now let $\beta > 1/2$.} By Theorem \ref{thm:omega_bnds} and Markov's inequality, for $\delta \in (0,1)$,
\begin{align*}
    &\PP\left(\max_{i \in\{1, \ldots, d\}} \left|\frac{\omega_i}{\sum_{j=1}^d \omega_j} - \frac{\EE[|(\nabla f(X))_i|^2]}{\EE[\|\nabla f(X)\|_2^2]}\right| > n^{-\frac{3(1-\delta)}{4d + 12}}\right) \\
    &\qquad \leq n^{\frac{3(1 - \delta)}{4d + 12}}\EE\left[\max_{i \in\{1, \ldots, d\}} \left|\frac{\omega_i}{\sum_{j=1}^d \omega_j} - \frac{\EE[|(\nabla f(X))_i|^2]}{\EE[\|\nabla f(X)\|_2^2]}\right|\right] \lesssim n^{-\frac{3\delta}{4d + 12}}.
\end{align*}
Then, for all $n$ large enough, there exists constants $C, c > 0$ such that with probability at least $1 - Cn^{-\frac{3\delta}{4d + 12}}$ with respect to $\mathcal{D}_n$, $\mathcal{P}$, and $\mathcal{X}_n$, we have that $\min_{i \in S} \frac{\omega_i}{\sum_{j=1}^d \omega_j} \geq c > 0$, and
\[\max_{i \notin S} \frac{\omega_i}{\sum_{j=1}^d \omega_j} \leq n^{-\frac{3(1-\delta)}{4d + 12}}.\]
Finally, again by Corollary 14 in \cite{OReilly_ObliqueMondrian}, letting $M_n \gtrsim \lambda_n$ and $\lambda_n \sim n^{\frac{1}{d+3}+ \frac{3(1 - \delta)(d-s)}{4(d+3)^2}}$ gives that with probability at least $1 - Cn^{-\frac{3\delta}{4d + 12}}$ with respect to $\mathcal{D}_n$, $\mathcal{P}$, and $\mathcal{X}_n$,
\begin{align*}
\EE[(\hat{f}_{n}(X) - f(X))^2] \lesssim n^{-\frac{3}{d+3} - \frac{9(1 - \delta)(d-s)}{4(d+3)^2}}.  
\end{align*}

\end{proof}

\section{Maximum Principal Angle}
\label{sec:MPA}

In this section, we describe the distance between two subspaces that is computed as part of our numerical experiments to assess the recovery of the relevant feature subspaces. For two subspaces \([U, W]\) of dimension $k$, we compute the QR decomposition of both. That is,
\begin{align*}
    U &= Q_u R_u \\
    W &= Q_w R_w,
\end{align*}
where \( Q_u \) and \( Q_w \in \mathbb{R}^{n \times k} \) are orthonormal bases such that \( Q_u^T Q_u = Q_w^T Q_w = I_k \) that span the same subspace as the original columns of \( U \) and \( W \), and \( R_u \) and \( R_w \in \mathbb{R}^{k \times k} \) are lower triangular matrices. Next, we compute the following matrix that contains the inner products between the two collections of basis vectors
\[ D = \langle Q_u, Q_w \rangle = Q_u^T Q_w \in \mathbb{R}^{k \times k}, \]
and then apply the singular value decomposition 
\[ D = U\Sigma V^T. \]

By interpreting \( D \) as the cross-covariance matrix, its singular vectors represent the main orthogonal axes of cross-covariation between the two subspaces, while the singular values represent angles. In order to compute the principal angles of the subspaces, we define
\[ \theta \coloneqq \cos^{-1}(\Sigma) = \cos^{-1} \left( \begin{bmatrix} \sigma_1 & \sigma_2 & \dots & \sigma_k \end{bmatrix} \right) \]
The metric we use to compare subspaces is the maximum principal angle, which is simply the maximum value of \( \theta \). For more details on this distance, we refer the reader to \cite{ye2016schubert}.

\section{Additional Details on the Numerical Experiments}

\subsection{Simulation Study}
\label{sec:reiterate}

Figure \ref{fig:dist_to_true_H reiterate} shows the maximum principal angles between the subspaces spanned by the true expected outer product matrix $H$ and the estimated expected outer product matrix $\hat{H}_{n,t}$ after $K=1$ iteration of the proposed \TRIM method. Figure \ref{fig:test MSE comparison reiterate} compares the test mean squared error (MSE) of the proposed method and the oracle method that has knowledge of the true $H$ after one round of iteration for the four scenarios described in Section~\ref{sec:simulation}.

\begin{figure}[hbt!]
    \centering
    \begin{subfigure}[b]{0.48\textwidth}
        \includegraphics[width=\textwidth]{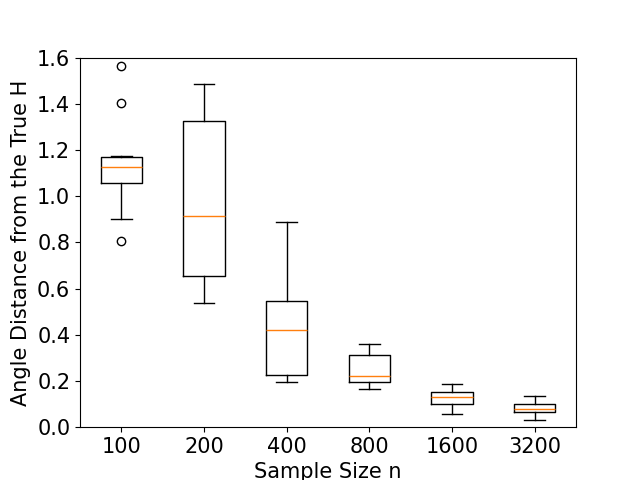}
        \caption{Scenario 1}
    \end{subfigure}
    \hfill %
    \begin{subfigure}[b]{0.48\textwidth}
        \includegraphics[width=\textwidth]{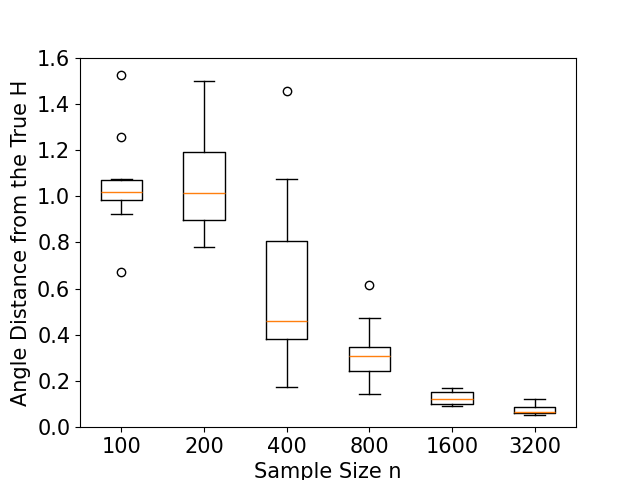}
        \caption{Scenario 2}
    \end{subfigure}
    \hfill %
    \begin{subfigure}[b]{0.48\textwidth}
        \includegraphics[width=\textwidth]{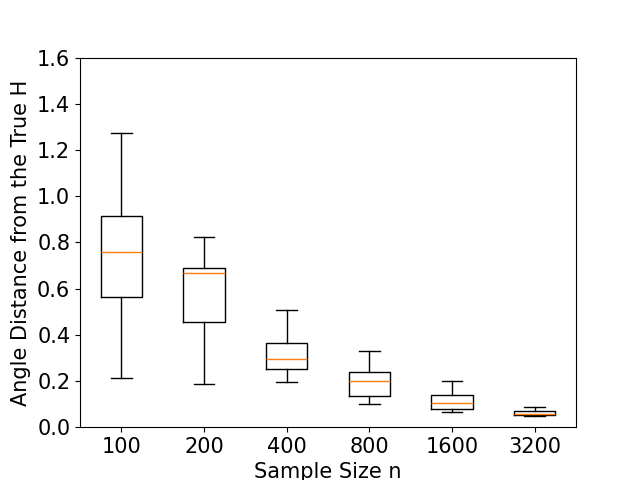}
        \caption{Scenario 3}
    \end{subfigure}
    \hfill %
    \begin{subfigure}[b]{0.48\textwidth}
        \includegraphics[width=\textwidth]{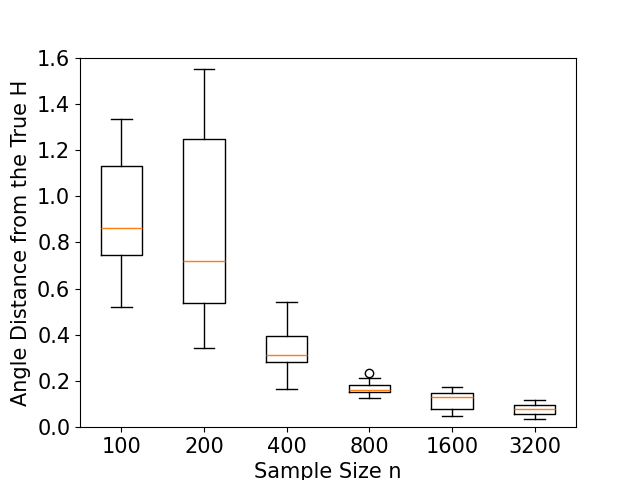}
        \caption{Scenario 4}
    \end{subfigure}
    \caption{Simulation study: Maximum principal angles between the subspaces spanned by the expected outer product matrix $H$ and estimated expected outer product matrix $\hat{H}_{n,t}$ after one round of iteration of the proposed method.}
    \label{fig:dist_to_true_H reiterate}
\end{figure}
\begin{figure}[hbt!]
    \centering
    \begin{subfigure}[b]{0.48\textwidth}
        \includegraphics[width=\textwidth]{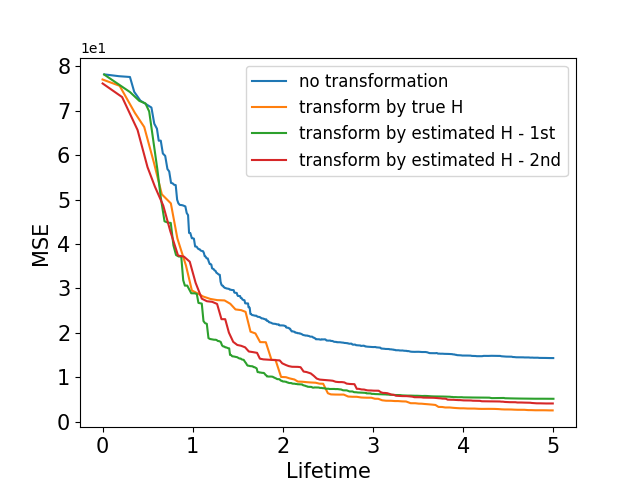}
        \caption{Scenario 1}
    \end{subfigure}
    \begin{subfigure}[b]{0.48\textwidth}
        \includegraphics[width=\textwidth]{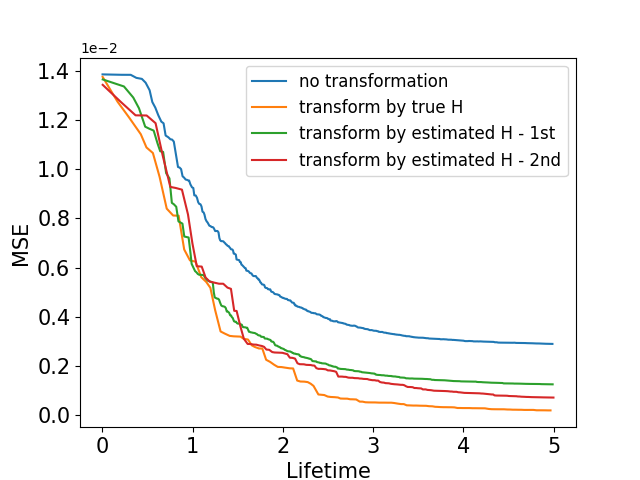}
        \caption{Scenario 2}
    \end{subfigure}
    \begin{subfigure}[b]{0.48\textwidth}
        \includegraphics[width=\textwidth]{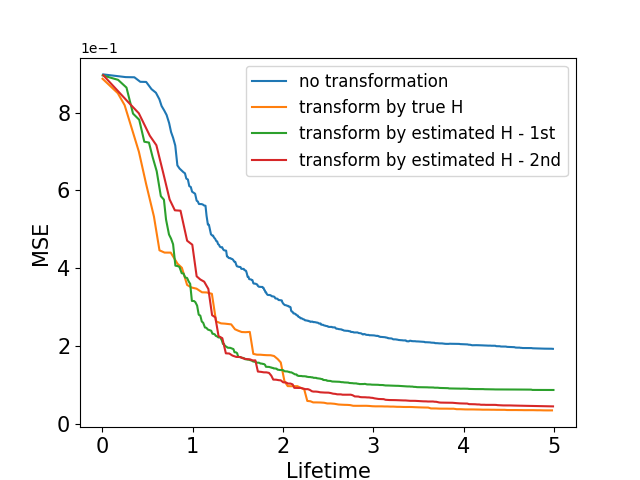}
        \caption{Scenario 3}
    \end{subfigure}
    \begin{subfigure}[b]{0.48\textwidth}
        \includegraphics[width=\textwidth]{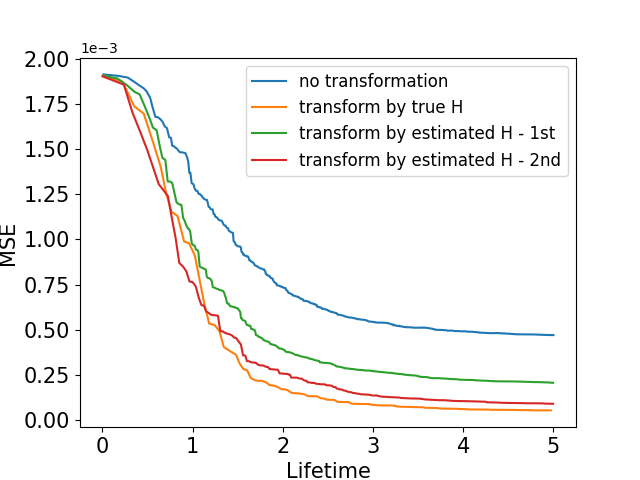}
        \caption{Scenario 4}
    \end{subfigure}
    \caption{Simulation study: Comparison of the test mean squared error (MSE) of the proposed method and the oracle method after one round of iteration of the proposed method.}
    \label{fig:test MSE comparison reiterate}
\end{figure}

Figures \ref{fig:ablation scenario 2}, \ref{fig:ablation scenario 3} and \ref{fig:ablation scenario 4} present ablation studies on the effect of two hyperparameter choices for Scenarios 2, 3 and 4, respectively. The conclusions from these studies are comparable to those discussed in Section~\ref{sec:simulation} for the first scenario of the synthetic data experiments.

\begin{figure}[hbt!]
    \centering
    \begin{subfigure}[b]{0.95\textwidth}
        \includegraphics[width=\textwidth]{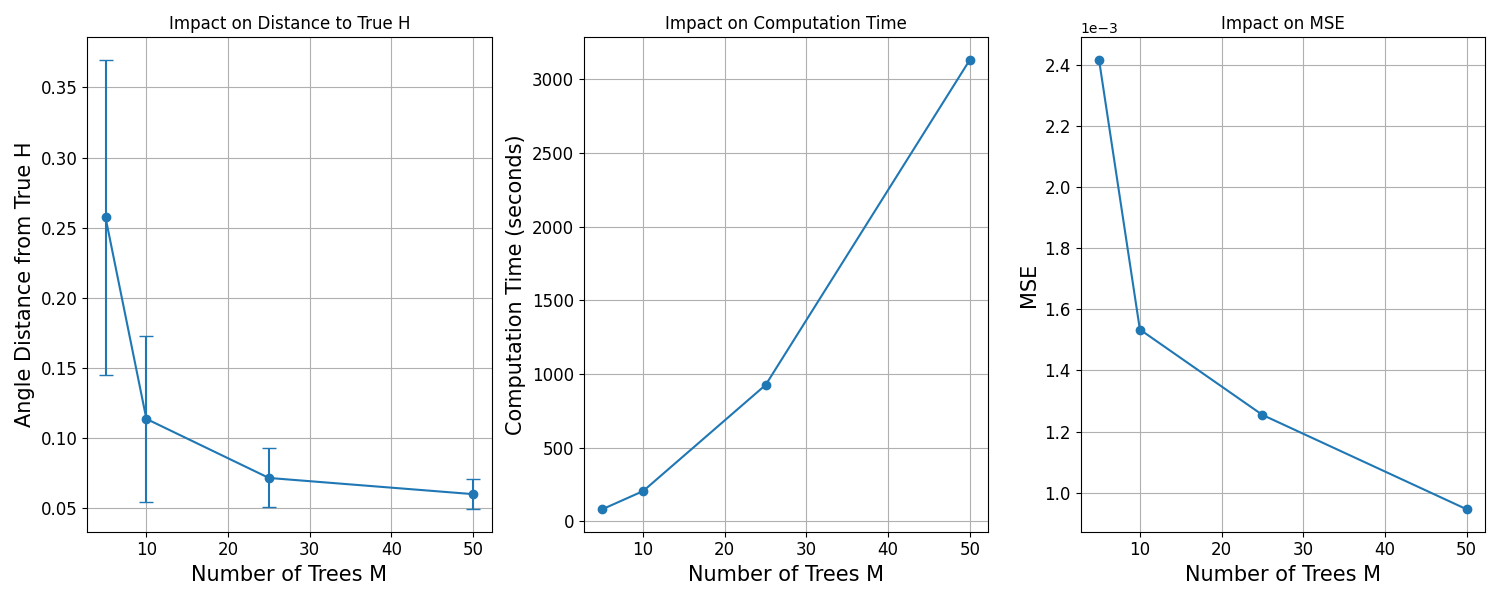}
        \caption{Ablation study on the number of trees $M$}
    \end{subfigure}
    \hfill %
    \begin{subfigure}[b]{0.95\textwidth}
        \includegraphics[width=\textwidth]{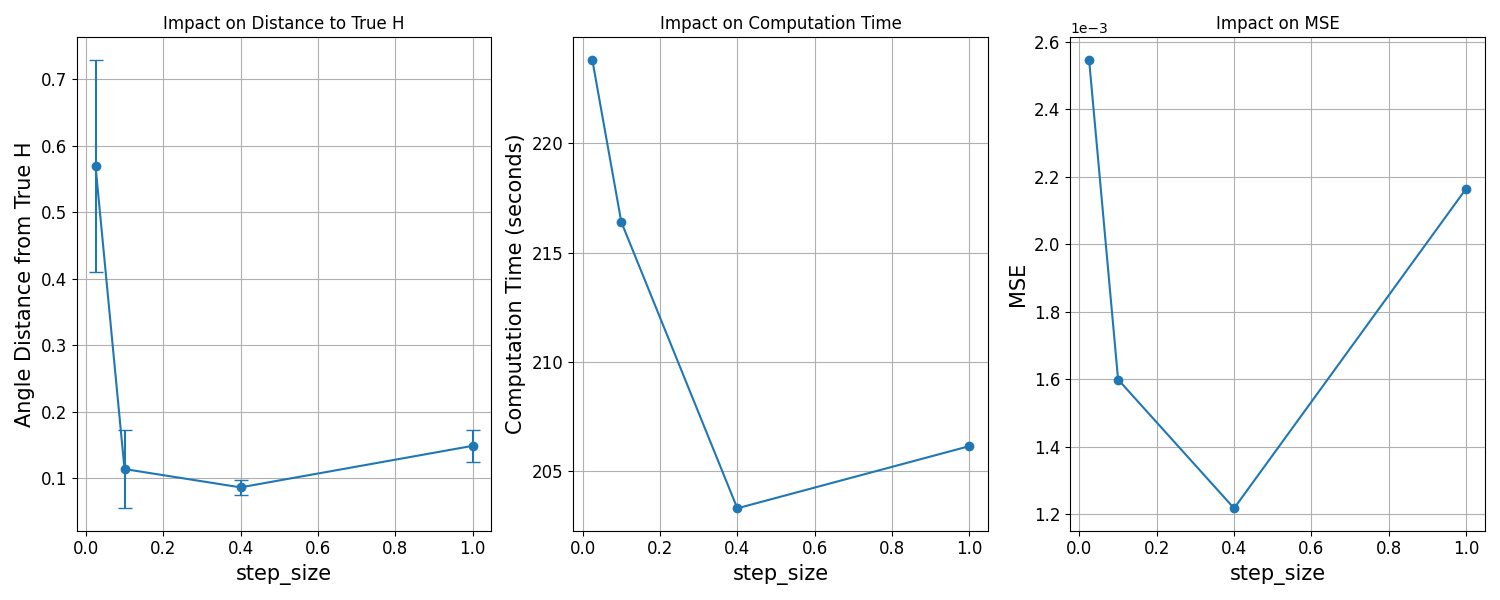}
        \caption{Ablation study on the step size $t$}
    \end{subfigure}
    \caption{Ablation studies to evaluate the effect of hyperparameter choices on the performance of the estimator for Scenario 2.}
    \label{fig:ablation scenario 2}
\end{figure}

\begin{figure}[hbt!]
    \centering
    \begin{subfigure}[b]{0.95\textwidth}
        \includegraphics[width=\textwidth]{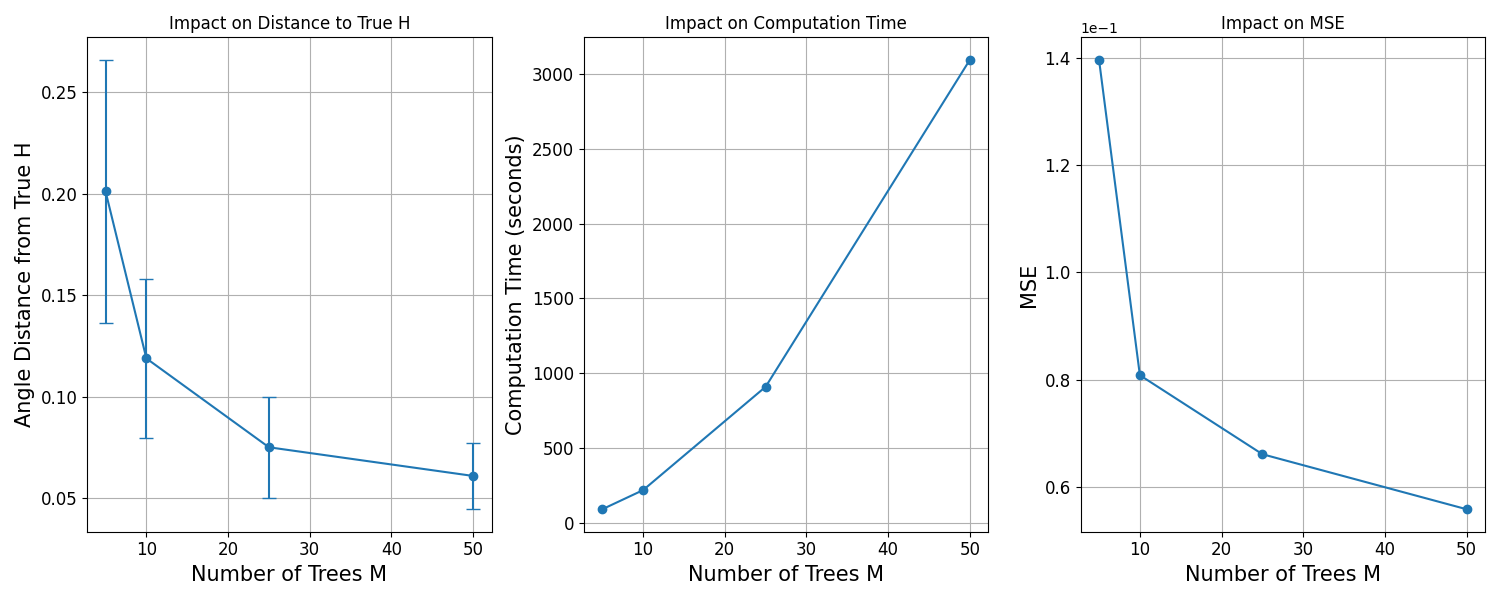}
        \caption{Ablation Study on the number of trees $M$}
    \end{subfigure}
    \hfill %
    \begin{subfigure}[b]{0.95\textwidth}
        \includegraphics[width=\textwidth]{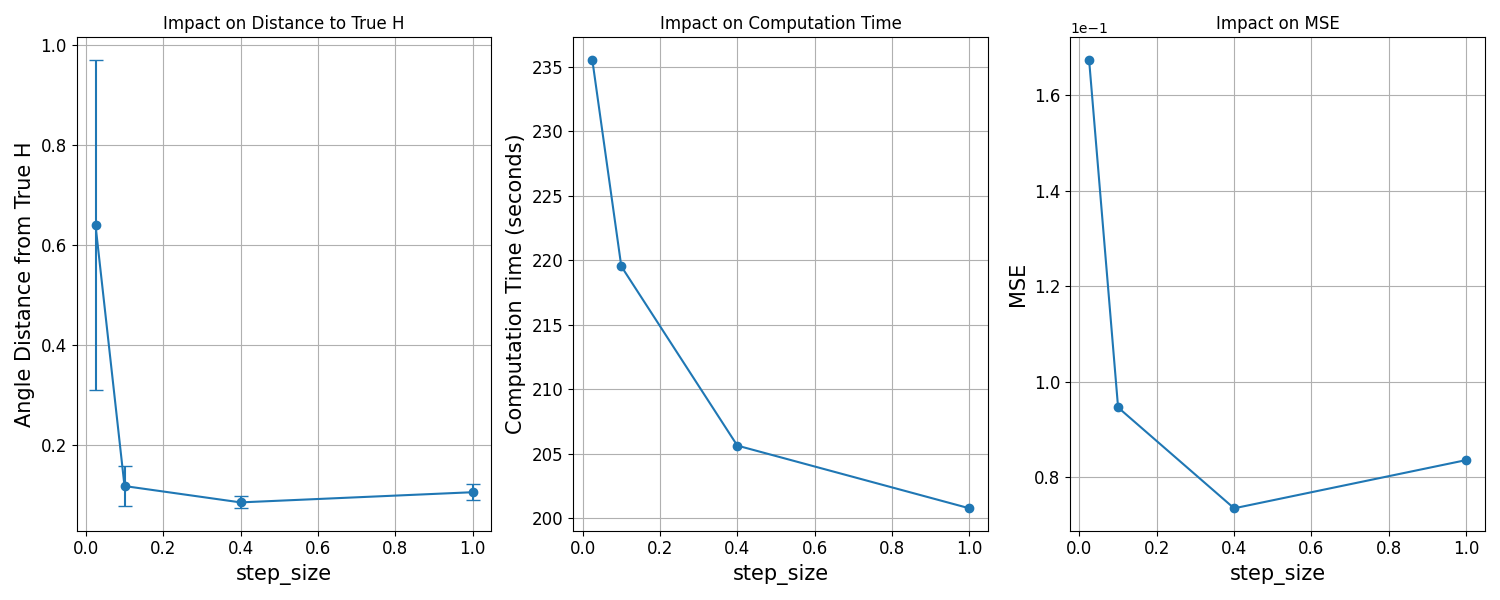}
        \caption{Ablation study on the step size $t$}
    \end{subfigure}
    \caption{Ablation studies to evaluate the effect of hyperparameter choices on the performance of the estimator for Scenario 3.}
    \label{fig:ablation scenario 3}
\end{figure}

\begin{figure}[hbt!]
    \centering
    \begin{subfigure}[b]{0.95\textwidth}
        \includegraphics[width=\textwidth]{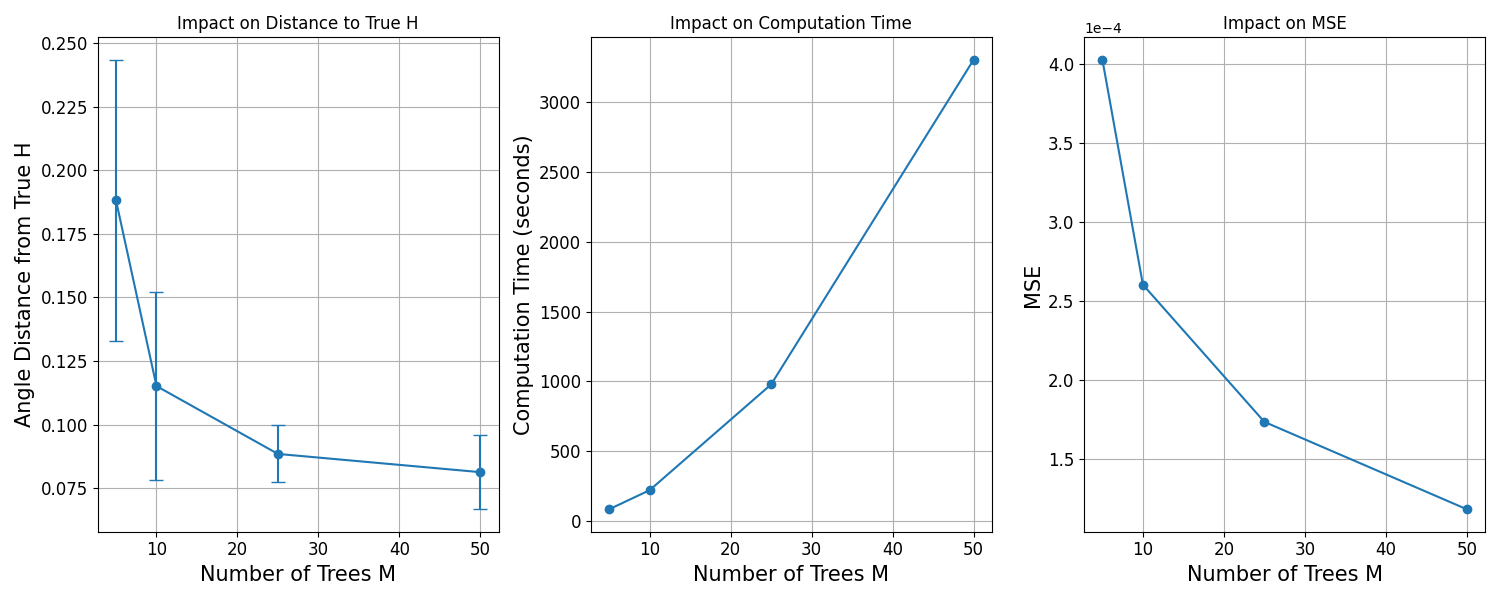}
        \caption{Ablation study on the number of trees $M$}
    \end{subfigure}
    \hfill %
    \begin{subfigure}[b]{0.95\textwidth}
        \includegraphics[width=\textwidth]{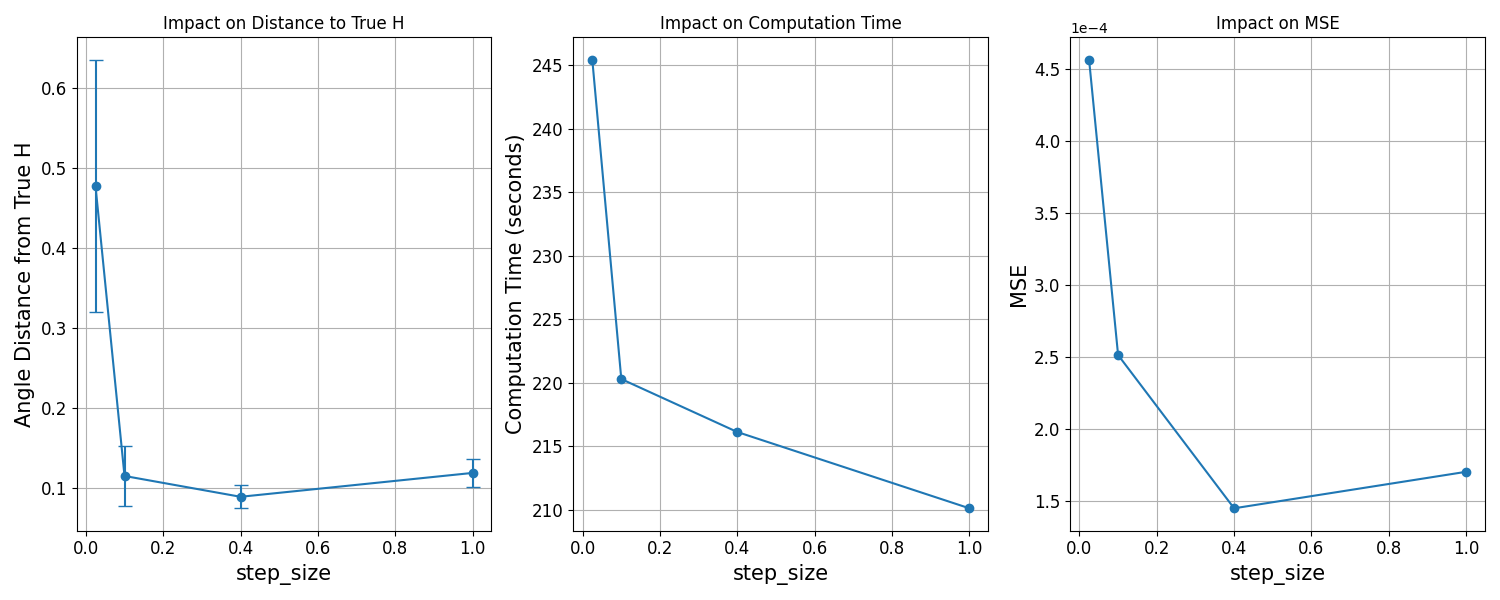}
        \caption{Ablation study on the step size $t$}
    \end{subfigure}
    \caption{Ablation studies to evaluate the effect of hyperparameter choices on the performance of the estimator for Scenario 4.}
    \label{fig:ablation scenario 4}
\end{figure}

\subsection{SEIR Model for the Spread of Ebola in Western Africa}
\label{sec:ebola_numerics}

The model we employ is a modified SEIR model for the spread of Ebola in Western Africa. The model incorporates specific assumptions to focus on key dynamics relevant to the spread of disease, following \cite{diaz2018modified, constantine_howard_2016_asdatasets}. In particular, the model makes the following three assumptions. First, stochastic effects, births and non-ebola deaths are negligible. Second, individuals who are removed (or deceased) can fall into one of three states: those who are infectious but have been improperly buried, posing a risk of further transmission; those who are non-infectious due to proper burial, effectively halting transmission from these individuals; and those who have recovered from the disease and are assumed to have immunity, thus not becoming susceptible again. Third, hospitalized individuals can still spread the disease. However, deaths in hospitals are properly buried and hospitalization increases the chances of recovery. Under these assumptions, the ordinary differential equations governing our model are
\begin{align*}
\frac{dS}{dt} &= -\beta_1 SI - \beta_2 SR_I - \beta_3 SH, \\
\frac{dE}{dt} &= \beta_1 SI + \beta_2 SR_I + \beta_3 SH - \delta E, \\
\frac{dI}{dt} &= \delta E - \gamma_1 I - \psi I, \\
\frac{dH}{dt} &= \psi I - \gamma_2 H, \\
\frac{dR_I}{dt} &= \rho_1 \gamma_1 I - \omega R_I, \\
\frac{dR_B}{dt} &= \omega R_I + \rho_2 \gamma_2 H, \\
\frac{dR_R}{dt} &= (1-\rho_1) \gamma_1 I + (1-\rho_2) \gamma_2 H,
\end{align*}
where \(S\) is the susceptible fraction of the population, \(E\) is the exposed population (infected but asymptomatic), \(I\) is the infected fraction, \(H\) is the hospitalized fraction, \(R_I\) is the infectious dead (not properly buried), \(R_B\) is the non-infectious dead (properly buried), and \(R_R\) is the recovered fraction. The basic reproduction number (a metric measuring how many new cases of disease each current case causes) is given by
\[
R_0(p) = \frac{\beta_1 + \frac{\beta_2 \rho_1 \gamma_1}{\omega} + \frac{\beta_3}{\gamma_2} \psi}{\gamma_1 + \psi}.
\]

We use uniform distributions, \(\mu = U(\text{min}, \text{max})\), for the input parameters to the model whose parameters are given in the following table.
\begin{center}
\begin{tabular}{|c|c|c|}
\hline
Parameter & Liberia & Sierra Leone \\
\hline
\(\beta_1\) & \(U(.1, .4)\) & \(U(.1, .4)\) \\
\(\beta_2\) & \(U(.1, .4)\) & \(U(.1, .4)\) \\
\(\beta_3\) & \(U(.05, .2)\) & \(U(.05, .2)\) \\
\(\rho_1\) & \(U(.41, 1)\) & \(U(.41, 1)\) \\
\(\gamma_1\) & \(U(.0276, .1702)\) & \(U(.0275, .1569)\) \\
\(\gamma_2\) & \(U(.081, .21)\) & \(U(.1236, .384)\) \\
\(\omega\) & \(U(.25, .5)\) & \(U(.25, .5)\) \\
\(\psi\) & \(U(.0833, .7)\) & \(U(.0833, .7)\) \\
\hline
\end{tabular}
\end{center}

The expected outer product matrix $H$ for this problem has the form 
$$H = \int \nabla R_0(x) \nabla R_0(x)^T \mu(x) dx,$$
where \(\nabla R_0\) is the gradient of \(R_0\) with respect to the normalized parameters, and \(\mu\) is the probability density function on the parameters. We approximate $H$ using Gauss-Legendre quadrature with 8 points in each of the 8 dimensions of the parameter space. Gradients with respect to normalized parameters are computed according to the chain rule: if $p$ is an un-normalized parameter and $x$ is its normalized version, then $\frac{\partial R_0}{\partial x} = \frac{\partial R_0}{\partial p}\frac{\partial p}{\partial x}$, and $p = l + \frac{u-l}{2}(x+1)\Rightarrow\frac{\partial p}{\partial x} = \frac{u - l}{2}$, where $u$ and $l$ are the upper and lower bounds on the parameter. The resulting expected outer product matrix corresponds to the one identified using the active subspace procedures described in~\cite{constantine2014active}.

Figures~\ref{fig:ebola dist_to_true_H} and Figure~\ref{fig:ebola test MSE comparison} show that our proposed method performs well with this model. The maximum principal angles between the subspaces spanned by the true expected outer product matrix $H$ and estimated expected outer product matrix $\hat{H}_{n,t}$ decrease as the training sample size $n$ increases. The test MSE of the proposed method is consistently lower than the baseline method. Moreover, the test MSE  after a second iteration of the proposed \TRIM method is closer to the test MSE of the oracle method than the test MSE of the proposed method with a single iteration, when the lifetime is long enough. Interestingly, when the lifetime is short, the test MSE with a single iteration of the proposed method is lower than the test MSE with two iterations of \TRIM as well as the oracle method.

\begin{figure}[hbt!]
    \centering
    \begin{subfigure}[b]{0.48\textwidth}
        \includegraphics[width=\textwidth]{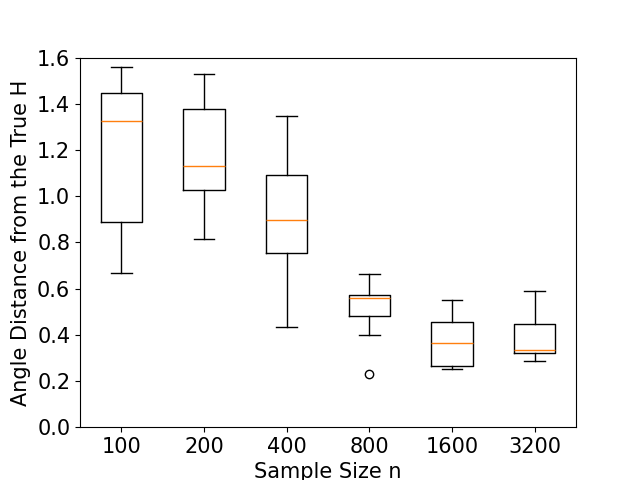}
        \caption{Liberian}
    \end{subfigure}
    \hfill %
    \begin{subfigure}[b]{0.48\textwidth}
        \includegraphics[width=\textwidth]{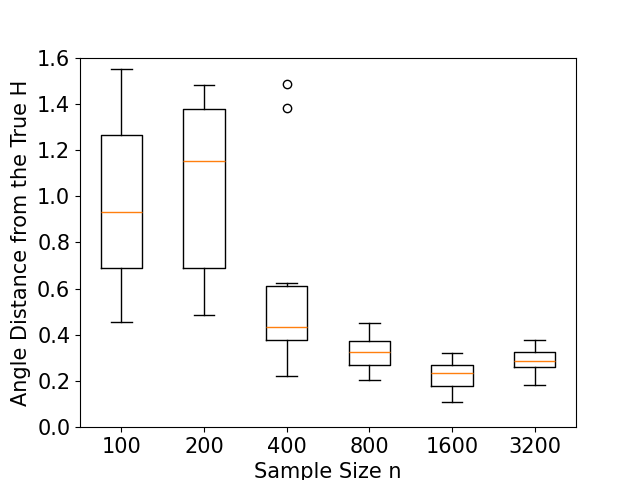}
        \caption{Liberian (reiterate)}
    \end{subfigure}
    \hfill %
    \begin{subfigure}[b]{0.48\textwidth}
        \includegraphics[width=\textwidth]{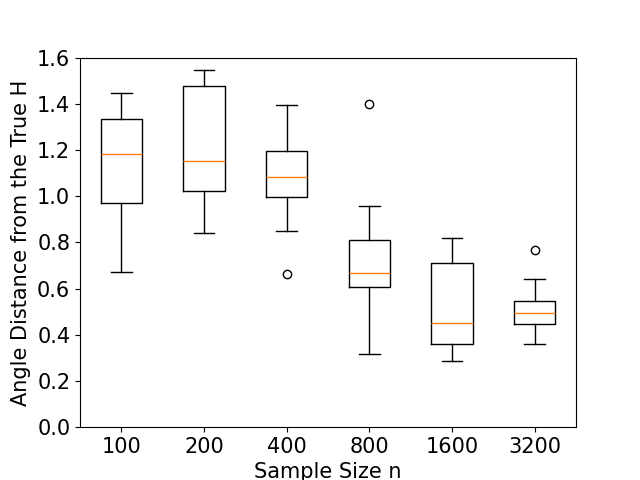}
        \caption{Sierra Leone}
    \end{subfigure}
    \hfill %
    \begin{subfigure}[b]{0.48\textwidth}
        \includegraphics[width=\textwidth]{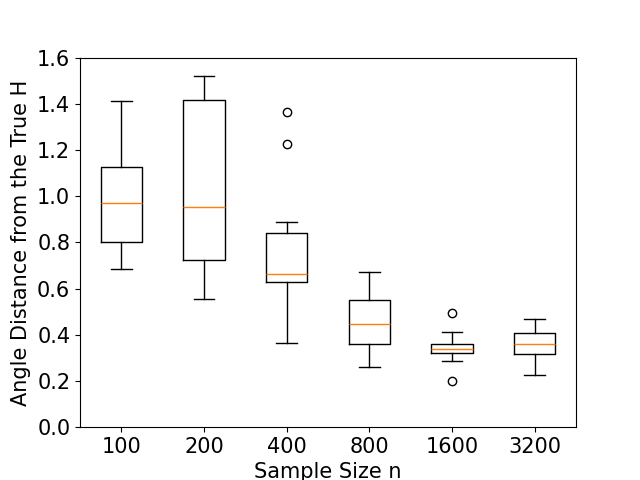}
        \caption{Sierra Leone (reiterate)}
    \end{subfigure}
    \caption{SEIR model for the spread of Ebola in Western Africa: Maximum principal angles between the subspaces spanned by the expected outer product matrix $H$ and estimated expected outer product matrix $\hat{H}_{n,t}$.}
    \label{fig:ebola dist_to_true_H}
\end{figure}

\begin{figure}[hbt!]
    \centering
    \begin{subfigure}[b]{0.48\textwidth}
        \includegraphics[width=\textwidth]{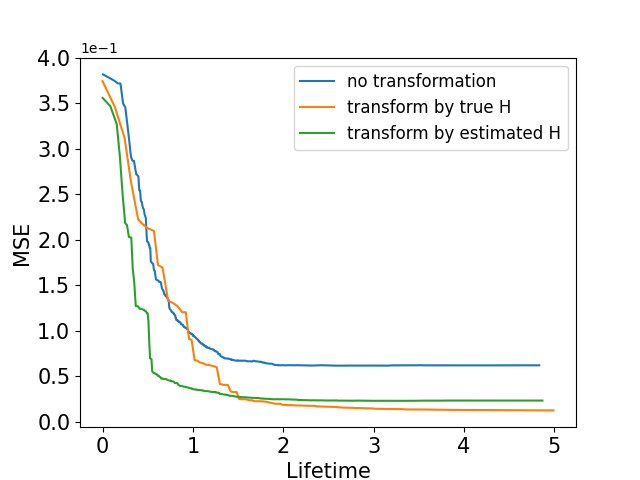}
        \caption{Liberian}
    \end{subfigure}
    \hfill %
    \begin{subfigure}[b]{0.48\textwidth}
        \includegraphics[width=\textwidth]{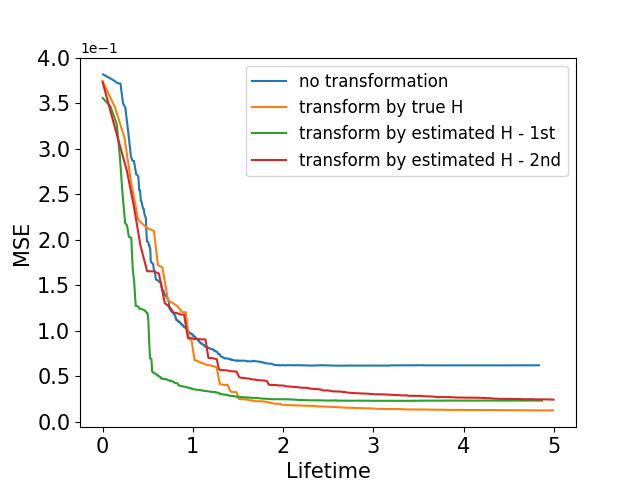}
        \caption{Liberian (reiterate)}
    \end{subfigure}
    \hfill %
    \begin{subfigure}[b]{0.48\textwidth}
        \includegraphics[width=\textwidth]{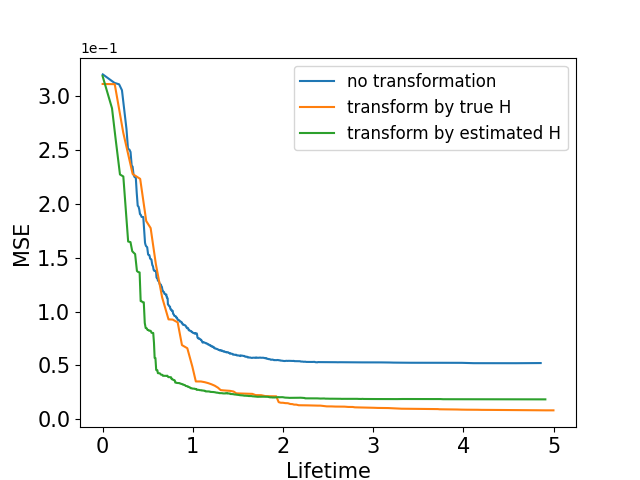}
        \caption{Sierra Leone}
    \end{subfigure}
    \hfill %
    \begin{subfigure}[b]{0.48\textwidth}
        \includegraphics[width=\textwidth]{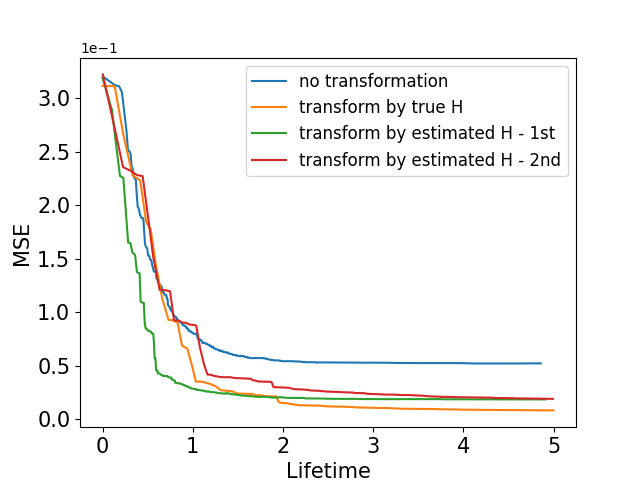}
        \caption{Sierra Leone (reiterate)}
    \end{subfigure}
    \caption{SEIR model for the spread of Ebola in Western Africa: Comparison of the test mean squared error (MSE) of the proposed method and the oracle method.}
    \label{fig:ebola test MSE comparison}
\end{figure}

\subsection{Real Data Experiments}
\label{sec:real_numerics}

Figure \ref{fig:real test MSE comparison 2} shows the comparison of the test mean squared error (MSE) of the proposed method and the oracle method for the real datasets. The proposed method performs well in this setting, as shown in Figure \ref{fig:real test MSE comparison} and Figure \ref{fig:real test MSE comparison 2}. The test MSE of \TRIM is consistently lower than the baseline Mondrian forest. Figure \ref{fig:real test accuracy comparison} shows the comparison of the test accuracy of different methods for the classification tasks.

\begin{figure}[hbt!]
    \begin{subfigure}[b]{0.48\textwidth}
        \includegraphics[width=\textwidth]{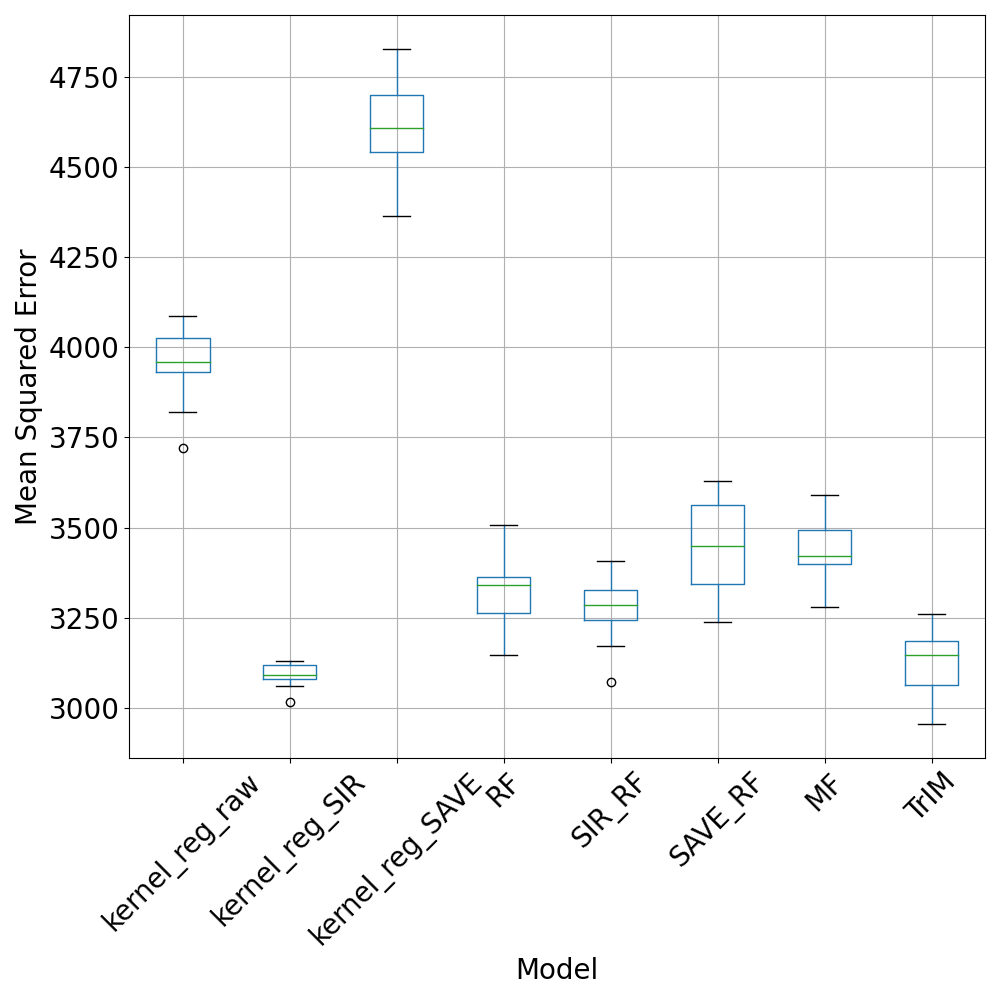}
        \caption{Diabetes}
    \end{subfigure}
    \hfill %
    \begin{subfigure}[b]{0.48\textwidth}
        \includegraphics[width=\textwidth]{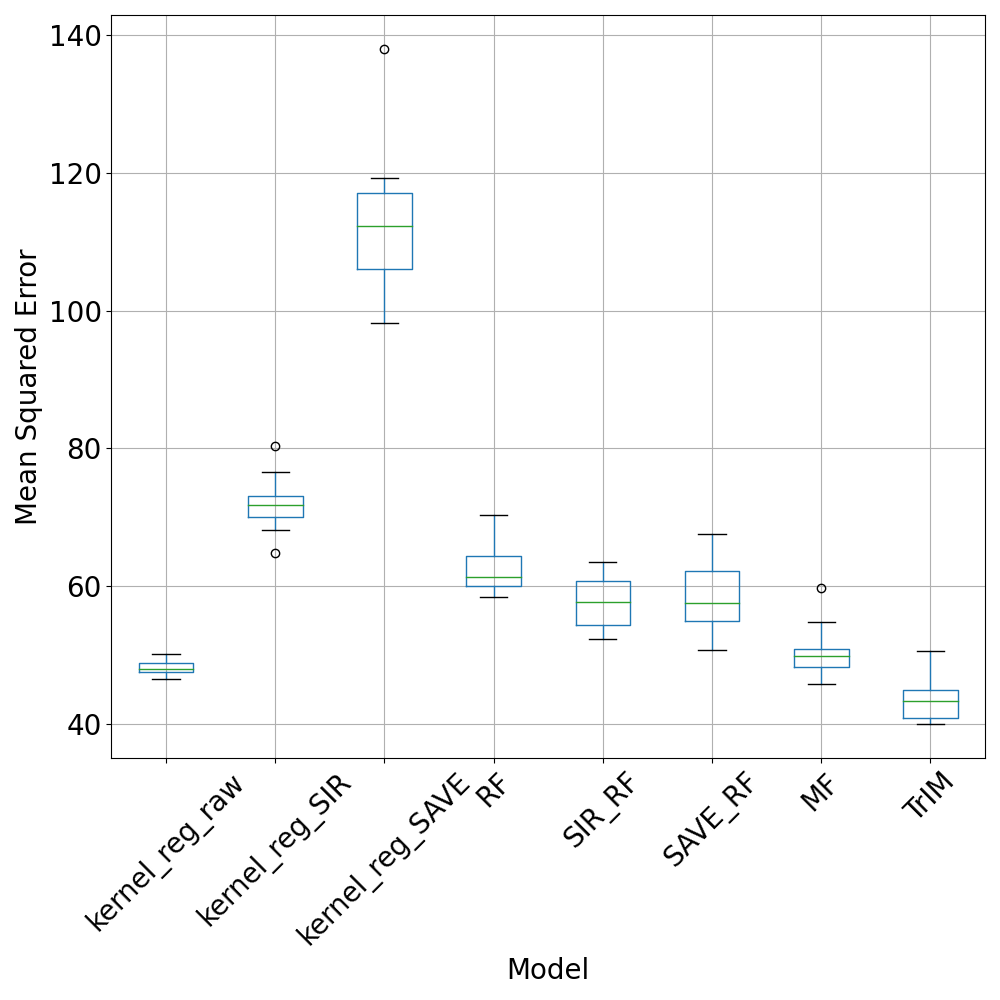}
        \caption{Mu284}
    \end{subfigure}
    \begin{subfigure}[b]{0.48\textwidth}
        \includegraphics[width=\textwidth]{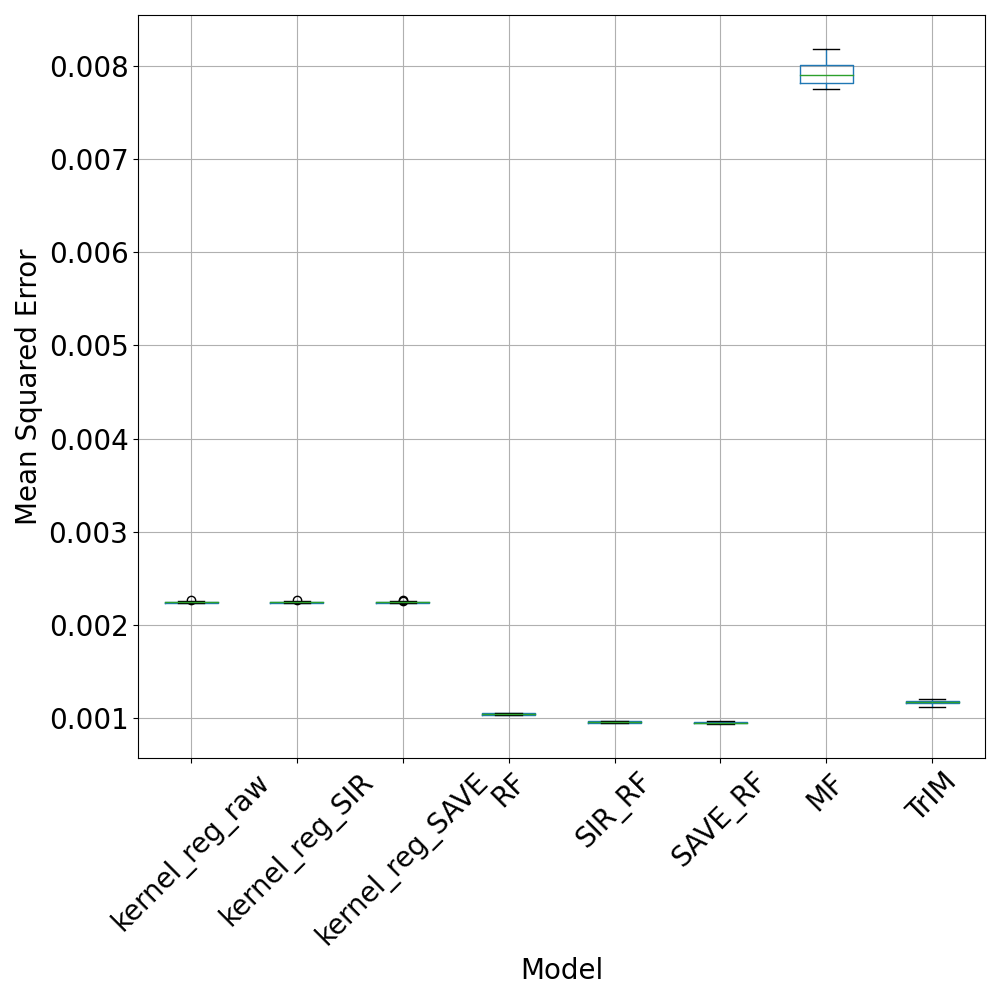}
        \caption{Bank8FM}
    \end{subfigure}
    \hfill %
    \begin{subfigure}[b]{0.48\textwidth}
        \includegraphics[width=\textwidth]{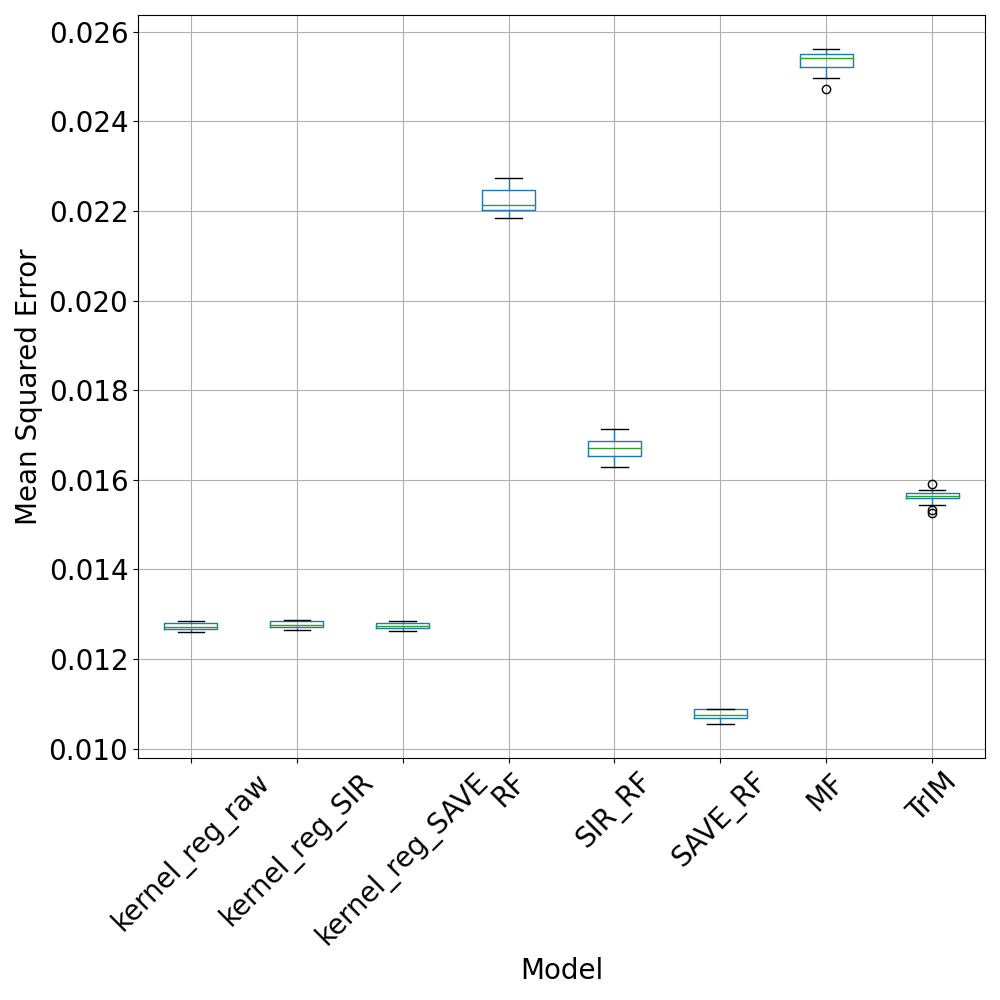}
        \caption{Kin8nm}
    \end{subfigure}
    \caption{Comparison of the test mean squared error (MSE) of different methods for real data applications, where the box plot displays the variation across 15 trials.}
    \label{fig:real test MSE comparison 2}
\end{figure}

\begin{figure}[hbt!]
    \begin{subfigure}[b]{0.48\textwidth}
        \includegraphics[width=\textwidth]{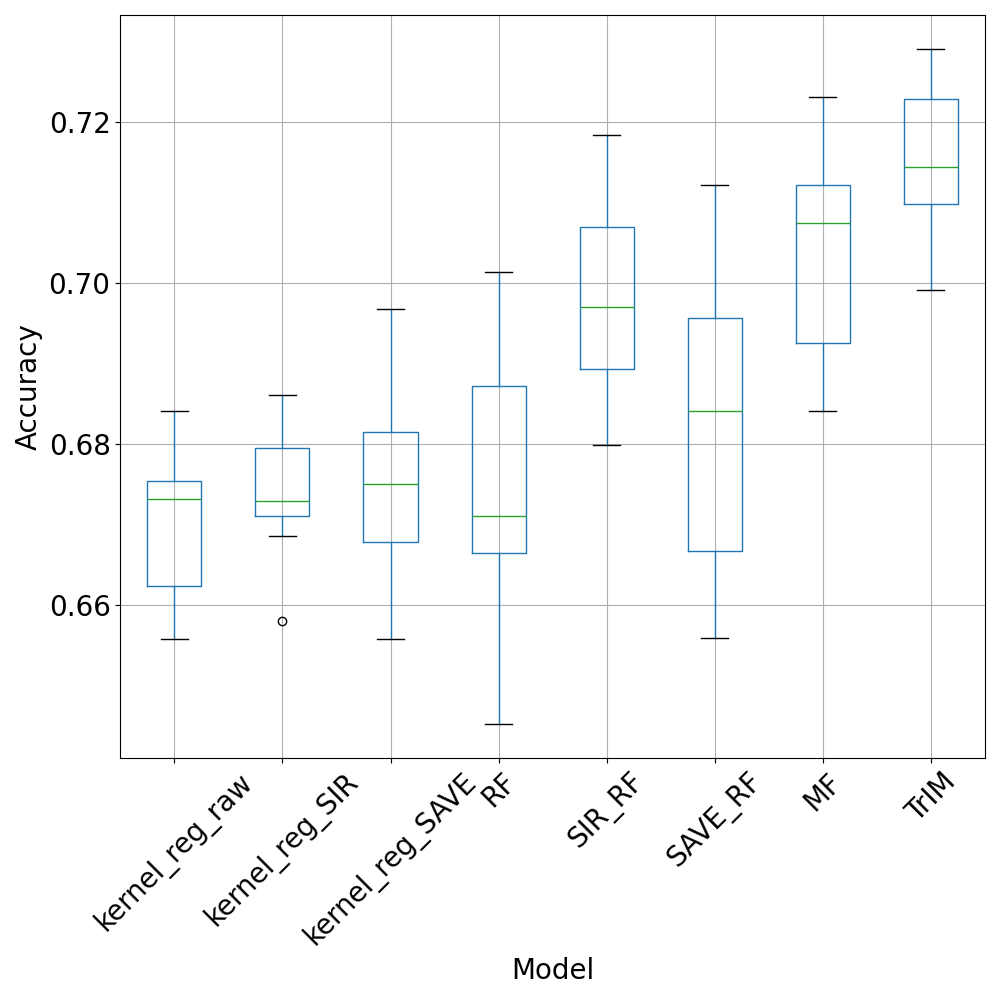}
        \caption{SA-heart}
    \end{subfigure}
    \hfill %
    \begin{subfigure}[b]{0.48\textwidth}
        \includegraphics[width=\textwidth]{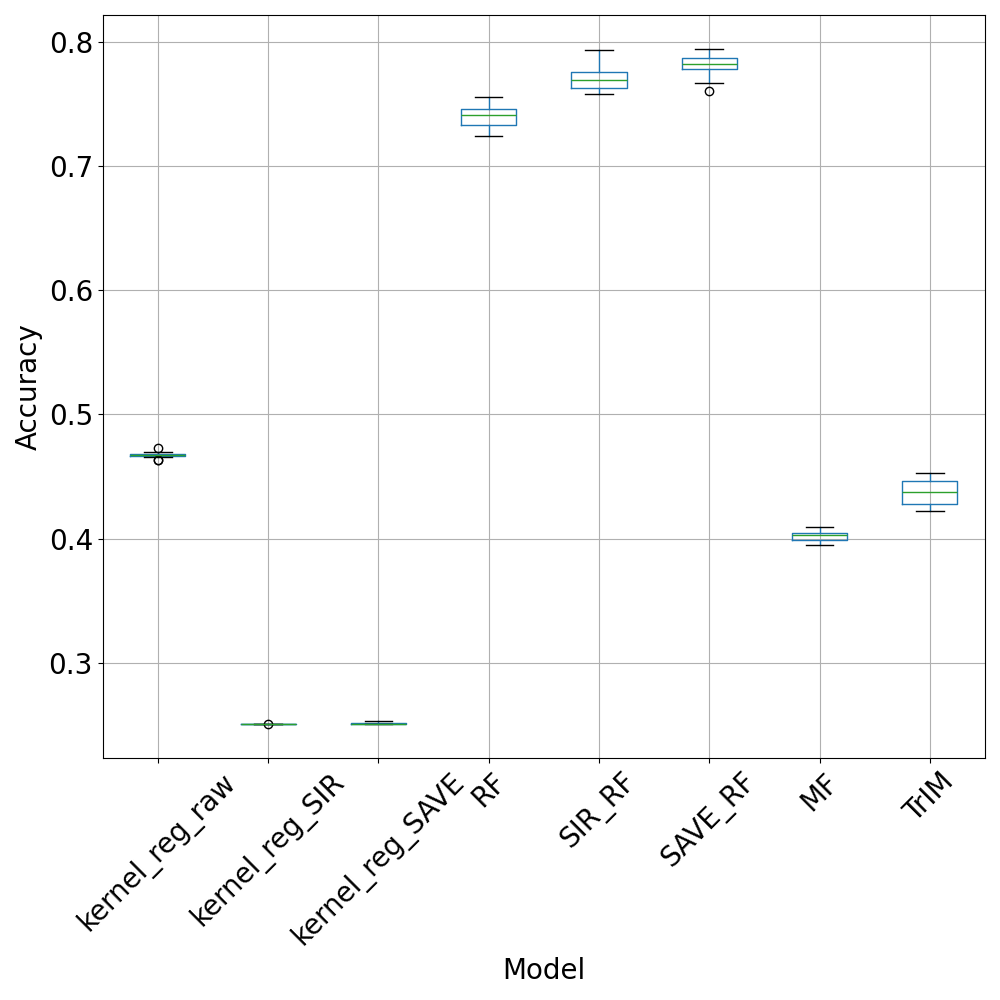}
        \caption{Vehicle}
    \end{subfigure}
    \begin{subfigure}[b]{0.48\textwidth}
        \includegraphics[width=\textwidth]{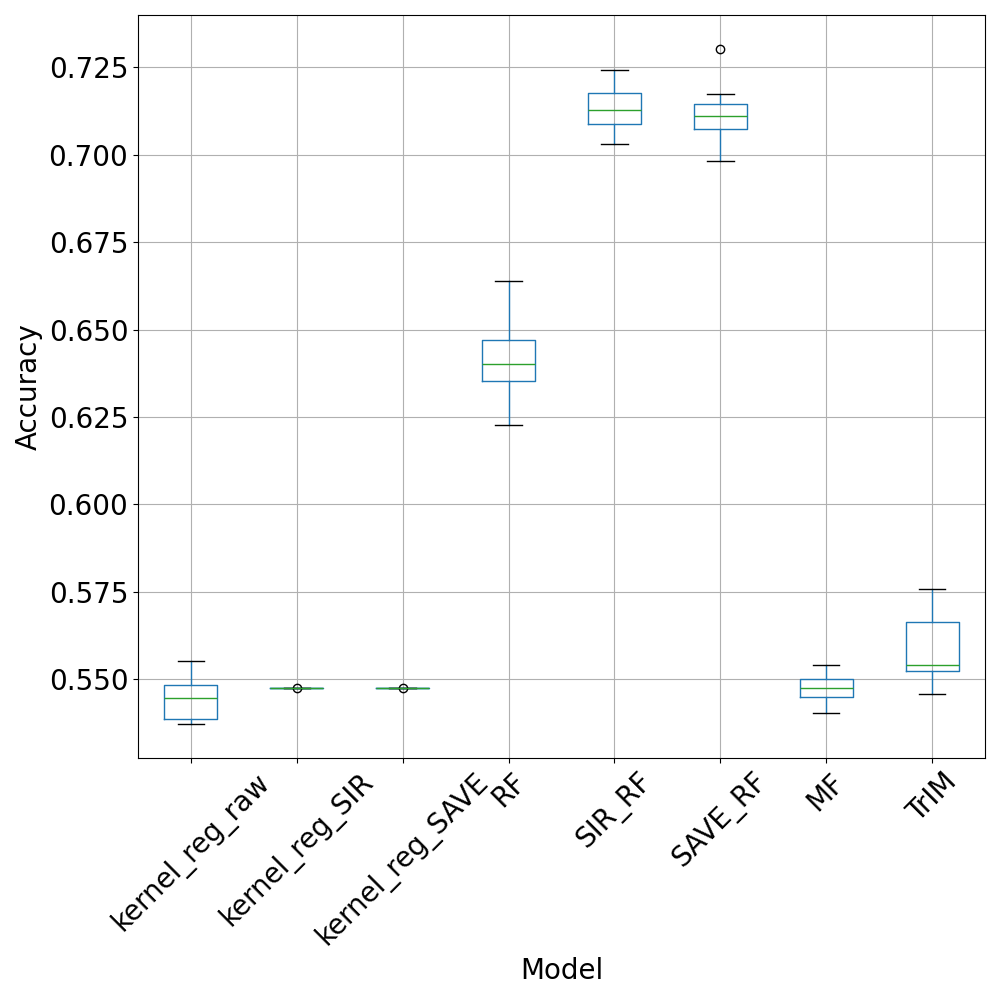}
        \caption{Forex}
    \end{subfigure}
    \hfill %
    \begin{subfigure}[b]{0.48\textwidth}
        \includegraphics[width=\textwidth]{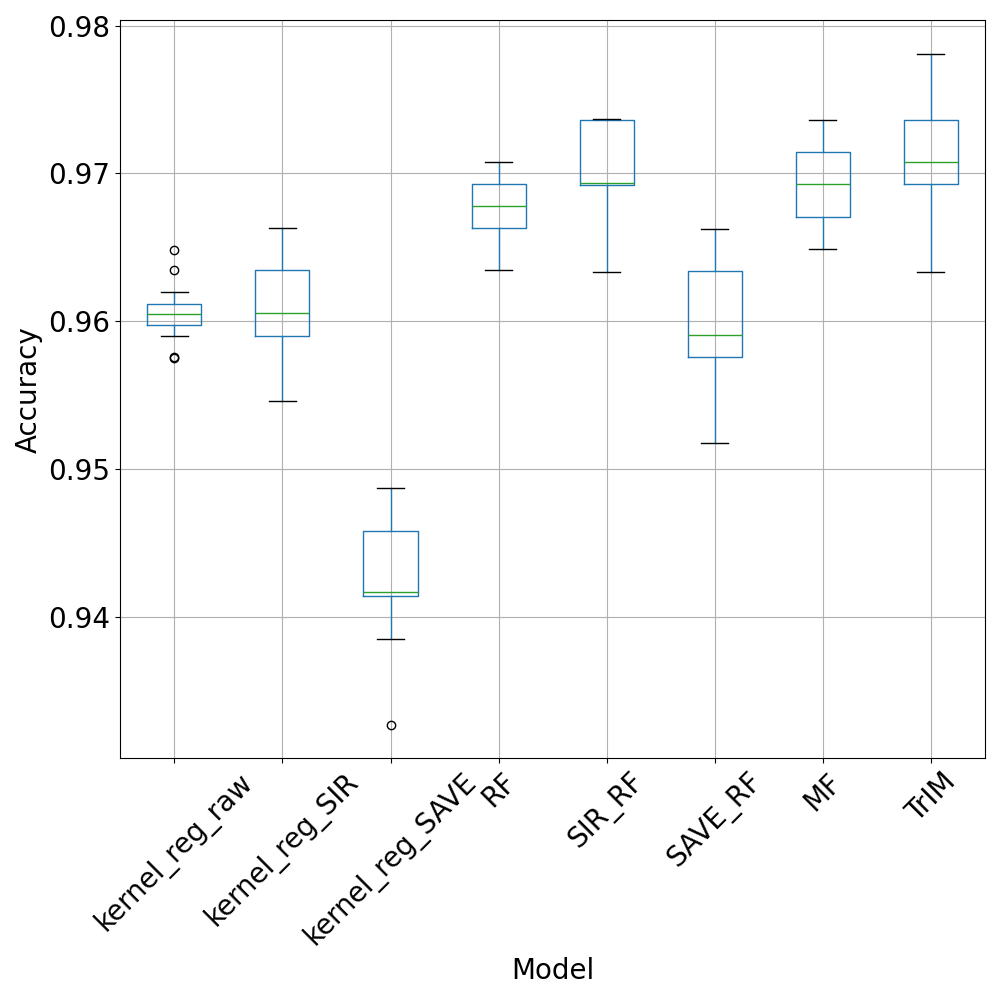}
        \caption{Breast}
    \end{subfigure}
    \caption{Comparison of the test accuracy of different methods for real data applications, where the box plot displays the variation across 15 trials.}
    \label{fig:real test accuracy comparison}
\end{figure}

\end{document}